\documentclass[11pt, oneside]{article}   	% use "amsart" instead of "article" for AMSLaTeX format
\usepackage{geometry}                		% See geometry.pdf to learn the layout options. There are lots.
\usepackage{graphicx}				% Use pdf, png, jpg, or eps§ with pdflatex; use eps in DVI mode
\usepackage{amssymb,amsmath,authblk}
\usepackage{color}

\usepackage[linecolor=white,backgroundcolor=white,bordercolor=white]{todonotes}
\usepackage[T1]{fontenc}
\usepackage[applemac]{inputenc}	
\usepackage{amsthm}
\usepackage{mathrsfs} 
\usepackage{caption}
\usepackage{subcaption}
\newtheorem{thm}{Theorem}[section]
\newtheorem{prop}[thm]{Proposition}
\newtheorem{coro}[thm]{Corollary}
\newtheorem{lem}[thm]{Lemma}

       \theoremstyle{definition}

       \theoremstyle{remark}

\newcommand{\PP}{\mathbb{P}}
\newcommand{\DD}{\mathbb{D}}
\newcommand{\E}{\mathbb{E}}
\newcommand{\R}{\mathbb{R}}
\newcommand{\C}{\mathbb{C}}
\renewcommand{\H}{\mathbb{H}}
\newcommand{\N}{\mathbb{N}}
\newcommand{\Z}{\mathbb{Z}}
\newcommand{\U}{\mathbb{U}}

\newcommand{\Sbb}{\mathbb{S}}

\newcommand{\Gbb}{\mathbb{G}}
\newcommand{\Vbb}{\mathbb{V}}

\newcommand{\Fbb}{\mathbb{F}}

\newcommand{\re}{\mathrm{Re}}
\newcommand{\im}{\mathrm{Im}}
\newcommand{\sym}{\mathrm{sym}}

\newcommand{\Tr}{\mathrm{Tr}}
\newcommand{\tr}{\mathrm{tr}}
\newcommand{\supp}{\mathrm{supp}}

\def\br{\begin{color}{red}}
\def\bb{\begin{color}{blue}}
\def\bg{\begin{color}{green}}
\def\er{\end{color}}
\def\eg{\end{color}}
\def\eb{\end{color}}
\def\a{\alpha}
\def\k{\kappa}
\def\th{\theta}
\def\t{\tau}
\def\b{\beta}
\def\d{\delta}
\def\ve{\varepsilon}
\def\o{\omega}
\def\O{\Omega}
\def\g{\gamma}
\def\G{\Gamma}
\def\l{\lambda}
\def\L{\Lambda}
\def\s{\sigma}
\def\q{\quad}
\def\pd{\partial}

\def\Log{\operatorname{Log}}

\def\spann{\operatorname{span}}
\def\area{\operatorname{area}}
\def\dist{\operatorname{dist}}
\def\supp{\operatorname{supp}}
\def\sgn{\operatorname{sgn}}
\def\sm{\setminus}
\def\sse{\subseteq}
\def\mfu{\mathfrak{u}}

\def\mfl{\mathfrak{l}}
\def\mfg{\mathfrak{g}}

\def\mfm{\mathfrak{m}}

\def\mfn{\mathfrak{n}}
\def\mfw{\mathfrak{l}}
\def\mfs{\mathfrak{s}}

\def\cA{\mathcal{A}}
\def\cB{\mathcal{B}}
\def\cE{\mathcal{E}}
\def\cV{\mathcal{V}}
\def\cD{\mathcal{D}}
\def\cI{\mathcal{I}}
\def\cL{\mathcal{L}}
\def\cO{\mathcal{O}}

\def\cM{\mathcal{M}}
\def\cS{\mathcal{S}}
\def\cT{\mathcal{T}}
\def\cG{\mathcal{G}}
\def\cC{\mathcal{C}}
\def\cF{\mathcal{F}}
\def\<{\langle}
\def\>{\rangle}

\title{Yang--Mills measure and the master field on the sphere}
\author{Antoine Dahlqvist \& James Norris\thanks{Research supported by EPSRC grant EP/I03372X/1}}
\affil{Statistical Laboratory, Centre for Mathematical Sciences, Wilberforce Road, Cambridge, CB3 0WB, UK} 
\begin{document}
\bibliographystyle{plain}
\maketitle
\abstract{We study the Yang--Mills measure on the sphere with unitary structure group.
In the limit where the structure group has high dimension, we show that the traces of loop holonomies converge in probability to a deterministic limit, 
which is known as the master field on the sphere. 
The values of the master field on simple loops are expressed in terms of the solution of a variational problem.
We show that, given its values on simple loops, the master field is characterized on all loops of finite length by a system of differential equations, known as the Makeenko--Migdal equations.
We obtain a number of further properties of the master field.
On specializing to families of simple loops, our results identify the high-dimensional limit, in non-commutative distribution, of the Brownian loop in the group of unitary matrices.}

\tableofcontents

\section{Introduction}
The Yang--Mills measure, associated to a (two-dimensional) surface $\Sigma$ and to a compact Lie group $G$, is a probability measure on (generalized) connections of principal $G$-bundles over $\Sigma$.
It was introduced in a series of works by Gross, King \& Sengupta \cite{MR1015789}, Fine \cite{MR1124272}, Driver \cite{MR1006295},  Witten \cite{MR1133264,MR1185834}, Sengupta \cite{MR1346931} and L\'evy \cite{MR2006374}, 
as a mathematical version of Euclidean Yang--Mills field theory.
In this paper, we will consider the Yang--Mills measure in the case where the surface $\Sigma$ is fixed and the group $G$ is a classical matrix group of high dimension.
The interest of such a set-up from the viewpoint of random matrix theory was first raised in the mathematics literature by Singer \cite{MR1373007}, 
who made several conjectures, based on earlier work in physics \cite{MR1401299, MR1352420, MR605753, MR596907}.
The high-dimensional limit of the Yang--Mills measure when $\Sigma$ is the whole plane has since been studied by Xu \cite{MR1489573}, Sengupta \cite{MR2757706}, L\'evy \cite{MR2667871}, Anshelevich \& Sengupta \cite{MR2864481}, 
Dahlqvist \cite{MR3554890} and others.
As we shall see, the general problem is closely related to another, addressed by Biane \cite{MR1426833}, L\'evy \cite{MR2407946}, L\'evy \& Ma\"ida \cite{MR2727643} and Collins, Dahlqvist \& Kemp \cite{1502.06186},
which is to understand the high-dimensional limit of the Brownian loop measure on the group as a non-commutative process.

We focus here on the case where the surface $\Sigma$ is a sphere.  
This has received particular attention in the physics literature \cite{MR1262293, MR1297298, MR1321333, Rusakov}, as it displays a phase transition of third order named after Douglas and Kazakov \cite{DOUGLAS1993219}.
A corresponding mathematical analysis of the partition function was achieved by Boutet de Monvel \& Shcherbina \cite{MR1692402} and L\'evy \& Ma\"ida \cite{MR3440793}.
The main result of the present work, Theorem \ref{main}, confirms a conjecture of Singer \cite{MR1373007}, 
showing that, under the Yang--Mills measure on the sphere for the unitary group $U(N)$, the traces of loop holonomies converge as $N\to\infty$ to a deterministic limit.
We characterize this limit analytically and derive some further properties.
Following the physics literature, the limit is called the master field on the sphere.  
As a by-product of our main result, we show that the Brownian loop in $U(N)$ converges in non-commutative distribution as $N\to\infty$ to a certain non-commutative process, which we call the free unitary loop.  

There is a system of relations, discovered by Makeenko and Migdal \cite{MakeenkoMigdal}, indexed by families of embedded loops,
between the expectations under the Yang--Mills measure of polynomials in the traces of loop holonomies.
These have now been proved for the whole plane by L\'evy \cite{MR2667871} and Dahlqvist \cite{MR3554890} and for any compact surface by Driver, Gabriel, Hall \& Kemp \cite{MR3631396}.
The Makeenko--Migdal equations provide a potential line of argument to prove convergence of the Yang--Mills measure as $N\to\infty$, 
which is to show a suitable concentration estimate for the holonomy traces, and to pass to the limit in the equations, showing that the limit equations determine a unique limit object.
In the whole plane case, moment estimates for unitary Brownian motion provide the needed concentration, and 
the Makeenko--Migdal equations may be augmented by a further equation, such that the whole system of equations then characterizes the limit field.
So the programme has been completed in that case \cite{MR3554890,MR2667871}.
However, as noted in \cite{MR3631396}, the concentration and characterization problems have remained open in general.
   
In this paper, we will establish two key points.
First, for simple loops, we show in Proposition \ref{Representation simple loops beta ensemble} that expectations and covariances of the holonomy traces can be represented by functional of a discrete $\beta$-ensemble.
This representation allows to identify the limit in probability of these traces as $N\to\infty$, following the work of 
Guionnet and Ma{\"{\i}}da \cite{MR2198201}, Johansson \cite{MR1737991} and F\'eral \cite{MR2483725} on discrete $\beta$-ensembles. 
This amounts to a rigorous version of ideas explained by Boulatov \cite{MR1262293} and Douglas \& Kazakov \cite{DOUGLAS1993219}.  
The second point, shown in Section \ref{MM} using the Makeenko--Migdal equations, is  
that the convergence of marginals to a deterministic limit for simple loops forces the same to hold for a more generic class of loops\footnote{This point has recently been shown independently also by Brian Hall \cite{1705.07808}.}.

An alternative line of argument for the first point, which we shall discuss elsewhere, 
would be to use the fact that the process of eigenvalues of the marginals of the Brownian loop is known to have the same law as a Dyson Brownian motion on the circle, starting from $1$ and conditioned to return to $1$.
Indeed, several scaling limits of this conditioned process have recently been understood by Liechty \& Wang \cite{MR3474469}.  
This link was first observed in the physics literature in Forrester, Majumdar \& Schehr \cite{MR2747559,MR2874214}.  
Section \ref{HADBE} gives another way to obtain macroscopic results on the empirical distribution of this process.       

The paper is organized as follows. 
Section \ref{Model Results} introduces the model and our results.  
Section \ref{HADBE} shows convergence and concentration of holonomy traces for simple loops, using a duality relation with a discrete $\beta$-ensemble.
Section \ref{MM} explains how  the  Makeenko--Migdal equations can be used to extend this convergence to a general class of regular loops.
Then, in Section \ref{RL}, we make a final extension to all loops of finite length.
Section \ref{Prop MF} presents some further properties of the master field, including a relation with the free Hermitian Brownian loop in the subcritical regime, and
a formula for the evaluation of the master field on a large class of loops. 

Subject to certain modifications, to be explained in a future work, the argument explained here applies to other series of compact groups and also with the projective plane in place of the sphere.

\section{Setting and statement of the main results \label{Model Results}}
We review the notion of a Yang--Mills holonomy field over a compact Riemann surface.
Then we discuss its relation, in the case of the sphere, to the Brownian loop in a Lie group.
Next, we state our main results on convergence of Yang--Mills holonomy in $U(N)$ over the sphere to the master field,
and on analytic characterization of the master field.
The proof of these main results has three steps, which are outlined in Section \ref{OUT}.
Then we discuss some consequences of our results, for the convergence of spectral measures of loop holonomies, and for the high-dimensional limit of the Brownian loop in $U(N)$.
Finally, we discuss how the master field can be considered as a natural family of infinite-dimensional unitary transport operators, following up some suggestions of Singer \cite{MR1373007}.
  
\subsection{Yang--Mills measure on a compact Riemann surface}\label{YMMC}
We recall in this subsection the approach of L\'evy \cite{MR2006374} to the Yang--Mills measure. 
Let $\Sigma$ be a compact Riemann surface and let $G$ be a compact Lie group.
Write $T$ for the area of $\Sigma$ and denote by $1$ the unit element of $G$.
Fix a bi-invariant Riemannian metric on $G$ and denote the associated heat kernel by $p=(p_t(g):t\in(0,\infty),g\in G)$.
Thus $p$ is the unique smooth positive function on $(0,\infty)\times G$ such that
$$
\frac{\pd p}{\pd t}=\frac12\Delta p
$$
and, for all continuous functions $f$ on $G$, in the limit $t\to0$,
$$
\int_Gf(g)p_t(g)dg\to f(1).
$$
Here we have written $\Delta$ for the Laplace--Beltrami operator and $dg$ for the normalized Haar measure on $G$.

We specialize in later sections to the case where $\Sigma$ is the sphere $\Sbb_T$ of area $T$, and where $G$ is the  group $U(N)$ of unitary $N\times N$ matrices.
The Lie algebra of $U(N)$ is the space of skew-Hermitian matrices $\mfu(N)$.
We specify a metric on $U(N)$ by the following choice of inner product on $\mfu(N)$
\begin{equation}\label{metric}
\<g_1,g_2\>=N\Tr(g_1g_2^*)
\end{equation}
where $\Tr(g)=\sum_{i=1}^Ng_{ii}$.
This dependence of the metric on $N$, which is standard in random matrix theory, is chosen so that the objects of interest to us have a non-trivial scaling limit as $N\to\infty$.

Write $P(\Sigma)$ for the set of oriented paths of finite length in $\Sigma$, considered modulo reparametrization. 
Denote the length of a path $\g\in P(\Sigma)$ by $\ell(\g)$. 
We consider $P(\Sigma)$ as a metric space, with the {\em length metric}
\begin{equation}\label{metricP}
d(\g,\g')=|\ell(\g)-\ell(\g')|+\inf_{\t,\t'}\sup_{t\in[0,1]}d(\g_{\t(t)},\g'_{\t'(t)})
\end{equation}
where the infimum is taken over reparametrizations $\t,\t'$ of $\g,\g'$ by $[0,1]$.
Each path $\g$ has a starting point $\underline\g$ and a terminal point $\overline\g$.
Write $\g^{-1}$ for the reversal of $\g$, that is, the path of reverse orientation from $\overline\g$ to $\underline\g$.
For paths $\g_1,\g_2$ such that $\overline\g_1=\underline\g_2$, we write $\g_1\g_2$ for the path obtained by their concatenation.
Write $L(\Sigma)$ for the set of loops of finite length in $\Sigma$. 
Thus
$$
L(\Sigma)=\{\g\in P(\Sigma):\underline\g=\overline\g\}.
$$
Write also $P_{x,y}(\Sbb_T)$ for the set of paths from $x$ to $y$, and $L_x(\Sbb_T)$ for the set of loops based at $x$.
Given paths $\g,\g_0$, we say that $\g_0$ is a {\em simple reduction} of $\g$ if we can write $\g$ and $\g_0$ as concatenations
$$
\g=\g_1\g_*\g_*^{-1}\g_2,\q \g_0=\g_1\g_2
$$
for some paths $\g_1,\g_2,\g_*$. 
More generally, we say that $\g_0$ is a {\em reduction} of $\g$ if there is a sequence of paths $(\g_1,\dots,\g_n)$ such that $\g_{i-1}$ is a simple reduction of $\g_i$ for all $i$ and $\g_n=\g$.
Given paths $\g_1,\g_2$, we write $\g_1\sim\g_2$ if there is a path $\g_0$ which is a reduction of both $\g_1$ and $\g_2$.

Given a subset $\G$ of $P(\Sigma)$ which is closed under reversal and concatenation, we call a function $h:\G\to G$ {\em multiplicative} if
$$
h_{\g^{-1}}=h^{-1}_\g,\q h_{\g_1\g_2}=h_{\g_2}h_{\g_1}
$$
for all $\g$ and for all $\g_1,\g_2$ with $\overline{\g}_1=\underline{\g}_2$.
We denote the set of such multiplicative functions by $\cM(\G,G)$.
Note that, for any such function $h$, we have $h_{\g_1}=h_{\g_2}$ whenever $\g_1\sim\g_2$.

We say that a finite subset $\Gbb=\{e_1,\dots,e_m\}\sse P(\Sigma)$ is an {\em embedded graph} in $\Sigma$ if each path $e_j$ is non-constant, is either simple or a simple loop, and meets other paths $e_k$ only at its endpoints.
Then we refer to the sequence $(e_1,\dots,e_m)$ as a {\em labelled embedded graph}.
We will sometimes write abusively $\Gbb=(V,E,F)$ to mean that $V$ is the set of endpoints of paths in $\Gbb$, $E=\Gbb$ and $F$ is the set of connected components of $\Sigma\sm\{e^*:e\in\Gbb\}$.
Here $e^*$ denotes the range of $e$.
We say that an embedded graph $\Gbb$ is a {\em discretization} of $\Sigma$ if each face $f\in F$ is a simply connected domain in $\Sigma$.
Write $P(\Gbb)$ for the subset of $P(\Sigma)$ obtained by concatenations of the paths in $\Gbb$ and their reversals.   

A random process $H=(H_\g:\g\in P(\Sigma))$ (on some probability space $(\O,\cF,\PP)$) taking values in $G$ is a {\em Yang--Mills holonomy field} if
\begin{itemize}
\item[(a)] $H$ is multiplicative, that is, $H(\o)\in\cM(P(\Sigma),G)$ for all $\o\in\O$, 
\item[(b)] for any discretization $\Gbb=(V,E,F)$ of $\Sigma$ and all $h\in\cM(P(\Gbb),G)$,
\begin{equation}
\label{Discrete Yang Mills}
\PP(H_e\in dh_e\text{ for all }e\in E)=p_T(1)^{-1}\prod_{f\in F}p_{|f|}(h_f)\prod_{e\in E}dh_e
\end{equation}
\item[(c)] for any convergent sequence $\g(n)\to\g$ in $P(\Sigma)$ with fixed endpoints,
$$
H_{\g(n)}\to H_\g\q\text{in probability.}
$$
\end{itemize}
Here, for each face $f$, we have written $|f|$ for the area of $f$ and we have chosen a simple loop $\g(f)\in L(\Gbb)$ whose range is the boundary of $f$ and set $h_f=h_{\g(f)}$.
The invariance properties of Haar measure and the heat kernel under inversion and conjugation guarantee that the expression \eqref{Discrete Yang Mills}
for the finite-dimensional distributions of $H$ does not depend on the orientations of the edges, nor on the choice of loops bounding the faces.

Define coordinate functions $H_\g:\cM(P(\Sigma),G)\to G$ by $H_\g(h)=h_\g$ and define a $\s$-algebra $\cC$ on $\cM(P(\Sigma),G)$ by
$$
\cC=\s(H_\g:\g\in P(\Sigma)).
$$
Then $(H_\g:\g\in P(\Sigma))$ is a multiplicative random process on $(\cM(P(\Sigma),G),\cC)$.
We use the same notation $(H_\g:\g\in P(\Sigma))$ both for this canonical coordinate process and also, more generally, for any multiplicative random process.

Our basic object of study is the {\em Yang--Mills measure} provided by the following theorem of L\'evy \cite[Theorem 2.62]{MR2006374},
building on earlier work of Driver \cite{MR1006295} and Sengupta \cite{MR1346931}.

\begin{thm}\label{Thm Exist YM}
There is a unique probability measure on $(\cM(P(\Sigma),G),\cC)$ under which the coordinate process $(H_\g:\g\in P(\Sigma))$ is a
Yang--Mills holonomy field.
\end{thm}

Let $H=(H_\g:\g\in P(\Sigma))$ be a Yang--Mills holonomy field in $G$.  
We note the following properties of {\em gauge invariance} and {\em invariance under area-preserving diffeomorphisms}, 
which follow from invariance properties of \eqref{Discrete Yang Mills} and the uniqueness statement of the theorem.
Let $s:\Sigma\to G$ be a measurable function and let $\psi:\Sigma\to\Sigma$ be an area-preserving diffeomorphism.
Consider the processes 
$$
H^s=(s(\overline\g)H_\g s(\underline\g)^{-1}:\g\in P(\Sigma)),\q H^\psi=(H_{\psi\circ\g}:\g\in P(\Sigma)).
$$
Then $H^s$ and $H^\psi$ have the same law as $H$.
In particular, the relevant data from $\Sigma$ are just its genus and the total area $T$.

\subsection{Embedded Brownian loops\label{section embedded bridge}}
We specialize now to the case where the surface $\Sigma$ is the sphere $\Sbb_T$ of area $T$.
In each Yang--Mills holonomy field $H=(H_\g:\g\in P(\Sbb_T))$, there are many embedded Brownian loops in $G$ based at $1$ and parametrized by $[0,T]$, as we now show.
Recall that a random process $B=(B_t:t\in[0,T])$ taking values in $G$ is a Brownian loop based at $1$ if
\begin{itemize}
\item[(a)] $B$ is continuous, that is, $B(\o)\in C([0,T],G)$ for all $\o\in\O$,
\item[(b)] for all $n\in\N$, all $g_1,\dots,g_{n-1}\in G$ and all increasing sequences $(t_1,\dots,t_{n-1})$ in $(0,T)$,
setting $g_0=g_n=1$ and $t_0=0$ and $t_n=T$ and writing $t_k=s_1+\dots+s_k$, 
$$
\PP(B_{t_k}\in dg_k\text{ for }k=1,\dots,n-1)=\frac{\prod_{i=1}^np_{s_i}(g_ig_{i-1}^{-1})}{p_T(1)}\prod_{k=1}^{n-1}dg_k.
$$
\end{itemize}
Choose a point $x$ in $\Sbb_T$ and let $P$ be a tangent plane to $\Sbb_T$ at $x$, considered in its usual embedding in $\R^3$.
Choose a line $L$ in $P$ through $x$ and rotate $\Sbb_T$ once around $L$.
The intersections of $P$ with $\Sbb_T$, which are a nested family of circles, may be given a consistent orientation and then 
considered as a family in $L(\Sbb_T)$, all starting from $x$.
We can parametrize this family of loops as $(l(t):t\in[0,T])$ so that the domain inside $l(t)$ has area $t$ for all $T$.
Then, for all $n\in\N$ and all sequences $(t_1,\dots,t_{n-1})$ in $(0,T)$, the loops $l(t_1),\dots,l(t_{n-1})$ are the edges of a discretization of $\Sbb_T$.
Define a random process $\b=(\b_t:t\in[0,T])$ in $G$ by
$$
\b_t=H_{l(t)}.
$$
It is straightforward to deduce from property (b) of the Yang--Mills holonomy field that the finite-dimensional distributions of $\b$ satisfy condition (b) for the Brownian loop.
Hence, by standard arguments, $\b$ has a continuous version, $B$ say, which is a Brownian loop in $G$ based at $1$.
The reader will see many ways to vary this construction while still obtaining a Brownian loop.
In each case, we obtain from a nested loop of loops of finite length in $\Sbb_T$ a Brownian loop in $G$, which of course does not have finite length.

\subsection{Convergence to the master field on the sphere}\label{MFSS}
We specialize from now on to the case where the structure group $G$ is the group of $N\times N$ unitary matrices $U(N)$.
Let $H=(H_\g:\g\in P(\Sbb_T))$ be a Yang--Mills holonomy field in $U(N)$ over the sphere $\Sbb_T$ of area $T$.
We will write $H$ rather than $H^N$ throughout, to lighten the notation.
Our main results establish a law of large numbers for this random field in the limit $N\to\infty$, which we express for now in terms of the
normalized trace
$$
\tr(g)=N^{-1}\sum_{i=1}^Ng_{ii}.
$$
The limit object is a certain function on loops
$$
\Phi_T:L(\Sbb_T)\to\C
$$
known in the physics literature as the {\em master field on the sphere}.
It will be convenient to define $\Phi_T$ by
$$
\Phi_T(l)=
\begin{cases}\lim_{N\to\infty}\E(\tr(H_l)),&\text{if this limit exists},\\0,&\text{otherwise}.\end{cases}
$$
Our first main result establishes concentration.
\begin{thm} 
\label{main}
For all loops $l\in L(\Sbb_T)$,
$$
\tr(H_l)\to\Phi_T(l)\q\text{in probability as $N\to\infty$}.
$$
\end{thm}

Since $|\tr(H_l)|\le1$, this implies in particular that the limit considered in the definition of the master field always exists:
$$
\Phi_T(l)=\lim_{N\to\infty}\E(\tr(H_l)).
$$
The master field then inherits certain properties from its finite-dimensional approximations $\E(\tr(H_l))$,
as the reader may easily check.

\begin{prop}\label{state master field} 
The master field $\Phi_T$ has the following properties: 
\begin{enumerate}
\item[{\rm (a)}]$\Phi_T=1$ on constant loops and $\Phi_T(l)=\Phi_T(l^{-1})\in[-1,1]$ for all loops $l$,  
\item[{\rm (b)}]$\Phi_T(\g_1\g_2)=\Phi_T(\g_2\g_1)$ for all pairs of paths $\g_1,\g_2$ such that $\g_1\g_2$ is a loop,
\item[{\rm (c)}]$\Phi_T(l_1)=\Phi_T(l_2)$ whenever $l_1\sim l_2$,
\item[{\rm (d)}]for all $x,y\in\Sbb_T$, all $n\in\N$, all $a_1,\dots,a_n\in\C$ and all $\g_1,\dots,\g_n\in P_{x,y}(\Sbb_T)$,
$$
\sum_{i,j=1}^n a_i\overline{a_j}\Phi_T(\g_i\g_j^{-1})\ge0
$$ 
\item[{\rm (e)}]for all loops $l$ and any area-preserving diffeomorphism $\psi$ of $\Sbb_T$,
$$
\Phi_T(\psi(l))=\Phi_T(l).
$$ 
\end{enumerate}
\end{prop}

\subsection{Characterization of the master field on the sphere}\label{MFSS2}
Our second main result is an analytic characterization of the master field. 
This will require some associated notions which we now introduce.
Consider the following variational problem: minimize the functional 
\begin{equation}\label{SCE}
\cI_T(\mu)=\int_{\R^2}\left\{\tfrac12(x^2+y^2)T-2\log|x-y|\right\}\mu(dx)\mu(dy)
\end{equation}
over the set of probability measures $\mu$ on $\R$ such that 
$$
\mu([a,b])\le b-a
$$ 
whenever $a\le b$.
We note for later use some statements concerning this problem, proofs of which may be found in L\'evy and Ma\"ida \cite{MR3440793}.
First, the functional $\cI_T$ is well-defined on the given set, with values in $(-\infty,\infty]$, and has a unique minimizer, which we denote by $\mu_T$.
Then $\mu_T$ has a continuous density function $\rho_T$ with respect to Lebesgue measure, with $0\le\rho_T(x)\le1$ for all $x$.
In the case $T\in(0,\pi^2]$, $\rho_T$ is the semi-circle density of variance $1/T$, given by
\begin{equation}\label{SCD}
\rho_T(x)=\frac{T}{2\pi}\sqrt{\frac4T-x^2},\q |x|\le2/\sqrt T.
\end{equation}
Note that the right-hand side in \eqref{SCD} exceeds $1$ when $x=0$ for $T>\pi^2$.

For $T\in(\pi^2,\infty)$, there is a unique $k\in(0,1)$ such that
$$
T=8EK-4(1-k^2)K^2
$$
where $K=K(k)$ and $E=E(k)$ are, respectively, the complete elliptic integrals of the first and second kind.
Set $\a=4kK/T$ and $\b=4K/T$.
Then the minimizing density $\rho_T$ is identically $1$ on $[-\a,\a]$, is supported on $[-\b,\b]$, and satisfies, for $|x|\in(\a,\b)$,
\begin{equation}
\rho_T(x)=\frac{2\sqrt{(x^2-\a^2)(\b^2-x^2)}}{\pi\b|x|}\int_0^1\frac{ds}{(1-\a^2s^2/x^2)\sqrt{(1-s^2)(1-\a^2s^2/\b^2)}}.
\label{DCG}
\end{equation}
See \cite[Lemma 4.7, equation (4.14)]{MR3440793}.
See also \cite[Figure 7]{MR3440793} for an informative plot of the family of densities $(\rho_T:T\in(0,\infty))$. 

Let us say that $l\in L(\Sbb_T)$ is a {\em regular loop} if there is a labelled embedded graph $\Gbb_l=(e_1,\dots,e_m)$ in $P(\Sbb_T)$ such that $l$ is given by the concatentation $e_1\dots e_m$, 
in which $\underline e_1$ has degree $2$ and in which $\underline e_2,\dots,\underline e_m$ have degree $4$ and are transverse self-intersections of $l$.
Here, we say that a self-intersection of $l$ at a vertex $v$ of degree $4$ is {\em transverse} if, as $l$ passes through $v$, it arrives and leaves by opposite edges.
Note that $\Gbb_l$ is then uniquely determined by $l$.

Given a regular loop $l$, and a point $v$ of self-intersection of $l$, there are two regular loops $l_v$ and $\hat l_v$ starting from $v$, 
obtained by {\em splitting $l$ at $v$}, that is, by following $l$ on its first and second exit from $v$, respectively, until it first returns to $v$.
Note that both $l_v$ and $\hat l_v$ have fewer self-intersections than $l$.
For each face $f$ of $\Gbb$, define
$$
\sgn_v(f)=
\begin{cases}
0,&\text{if $v$ is not a boundary vertex of $f$},\\
1,&\text{if $f$ is adjacent to both outgoing or both incoming edges at $v$},\\
-1,&\text{if $f$ is adjacent to one outgoing and one incoming edge at $v$}.
\end{cases}
$$
For $\eta>0$, we say that a smooth map 
$$
\th:[0,\eta)\times\Sbb_T\to\Sbb_T
$$
is a {\em Makeenko--Migdal flow} at $(l,v)$ if 
\begin{enumerate}
\item[(a)] $\th(0,x)=x$ for all $x$,
\item[(b)] $\th(t,.)$ is a diffeomorphism of $\Sbb_T$ for all $t$,
\item[(c)] for any face $f$ of the embedded graph $\Gbb$,   
$$
\frac{d}{dt}|\theta(t,f)|=\sgn_v(f).
$$ 
\end{enumerate}
We can now state our analytic characterization of the master field.
\begin{thm}\label{EUMF}
The master field $\Phi_T:L(\Sbb_T)\to\C$ has the following properties, which characterize it uniquely:
\begin{enumerate}
\item[{\rm (a)}] $\Phi_T$ is continuous in length,
\item[{\rm (b)}] $\Phi_T$ is {\em invariant under reduction:} for all pairs of loops $l_1,l_2$ with $l_1\sim l_2$,
$$
\Phi_T(l_1)=\Phi_T(l_2)
$$
\item[{\rm (c)}] $\Phi_T$ is {\em invariant under area-preserving homeomorphisms:} for all regular loops $l$ and any area-preserving homeomorphism $\th$ of $\Sbb_T$ such that $\th(l)\in L(\Sbb_T)$,
$$
\Phi_T(\th(l))=\Phi_T(l)
$$ 
\item[{\rm (d)}] $\Phi_T$ {\em satisfies the Makeenko--Migdal equations:} for all regular loops $l$, all points $v$ of self-intersection of $l$, and any Makeenko--Migdal flow $\th$ at $(l,v)$,
\begin{equation}\label{SLMM}
\left.\frac{d}{dt}\right|_{t=0} \Phi_T(\theta(t,l))=\Phi_T(l_v)\Phi_T(\hat l_v)
\end{equation}
\item[{\rm (e)}] for all simple loops $l$ and all $n\in\N$,
\begin{equation}\label{SLD}
\Phi_T(l^n)=\frac2{n\pi}\int_0^\infty\cosh\left\{(a_1-a_2)nx/2\right\}\sin\{n\pi\rho_T(x)\}dx
\end{equation}
where $a_1$ and $a_2$ are the areas of the connected components of $\Sbb_T\sm l^*$.
\end{enumerate}
\end{thm}

Note that the integrand in \eqref{SLD} vanishes whenever $\rho_T(x)=0$ or $\rho_T(x)=1$.
In fact, it suffices for uniqueness that property (e) hold in the case $n=1$, as we show in Subsection \ref{MOREU}.

\subsection{Outline of the main argument}\label{OUT}
We now outline the main steps in our proof of Theorems \ref{main} and \ref{EUMF}.
We build progressively an understanding of the limit, first for simple loops, then regular loops, and finally for all loops of finite length.
First, we prove in Subsection \ref{PRSIM} the following statement for simple loops.
The argument uses harmonic analysis in $U(N)$ to express means and covariances of $\tr(H_l^n)$ in terms of a discrete Coulomb gas, whose asymptotics as $N\to\infty$ we can compute. 
Write $L_0(\Sbb_T)$ for the set of simple loops in $L(\Sbb_T)$.
For $l\in L_0(\Sbb_T)$, write $l^*$ for the range of $l$.
Then $\Sbb_T\sm l^*$ has two connected components.
We write $a_1(l)$ for the area of the the component on the left of $l$ and $a_2(l)$ for the area of the component on the right.
Then $a_1(l),a_2(l)>0$ and $a_1(l)+a_2(l)=T$.
Set
\begin{equation}\label{PHITD}
\phi_T(n,a_1,a_2)=\frac2{n\pi}\int_0^\infty\cosh\left\{(a_1-a_2)nx/2\right\}\sin\{n\pi\rho_T(x)\}dx.
\end{equation}

\begin{prop}\label{NFL}
For all $n\in\N$, 
$$
\tr(H_l^n)\to\Phi_T(l^n)=\phi_T(n,a_1(l),a_2(l))
$$ 
uniformly in $l\in L_0(\Sbb_T)$ in $L^2(\PP)$ as $N\to\infty$.
\end{prop}

For $n\in\N$, denote by $L_n(\Sbb_T)$ the set of regular loops having at most $n$ self-intersections.
Write $\overline{L_n(\Sbb_T)}$ for the closure of $L_n(\Sbb_T)$ in $L(\Sbb_T)$.
We say that a uniformly continuous function $\Phi$ on $L_n(\Sbb_T)$ is {\em invariant under reduction} if 
$$
\bar\Phi(l_1)=\bar\Phi(l_2)
$$
for all loops $l_1,l_2\in\overline{L(\Sbb_T)}$ with $l_1\sim l_2$, 
where $\bar\Phi$ is the continuous extension of $\Phi$ to $\overline{L_n(\Sbb_T)}$.
For a simple loop $s$ and $k\in\N$, the $k$-fold concatenation $s^k$ is a limit point of $L_n(\Sbb_T)$ if and only if $k\le n+1$.

The next step is the following proposition, which is proved in Subsection \ref{PRREG}.
The argument is based on the Makeenko--Migdal equations for Wilson loops.

\begin{prop}\label{REG}
For all $n\in\N$,
$$
\tr(H_l)\to\Phi_T(l)
$$ 
uniformly in $l\in L_n(\Sbb_T)$ in $L^2(\PP)$ as $N\to\infty$.
Moreover, the restriction of the master field $\Phi_T$ to $L_n(\Sbb_T)$ is the unique function $L_n(\Sbb_T)\to\C$ with the following properties:
it is uniformly continuous, 
invariant under reduction and under area-preserving homeomorphisms, 
satisfies the Makeenko--Migdal equations \eqref{SLMM},
and satisfies, for all simple loops $s$ and all $k\le n+1$,
$$
\bar\Phi_T(s^k)=\phi_T(k,a_1(s),a_2(s)).
$$
\end{prop}

Finally, we extend to all loops of finite length in the following proposition, which combines the statements of Theorems \ref{main} and \ref{EUMF}.
The proof is given in Section \ref{RL}, using approximation by piecewise geodesics, and by adapting some general arguments of L\'evy \cite{MR2667871}.

\begin{prop}\label{ALL}
For all $l\in L(\Sbb_T)$,
$$
\tr(H_l)\to\Phi_T(l)
$$
in probability as $N\to\infty$.
Moreover, the master field $\Phi_T$ is the unique function $L(\Sbb_T)\to\C$ with the following properties:
it is continuous, invariant under reduction, 
invariant under area-preserving homeomorphisms, 
satisfies the Makeenko--Migdal equations \eqref{SLMM} on regular loops, 
and satisfies \eqref{SLD} for simple loops.
\end{prop}

\subsection{Convergence of spectral measures}\label{SPM}
Let $(H_\g:\g\in P(\Sbb_T))$ be a Yang--Mills holonomy field in $U(N)$.
For $l\in L(\Sbb_T)$, consider the empirical eigenvalue distribution on the unit circle $\U$, given by
$$
\nu_T^N(l)=\frac1N\sum_{i=1}^N\d_{\l_i}
$$
where $\l_1,\dots,\l_N$ are the eigenvalues of $H_l$ enumerated with multiplicity.

\begin{coro}
There is a function $\nu_T:L(\Sbb_T)\to\cM_1(\U)$ such that, for all $l\in L(\Sbb_T)$,
$$
\nu_T^N(l)\to\nu_T(l)
$$
weakly in probability on $\U$ as $N\to\infty$.
Moreover, for all simple loops $l$ and all $n\in\N$,
$$
\int_\U\o^n\nu_T(l)(d\o)=\frac2{n\pi}\int_0^\infty\cosh\left\{(a_1(l)-a_2(l))nx/2\right\}\sin\{n\pi\rho_T(x)\}dx.
$$
Moreover, for $T\in(0,\pi^2]$, all simple loops $l$, and all bounded Borel functions $f$, 
\begin{equation}\label{NUF}
\<f,\nu_T(l)\>=\int_{-\pi}^\pi f(e^{i\th})s_{a_1a_2/T}(\th)d\th
\end{equation}
where $s_t$ is the semi-circle density of variance $t$, given by
\begin{equation}\label{SUF}
s_t(x)=\frac1{2\pi t}\sqrt{4t-x^2},\q |x|\le2\sqrt t.
\end{equation}
\end{coro}
\begin{proof}
By Theorem \ref{main}, for $l\in L(\Sbb_T)$ and all $n\in\N$, we have
$$
\int_\U\o^n\nu_T^N(l)(d\o)=\tr(H_l^n)=\tr(H_{l^n})\to\Phi_T(l^n)
$$
in probability as $N\to\infty$.
Since $\U$ is compact, by a standard tightness argument, it follows that there exists a probability measure $\nu_T(l)$ on $\U$ such that
$$
\int_\U\o^n\nu_T(l)(d\o)=\Phi_T(l^n)
$$
for all $n\in\N$ and such that $\nu_T^N(l)\to\nu_T(l)$ weakly in probability as $N\to\infty$.
By Theorem \ref{EUMF}, $\Phi_T(l^n)$ is given by \eqref{SLD} for all simple loops $l$.
Finally, we will show in Subsection \ref{EVAL} that, for all $T\in(0,\pi^2]$ and all $n\in\N$,
$$
\Phi_T(l^n)=\int_{-\pi}^\pi e^{in\th}s_{a_1a_2/T}(\th)d\th
$$
so \eqref{NUF} holds for polynomials, and so it holds in general.
\end{proof}
Thus, for $T\in(0,\pi^2]$ and for simple loops $l$, the limiting spectral measure $\nu_T(l)$ has a semi-circle density on $\U$, with
$$
\supp(\nu_T(l))=\{e^{i\th}:|\th|\le2\sqrt{a_1a_2/T}\}.
$$
The maximal support is then $\{e^{i\th}:|\th|\le\sqrt{T}\}$, achieved when $a_1=a_2=T/2$.
Note that, in the critical case $T=\pi^2$, the two endpoints of the maximal support meet at $\th=\pm\pi$.

\subsection{Free unitary Brownian loop}\label{FUBL} 
As a corollary of Theorem \ref{main}, we show that the Brownian loop in $U(N)$ based at $1$ of parameter $T$ converges in non-commutative distribution as $N\to\infty$.   
Moreover, we identify the limiting empirical distribution of eigenvalues at each time $t\in[0,T]$.

Consider the free unital $*$-algebra $\cA_T$ of polynomials over $\C$ in the variables $(X_t:t\in[0,T])$ and their inverses.
Thus, each element $P\in\cA_T$ is a non-commutative polynomial 
$$
P=p(X_t,X_t^{-1}:t\in[0,T])
$$
with coefficients in $\C$,
and $*$ is the conjugate-linear, anti-multiplicative involution on $\cA_T$ such that
$$
X_t^*=X_t^{-1}.
$$

For each $N\in\N$, there exists a Brownian loop $B^N=(B^N_t:t\in[0,T])$ in $U(N)$ based at $1$ of parameter $T$.
Define a random non-negative unit trace%
%%%%%%%%%%%%%%%%%%%%%%%%%%%%%%%%%%%%%%%%%%%%%%%%%%%%%%%%%%%%%%%%%%%%%%%%%%%%%%%%%%%%%%%%
\footnote{Recall that a linear map $\t$ on a unital $*$-algebra $\cA$ is a {\em non-negative unit trace} if, for all $x,y\in\cA$,
$$
\t(xx^*)\ge0,\q \t(1)=1,\q\t(xy)=\t(yx).
$$
The pair $(\cA,\t)$ is then a {\em non-commutative probability space}.}
%%%%%%%%%%%%%%%%%%%%%%%%%%%%%%%%%%%%%%%%%%%%%%%%%%%%%%%%%%%%%%%%%%%%%%%%%%%%%%%%%%%%%%%%
on $\cA_T$ by setting 
$$
\t_N(P)=\tr(p(B^N_t,(B^N_t)^{-1}:t\in[0,T])).
$$ 

\begin{thm}\label{FUBLT} 
There is a non-negative unit trace $\t_\infty$ on $\cA_T$ such that, for all $P\in\cA_T$, 
$$
\tau_N(P)\to\t_\infty(P)\q\text{in probability as $N\to\infty$.}
$$
\end{thm}
\begin{proof} 
It will suffice to consider the case where $B^N$ is constructed from a Yang--Mills holonomy field $(H_\g:\g\in P(\Sbb_T))$ in $U(N)$, as in Section \ref{section embedded bridge}.
Then, for some $p\in\Sbb_T$ and some family of loops $l(t)\in L(\Sbb_T)$ based at $p$, we have
$$
B^N_t=H_{l(t)}\q\text{almost surely, for all $t\in[0,T]$}.
$$
Consider first the case of a monomial $P=X_{t_1}^{\ve_1}\dots X_{t_n}^{\ve_n}$ with $\ve_1,\dots,\ve_n\in\{-1,1\}$ and set $l_P=l(t_n)^{\ve_n}\dots l(t_1)^{\ve_1}$. 
Then, by Theorem \ref{main}, 
$$
\t_N(P)=\tr((B_{t_1}^N)^{\ve_1}\dots(B_{t_n}^N)^{\ve_n})=\tr(H_{l(t_1)}^{\ve_1}\dots H_{l(t_n)}^{\ve_n})=\tr(H_{l_P})\to\Phi_T(l_P)
$$
in probability as $N\to\infty$.
Define $\t_\infty(P)=\Phi_T(l_P)$ for all monomials $P$ and extend $\t_\infty$ linearly to $\cA_T$.
Then $\tau_N(P)\to\tau_\infty(P)$ in probability as $N\to\infty$, for all $P\in\cA_T$, 
and $\t_\infty$ inherits the property of being a non-negative unit trace from its random approximations $\t_N$.
\end{proof}

Given a non-commutative random process $x=(x_t:t\in[0,T])$ in a non-commutative probability space $(\cA,\t)$, let us say that $x$ is a
{\em free unitary Brownian loop} if, for all $n$, all $t_1,\dots,t_n\in[0,T]$ and all $(y_{t_k},Y_{t_k})\in\{(x_{t_k},X_{t_k}),(x^*_{t_k},X^*_{t_k})\}$, 
$$
\t(y_{t_1}\dots y_{t_n})=\t_\infty(Y_{t_1}\dots Y_{t_n}).
$$
In particular, the canonical process $(X_t:t\in[0,T])$ is a free unitary Brownian loop in $(\cA_T,\t_\infty)$.
We shall see in Section \ref{Prop MF} that, in the subcritical regime $T\le \pi^2$, 
a free unitary Brownian loop $x$ has the same marginal distributions as $e^{ib}$, where $b$ is a free Brownian loop with the same lifetime.
Thus $\nu_t$ is the push-forward of a Wigner law by the exponential mapping to the circle. 
However, we shall also see that the full non-commutative distributions of $x$ and $e^{ib}$ are different.

\subsection{The master field as a holonomy in $U(\infty)$}
For simplicity, we have presented the notion of Yang--Mills holonomy field as a process $(H_\g:\g\in P(\Sbb_T))$ with values in $U(N)$.
However, the property of gauge-invariance allows us to think of it a little more generally, which will help to motivate the main construction of this subsection.
Suppose we are given a family of complex vector spaces $V=(V_x:x\in\Sbb_T)$, each equipped with a Hermitian inner product and having dimension $N$.
Choose\footnote{If $V$ is given the structure of a non-trivial vector bundle, then $s$ will necessarily be a discontinuous section of $V$.}, 
for each $x\in\Sbb_T$, a complex linear isometry $s(x):\C^N\to V_x$.
Given a Yang--Mills holonomy field $(H_\g:\g\in P(\Sbb_T))$ in $U(N)$, for each $\g\in P_{x,y}(\Sbb_T)$, we can define a complex linear isometry $T_\g:V_x\to V_y$ by
$$
T_\g=s(y)H_\g s(x)^{-1}.
$$
Then, by gauge invariance, the law of the process $(T_\g:\g\in P(\Sbb_T))$ does not depend on the choice of the family of isometries $(s(x):x\in\Sbb_T)$.
We call any process with this law a {\em Yang--Mills holonomy field in $V$}.
The original holonomy field $(H_\g:\g\in P(\Sbb_T))$ then corresponds to the case where $V_x=\C^N$ for all $x$.

We now carry out the suggestion of Singer \cite{MR1373007}, to use a variation of the Gelfand-Naimark-Segal construction to obtain from the master field a family of Hilbert spaces $(V_x:x\in\Sbb_T)$, 
equipped with a canonical connection, viewed as a family of unitary transport operators indexed by $P(\Sbb_T)$. 
Fix a reference point $r\in\Sbb_T$ and consider for each $x\in\Sbb_T$ the vector space $\cV_x$ of complex functions on $P_{r,x}(\Sbb_T)$ of finite support.
Thus, each $v\in\cV_x$ has the form
$$
v=\sum_{i=1}^na_i\d_{\g_i}
$$
for some $n\ge0$, with $a_i\in\C$ and $\g_i\in P_{r,x}(\Sbb_T)$ for all $i$.
There is a unique Hermitian form $\<.,.\>$ on $\cV_x$ such that, for all $\g_1,\g_2\in P_{r,x}(\Sbb_T)$,
$$
\<\d_{\g_1},\d_{\g_2}\>=\Phi_T(\g_1\g_2^{-1}).
$$
By Proposition \ref{state master field}, this form is non-negative definite.
For $x,y\in\Sbb_T$ and $\g\in P_{x,y}(\Sbb_T)$, there is a unique complex linear map $\cT_\g:\cV_x\to\cV_y$ such that, for all $\g_1\in P_{r,x}(\Sbb_T)$,
$$
\cT_\g\d_{\g_1}=\d_{\g_1\g}.
$$
Note that, for $\g_1,\g_2\in P_{r,x}(\Sbb_T)$, 
$$
\<\cT_\g\d_{\g_1},\cT_\g\d_{\g_2}\>=\<\d_{\g_1\g},\d_{\g_2\g}\>=\Phi_T(\g_1\g\g^{-1}\g_2^{-1})=\Phi_T(\g_1\g_2^{-1})=\<\d_{\g_1},\d_{\g_2}\>.
$$
It follows that $\cT_\g$ preserves the Hermitian form $\<.,.\>$.

Note that, if $\g_1\sim\g_2$, then $\<\d_{\g_1},\d_{\g_2}\>=1$ and so $\<\d_{\g_1}-\d_{\g_2},\d_{\g_1}-\d_{\g_2}\>=0$.
Set
$$
\cI_x=\{v\in\cV_x:\<v,v\>=0\}
$$
and denote by $V_x$ the complex Hilbert space obtained by completing the quotient space $\cV_x/\cI_x$ with respect to the Hermitian inner product induced by $\<.,.\>$.
Then, for $x,y\in\Sbb_T$ and $\g\in P_{x,y}(\Sbb_T)$, the map $\cT_\g$ induces a Hilbert space isometry $T_\g:V_x\to V_y$.
Moreover, the family of maps $(T_\g:\g\in P(\Sbb_T))$ has the following properties
$$
T_x=1_x,\q T_{\g^{-1}}=(T_\g)^{-1},\q T_{\g_1\g_2}=T_{\g_2}T_{\g_1}.
$$
Here, we write $x$ for the constant loop at $x$, $1_x$ for the identity map on $V_x$, and we assume that $\g_1$ ends where $\g_2$ starts.
For each $x\in\Sbb_T$, given a path $\g\in P_{r,x}(\Sbb_T)$, we can define a state $\t_\g$ on the set of bounded linear operators $\cB(V_x)$ by
$$
\t_\g(A)=\<[\d_\g],A[\d_\g]\>
$$
where $[\d_\g]=\d_\g+\cI_x$.
Then, for all $l\in L_x(\Sbb_T)$,
$$
\t_\g(T_l)=\<\d_\g,\cT_l\d_\g\>=\<\d_\g,\d_{\g l}\>=\Phi_T(\g l^{-1}\g^{-1})=\Phi_T(l).
$$
Recall from Proposition \ref{state master field} that $\Phi_T(x)=1$ and $\Phi_T(l_1l_2)=\Phi_T(l_2l_1)$.
Then, on restricting $\t_\g$ to the von Neumann algebra $\cA_x$ in $\cB(V_x)$ generated by $(T_l:l\in L_x(\Sbb_T))$, we obtain a non-negative unit trace $\t_x$ on $\cA_x$, which does not depend on the choice of path $\g$.

We note some further properties of $(\cA_x,\t_x)$.
First, for all integers $n$, and all $l\in L_x(\Sbb_T)$,
$$
\t_x(T_l^n)=\Phi_T(l^n)=\int_\U\o^n\nu_l(d\o)
$$
where $\nu_l$ is the limit spectral measure obtained in Subsection \ref{SPM}.
So $\nu_l$ is the spectral measure of $T_l$.
Second, since the master field is invariant under area-preserving diffeomorphisms $\Sbb_T$, the choice of such a diffeomorphism $\psi$ gives an isomorphism $(\cA_x,\t_x)\to(\cA_y,\t_y)$ whenever $\psi(x)=y$.

Singer \cite{MR1373007} conjectured, without explicit construction, that the von Neumann algebras $\cA_x$ were factors, that is to say, their centres were trivial\footnote{See for example  \cite{MR1873025}.}.
If this conjecture holds then, since%
%%%%%%%%%%%%%%%%%%%%%%%%%%%%%%%%%%%%%%%%%%%%%%%%%%%%%%%%%%%%%%%%
\footnote{See for example, Section 8.4 of \cite{MR859186}.} 
%%%%%%%%%%%%%%%%%%%%%%%%%%%%%%%%%%%%%%%%%%%%%%%%%%%%%%%%%%%%%%%%
the spectral measures $\nu_l$ are absolutely continuous, at least for simple loops, and since $\t_x$ is a finite normalized trace, 
we see that 
$$
\{\t_x(p):p\in\cA_x, p^2=p\}=[0,1]
$$ 
and $\cA_x$ must be of type $\text{II}_1$ and have unique state $\t_x$.

\section{Harmonic analysis in $U(N)$ and a discrete $\b$-ensemble}\label{HADBE}
\subsection{A representation formula}
Let $(H_\g:\g\in P(\Sbb_T))$ be a Yang--Mills holonomy field in $U(N)$.
We obtain in this subsection a formula for the moments of the holonomy $H_l$ of a simple loop $l$ in terms of a certain discrete $\b$-ensemble, with $\b=2$.
Set 
$$
\Z_\sym=\begin{cases}\Z,&\text{ if $N$ is odd},\\\Z+1/2,&\text{ if $N$ is even.}\end{cases}
$$
Consider the discrete $\b$-ensemble $\L$ in $N^{-1}\Z_\sym$ given by
\begin{equation}
\label{def beta ens}
\PP\left(\L=\l\right)\propto\prod_{\substack{j,k=1\\j<k}}^N(\l_j-\l_k)^2\prod_{i=1}^Ne^{-N\l_i^2T/2}
\end{equation}
where $\l$ runs over decreasing sequences $(\l_1,\dots,\l_N)$ in $N^{-1}\Z_\sym$.
For $\a\in\R\sm\{0\}$ and for $z\in\C$ with $|\a||z-\l_j|>1$ for all $j$, set
$$
G^\a_\l(z)=\frac\a N\sum_{j=1}^N\Log\left(1+\frac1{\a(z-\l_j)}\right)
$$ 
where $\Log$ denotes the principal value of the logarithm.
Then, for $a\in(0,T)$, set $I_0^a(\l)=1$ and define for $n\in\Z\sm\{0\}$
\begin{equation*}
\label{INA}
I_n^a(\l)=\frac{e^{-an^2/(2N)}}{2\pi in}\int_\g\exp\{-n(az-G^{N/n}_\l(z))\}dz
\end{equation*}
where $\g$ is any positively oriented simple loop around the set 
$$
[\l_N,\l_1]+\{z\in\C:|z|\le|n|/N\}.
$$

\begin{prop} 
\label{Representation simple loops beta ensemble}
Let $l\in L(\Sbb_T)$ be a simple loop which divides $\Sbb_T$ into components of areas $a$ and $b$.
Then, for all $m,n\in\Z$,
$$
\E(\tr(H_l^{-m})\tr(H_l^n))=\E(I_m^a(\L)I_n^b(\L)).
$$
\end{prop}

To prove these identities, we will use the decomposition of the heat kernel as a sum over the characters of $U(N)$. 
The results we use may be found for example in \cite{MR1920389}.  
For $\l\in(\Z_\sym)^N$, set
$$
\|\l\|^2=\frac1N\sum_{j=1}^N\l_j^2. 
$$
Write $\rho=(\rho_1,\dots,\rho_N)$ for the unique minimizer of $\|.\|$ among decreasing sequences in $(\Z_\sym)^N$,
which is given by
$$
\rho_j=\frac12(N+1)-j.
$$
For $\l\in\Z^N$, there is a unique continuous function $\chi_\l:U(N)\to\C$ given by the Weyl character formula
\begin{equation}
\label{Weyl character formula}
\chi_\l(g)\det(e^{i\th_j\rho_k})_{j,k=1}^N=\det(e^{i\th_j(\l_k+\rho_k)})_{j,k=1}^N,\q g\in U(N)
\end{equation}
where $e^{i\th_1},\dots,e^{i\th_N}$ are the eigenvalues of $g$.  
Then 
$$
(\chi_\l:\l\in\Z^N,\l_1\ge\dots\ge\l_N)
$$ 
is a parametrization of the set of characters of irreducible representations of $U(N)$.
For characters $\chi_\l$ and $\chi_\mu$, we have
\begin{equation}
\label{orthochar}
\int_{U(N)}\chi_\l(g)\overline{\chi_\mu(g)}dg=\int_{U(N)}\chi_\l(g)\chi_\mu(g^{-1})dg=\d_{\l,\mu} 
\end{equation}
and
\begin{equation}
\label{Eigenvalues Laplacian}
\Delta\chi_\l=-(\|\l+\rho\|^2-\|\rho\|^2)\chi_\l.
\end{equation}
Moreover, the heat kernel $(p_t(g):t\in(0,\infty),g\in U(N))$ is given by the following absolutely converging sum over characters
\begin{equation}\label{heat}
p_t(g)=e^{\|\rho\|^2t/2}\sum_\l\chi_\l(1)\chi_\l(g)e^{-\|\l+\rho\|^2t/2}.
\end{equation}
The character values at the identity are given by the Weyl dimension formula
\begin{equation}
\label{dimensionIrrep}
\chi_\l(1)=\prod_{\substack{j,k=1\\j<k}^N}\frac{\l_j+\rho_j-\l_k-\rho_k}{\rho_j-\rho_k}.  
\end{equation}
The change of variable $\mu=\l+\rho$ gives a convenient reparametrization of the set of characters by
$$
W=\{\mu\in(\Z_\sym)^N:\mu_1>\dots>\mu_N\}.
$$
For $x\in(\Z_\sym)^N$ with all components distinct, we will write $[x]$ for the decreasing rearrangement of $x$.
From \eqref{Weyl character formula}, we see that, 
$$
\chi_{x-\rho}=\ve(x)\chi_{[x]-\rho}
$$
where 
$$
\ve(x)=\begin{cases}\sgn(\s),&\text{if $x$ has all components distinct},\\0,&\text{otherwise},\end{cases}
$$
where $\s$ is the unique permutation such that $[x]_j=x_{\s(j)}$ for all $j$.
Then the orthogonality relation (\ref{orthochar}) extends to all $x,y\in(\Z_\sym)^N$ in the form 
\begin{equation}
\label{orthocharsign}
\int_{U(N)}\chi_{x-\rho}(g)\chi_{y-\rho}(g^{-1})dg=\ve(x)\ve(y)\d_{[x],[y]}. 
\end{equation}
To compute the desired moments of holonomy traces, we shall need to take the pointwise product of the trace on the fundamental representation $\C^N$ with the characters.   
A straightforward computation using (\ref{Weyl character formula}) shows that, for all $n\in\Z$,%
\begin{equation}\label{product Trace}
\chi_\l(g)\Tr(g^n)
=\sum_{j=1}^N\chi_{\l+n\o^j}(g).%1_{\{\l+n\o^j\in\cO\}}
\end{equation}
where $\o^j$ is the $j$th elementary vector in $\Z^N$.

\begin{proof}[Proof of Proposition \ref{Representation simple loops beta ensemble}]    
From the definition of the Yang--Mills measure, we have
$$
\E(\tr(H_l^{-m})\tr(H_l^{n}))\propto\int_{U(N)}p_a(g)\tr(g^{-m})\tr(g^{n})p_b(g^{-1})dg
$$
where $\propto$ signifies equality up to a constant independent of $m$ and $n$.
We expand the heat kernel in characters to obtain
\begin{align*} 
&\int_{U(N)}p_a(g)\tr(g^{-m})\tr(g^{n})p_b(g^{-1})dg\\ 
&\propto\sum_{\l,\mu\in W}e^{-\|\l\|^2a/2-\|\mu\|^2b/2}\chi_{\l-\rho}(1)\chi_{\mu-\rho}(1)
\int_{U(N)}\chi_{\l-\rho}(g)\tr(g^{-m})\tr(g^{n})\chi_{\mu-\rho}(g^{-1})dg.
\end{align*} 
The interchange of summation and integration here is valid because $a,b>0$ which ensures absolute convergence.
By orthogonality of characters (\ref{orthocharsign}) and the product rule (\ref{product Trace}), for all $\l,\mu\in W$,
\begin{align*}
&\int_{U(N)}\chi_{\l-\rho}(g)\tr(g^{-m})\tr(g^{n})\chi_{\mu-\rho}(g^{-1})dg\\
&\q\q=\frac1{N^2}\sum_{j,k=1}^N\ve(\l-m\o^j)\ve(\mu-n\o^k)\d_{[\l-m\o^j],[\mu-n\o^k]}.
\end{align*}
Now, for $\nu\in W$, we have $[\l-m\o^j]=[\mu-n\o^k]=\nu$ for some $j,k\in\{1,\dots,N\}$
 if and only if $\l=[\nu+m\o^{j'}]$ and $\mu=[\nu+n\o^{k'}]$ for some $j',k'\in\{1,\dots,N\}$, 
and then 
$$
N\|\l\|^2=N\|\nu\|^2+2m\nu_{j'}+m^2,\q
N\|\mu\|^2=N\|\nu\|^2+2n\nu_{k'}+n^2
$$
and
$$
\ve(\l-m\o^j)=\ve(\nu+m\o^{j'}),\q
\ve(\mu-m\o^k)=\ve(\nu+n\o^{k'})
$$
so, using the dimension formula \eqref{dimensionIrrep},
\begin{align*}
\chi_{\l-\rho}(1)\ve(\l-m\o^j)&=\chi_{\nu+m\o^{j'}-\rho}(1)=\chi_{\nu-\rho}(1)\prod_{i\not=j}\frac{\nu_j+m-\nu_i}{\nu_j-\nu_i},\\
\chi_{\mu-\rho}(1)\ve(\mu-n\o^k)&=\chi_{\nu+n\o^{k'}-\rho}(1)=\chi_{\nu-\rho}(1)\prod_{i\not=k}\frac{\nu_k+n-\nu_i}{\nu_k-\nu_i}.
\end{align*}
Hence
\begin{align*}
&\int_{U(N)}p_a(g)\tr(g^{-m})\tr(g^{n})p_b(g^{-1})dg\\
&\q\q\propto\sum_{\nu\in W}\prod^N_{\substack{j,k=1\\j<k}}(\nu_j-\nu_k)^2e^{-\|\nu\|^2T/2}J(\nu,m,a)J(\nu,n,b)
\end{align*}
where 
$$
J(\nu,m,a)=e^{-m^2a/(2N)}\frac1N\sum_{j=1}^Ne^{-ma\nu_j/N}\prod_{i\not=j}\frac{\nu_j+m-\nu_i}{\nu_j-\nu_i}.
$$
Note that $J(\nu,0,a)=1=I^a_0(\nu/N)$ and, for $|m|\ge1$,
\begin{align*}
J(\nu,m,a)
&=\frac{e^{-m^2a/(2N)}}{2\pi imN}\int_{N\g(\nu)}\prod_{j=1}^N\left(1+\frac{m}{z-\nu_j}\right)e^{-maz/N}dz\\
&=\frac{e^{-m^2a/(2N)}}{2\pi im}\int_{\g(\nu)}\exp\{-m(az-G_{\nu/N}^{N/m}(z))\}dz=I_m^a(\nu/N).
\end{align*}   
So we obtain
\begin{align*}
\E(\tr(H_l^{-m})\tr(H_l^{n}))
&\propto\sum_{\nu\in W}\prod_{\substack{j,k=1\\j<k}}^N(\nu_j-\nu_k)^2e^{-\|\nu\|^2T/2}I_m^a(\nu/N)I_{n}^b(\nu/N)\\
&\propto\sum_{N\l\in W}\prod_{\substack{j,k=1\\j<k}}^N(\l_j-\l_k)^2\prod_{i=1}^Ne^{-N\l^2_iT/2}I_m^a(\l)I_{n}^b(\l)\\
&\propto\E(I^a_m(\L)I_{n}^b(\L)).
\end{align*}
Since the identity $\E(\tr(H_l^{-m})\tr(H_l^{n}))=\E(I^a_m(\L)I_{n}^b(\L))$ holds for $m=n=0$, it therefore holds for all $m$ and $n$.
\end{proof}
The first part of the above proof follows ideas from the physics literature \cite{MR1262293,MR1297298}.
The use of contour integrals in writing the function $J$ and in the formulation of Proposition \ref{Representation simple loops beta ensemble}
is new and provides us with a route to make rigorous the asymptotics performed in \cite{MR1262293,MR1297298}.
 
\subsection{Concentration for the discrete $\b$-ensemble and tightness of the support}
We shall need two facts about the discrete $\b$-ensemble $\L$ defined in equation \eqref{def beta ens}.
Recall from \eqref{SCE} the functional
$$
\cI_T(\mu)=\int_{\R^2}\left\{\tfrac12(x^2+y^2)T-2\log|x-y|\right\}\mu(dx)\mu(dy)
$$
defined for probability measures $\mu$ on $\R$ such that $\mu([a,b])\le b-a$ for all intervals $[a,b]$.
We extend $\cI_T$ to $\cM_1(\R)$ by setting $\cI_T(\mu)=\infty$ if $\mu$ does not satisfy this constraint.
Guionnet and Ma{\"{\i}}da \cite{MR2198201} showed the following large deviation principle.

\begin{thm}\label{LDP MaidaGuionnet} 
The laws of the normalized empirical distributions
$$
\mu_\L=\frac1N\sum_{i=1}^N\d_{\L_i}
$$
satisfy a large deviation principle on $\cM_1(\R)$ with rate function $\cI_T$ and speed $N^2$.
\end{thm}

We need also a tightness result for the positions $\L_N$ and $\L_1$ of the leftmost and rightmost particles, which is obtained by a variation on ideas of Johansson \cite{MR1737991}.
See also F\'eral \cite{MR2483725}.

\begin{lem}\label{bounded support lemma}  
Set
$$
\L^*=\max\{|\L_1|,|\L_N|\}.
$$
For all $p\in[0,\infty)$, there are constants $C,R<\infty$ depending only on $p$ and $T$ such that
$$
\E\left(e^{p\L^*}1_{\{\L^*>R\}}\right)\le Ce^{-N}.   
$$
\end{lem}
\begin{proof}  
It will be convenient in this proof to switch our convention so that we label the particle positions in increasing order, so $\L_N$ is the position of the rightmost particle.
Then, by symmetry, it will suffice to show that, for all $p\in[0,\infty)$, there are constants $C,R<\infty$ depending only on $p$ and $T$ such that
$$
\E\left(e^{p\L_N}1_{\{\L_N>R\}}\right)\le Ce^{-N}.   
$$
Fix $N$ and, for $M=N-1$ and $M=N$, set
$$
Z_M=\sum_\l\prod_{\substack{j,k=1\\j<k}}^M(\l_j-\l_k)^2\prod_{i=1}^Me^{-N\l_i^2T/2}
$$
where the sum is taken over the set $S_M$ of increasing sequences $\l=(\l_1,\dots,\l_M)$ in $N^{-1}\Z_\sym$.
Then
\begin{align*}
&\E\left(e^{p\L_N}1_{\{\L_N>R\}}\right)\\
&=\frac1{Z_N}\sum_{\l_N}\sum_\l e^{p\l_N}1_{\{\l_N>R\vee\l_{N-1}\}}\prod_{\substack{j,k=1\\j<k}}^N(\l_j-\l_k)^2\prod_{i=1}^Ne^{-N\l_i^2T/2}\\
&\le\frac1{Z_N}\sum_se^{ps-Ns^2T/2}1_{\{s>R\}}\sum_\l\prod_{\substack{j,k=1\\j<k}}^{N-1}(\l_j-\l_k)^2\prod_{i=1}^{N-1}(s-\l_i)^2e^{-N\l_i^2T/2}\\
&=\frac{Z_{N-1}}{Z_N}\sum_se^{ps-Ns^2T/2}1_{\{s>R\}}\E\left(\exp\left\{\int_\R\log((s-x)^2)M_{N-1}(dx)\right\}\right)
\end{align*}
where $\l_N$ and $s$ are summed over $N^{-1}\Z_\sym$ and $\l=(\l_1,\dots,\l_{N-1})$ is summed over $S_{N-1}$, 
where $M_{N-1}=\sum_{i=1}^{N-1}\d_{\L_i}$
with $\L=(\L_1,\dots,\L_{N-1})$ a random variable in $S_{N-1}$ having distribution
$$
\PP(\L=\l)=\frac1{Z_{N-1}}\prod_{\substack{j,k=1\\j<k}}^{N-1}(\l_j-\l_k)^2\prod_{i=1}^{N-1}e^{-N\l_i^2T/2}.
$$
We use the inequality $(s-x)^2\le(1+s^2)(1+x^2)$ to see that
$$
\int_\R\log((s-x)^2)M_{N-1}(dx)\le(N-1)\log(1+s^2)+\int_\R\log(1+x^2)M_{N-1}(dx).
$$
Hence we have
\begin{align*}
&\E\left(e^{p\L_N}1_{\{\L_N>R\}}\right)\\
&=\frac{Z_{N-1}}{Z_N}\E\left(\exp\left\{\int_\R\log(1+x^2)M_{N-1}(dx)\right\}\right)\sum_s(1+s^2)^Ne^{ps-Ns^2T/2}1_{\{s>R\}}.
\end{align*}
Now by a straightforward modification of the arguments in \cite[Lemmas 4.1 and 4.5]{MR1737991}, there is a constant $c<\infty$, depending only on $T$, such that
$$
\frac{Z_{N-1}}{Z_N}\le ce^{cN},\q\E\left(\exp\left\{\int_\R\log(1+x^2)M_{N-1}(dx)\right\}\right)\le ce^{cN}.
$$
On the other hand, there exist $C,R<\infty$, depending only on $p,c$ and $T$ such that
$$
\sum_s(1+s^2)^Ne^{ps-Ns^2T/2}1_{\{s>R\}}\le Cc^{-2}e^{-(2c+1)N}
$$
for all $N$.
The claim follows.
\end{proof}

\subsection{Evaluation of some contour integrals}\label{EVAL}
In passing from the limit particle density $\rho_T$ for the $\b$-ensemble to the evaluation of the master field on simple loops,
we will need to evaluate certain contour integrals expressed in terms of the Stieltjes transform
$$
G_T(z)=\int_\R\frac{\rho_T(x)dx}{z-x}.
$$
The following calculation is taken from \cite{MR1262293,MR1297298}. 
\begin{prop}\label{SCC}
Let $T\in(0,\pi^2]$ and let $a,b\in(0,T)$ with $a+b=T$.
Let $\g$ be a positively oriented closed curve around the set $[-2/\sqrt T,2/\sqrt T]$.
Then, for all $n\in(0,\infty)$, %do not change -- needed for scaling argument
$$
\frac1{2\pi in}\int_\g\exp\{-n(az-G_T(z))\}dz=\int_\R e^{inx}s_{a(T-a)/T}(x)dx
$$
where $s_t$ is the semi-circle density \eqref{SUF} of variance $t$.
\end{prop}
\begin{proof} 
Since $T\in(0,\pi^2]$, we have $\rho_T=s_{1/T}$.
Then $\rho_T(x)=\sqrt{T}\rho_1(\sqrt Tx)$ so, by a scaling argument, it will suffice to consider the case $T=1$.
A standard calculation of the Steiltjes transform gives
$$
G_1(z)=\int_\R\frac{\rho_1(x)dx}{z-x}=\frac{z-\sqrt{z^2-4}}2.
$$
Note that $G_1$ maps $\C\sm[-2,2]$ conformally to the punctured unit disc $\DD\sm\{0\}$ with inverse $z+1/z$.
Also, $G_1(\g)$ is a negatively oriented closed curve around $\{0\}$.
Write $b=1-a$. 
We make the change of variable $w=G_1(z)$ to obtain
\begin{align*}
&\frac1{2\pi in}\int_{\g}\exp\{-n(az-G_1(z))\}dz
=\frac1{2\pi in}\int_{G_1(\g)}\exp\{n(bw-aw^{-1})\}(1-w^{-2})dw\\
&\q\q=\frac1{2\pi n}\int_0^{2\pi}\exp\{n(be^{-i\th}-ae^{i\th})\}(e^{i\th}-e^{-i\th})d\th\\
&\q\q=\frac1{2\pi n}\sum_{k=0}^\infty\frac{n^k}{k!}\sum_{j=0}^k\binom{k}j\int_0^{2\pi}(be^{-i\th})^j(-ae^{i\th})^{k-j}(e^{i\th}-e^{-i\th})d\th\\
&\q\q=\sum_{m=0}^\infty\frac{(-n^2ab)^m}{m!(m+1)!}=\int_\R e^{inx}s_{ab}(x)dx 
\end{align*}
where we used in the last equality the moment formula
$$
\int_\R x^{2m}s_t(x)dx=\frac{t^{2m}}{m!(m+1)!}.
$$
\end{proof}

More generally, for all $T\in(0,\infty)$, the following is obtained in \cite[equation (4.12)]{MR3440793}
\begin{equation}
G_T(z)=\frac{zT}2-\frac2{\b z}\sqrt{(z^2-\a^2)(z^2-\b^2)}\int_0^1\frac{ds}{(1-\a^2s^2/z^2)\sqrt{(1-s^2)(1- k^2s^2)}}\label{Stieltjes Coulomb}
\end{equation}
where $k=\a/\b\in(0,1)$.
Moreover, for $|x|\in[\a,\b]$, in the limit $z\to x$ with $z\not\in\R$, we have
\begin{equation}\label{conjugate density}
\re(G_T(z))\to\mathrm{PV}\int\frac{\rho_T(y)dy}{x-y}=\frac{xT}2.
\end{equation}

\begin{prop} 
\label{formula sin} 
Let $T\in(0,\infty)$ and let $a,b\in(0,T)$ with $a+b=T$.
Let $\g$ be a positively oriented closed curve around the set $[-\b,\b]$.
Then, for all $n\in\N$, 
$$
\frac1{2\pi in}\int_\g\exp\{-n(az-G_T(z))\}dz=\frac2{n\pi}\int_0^\infty\cosh\left\{(a-b)nx/2\right\}\sin\{n\pi\rho_T(x)\}dx.
$$
\end{prop}
\begin{proof} 
Since the integrand of the left-hand side is holomorphic in $\C\sm[-\b,\b]$, 
we can take $\g$ to be the anti-clockwise boundary of $[-\b-\ve,\b+\ve]\times[-\ve,\ve]$ for any $\ve>0$. 
Now, as $\rho_T$ is H\"older continuous, by the Plemelj-Sokhotskyi formula \cite{MR1106850}, %Gakhov, F. D. Boundary value problems.
$G_T$ can be continuously extended, as $G_{+}$ and $G_{-}$ say, on $\overline{\H}=\{z\in\C:\im(z)\ge0\}$ and  $-\overline{\H}$, with
$$
G_\pm(x)=\mathrm{PV}\int_\R\frac{\rho_T(y)dy}{x-y}\mp i\pi\rho_T(x)=\frac{xT}2\mp i\pi\rho_T(x)
$$
for any $x\in\R$.  
We can take the limit $\ve\to0$ in the contour integral, using the dominated convergence theorem, to obtain
$$
\frac1{n\pi}\int_\R\exp\left\{(a-b)nx/2\right\}\sin\{n\pi\rho_T(x)\}dx.     
$$
But $\rho_T$ is symmetric, so this gives the claimed identity.
\end{proof}

\subsection{Proof of Proposition \ref{NFL}}\label{PRSIM}
Consider the discrete $\b$-ensemble $\L$ defined by \eqref{def beta ens}.
By Theorem \ref{LDP MaidaGuionnet}, 
\begin{equation}\label{WPR}
\mu_\L\to\mu_T\q\text{weakly in probability on $\R$ as $N\to\infty$}.
\end{equation}
Fix $n\in\N$.
By Lemma \ref{bounded support lemma}, there exist $C,R\in(0,\infty)$, independent of $N$, such that
\begin{align}\label{ENT}
\E(e^{2nT\L^*}1_{\O_R^c})\le Ce^{-N}
\end{align}
where
$$
\L^*=\max\{|\L_1|,|\L_N|\},\q\O_R=\{\supp(\mu_\L)\sse[-R,R]\}=\{\L^*\le R\}.
$$
We increase the value of $R$ if necessary so that
$$
\supp(\mu_T)\sse[-R,R].
$$
Denote by $\g_R$ the positively oriented boundary of the set
$$
[-R,R]+\{z\in\C:|z|\le1\}.
$$
Recall that, for $\a\in(0,\infty)$ and $\dist(z,\supp(\mu_\L))>1/\a$,  we set
$$
G_\L^\a(z)=\a\int_\R\Log\left(1+\frac1{\a(z-x)}\right)\mu_\L(dx).
$$
For $N\ge n+1$, the contour $\g_{R\vee\L^*}$ contains the set 
$$
\supp(\mu_\L)+\{z\in\C:|z|\le n/N\}
$$
so we can write, for $a\in(0,T)$,
$$
I_n^a(\L)=\frac{e^{-an^2/(2N)}}{2\pi in}\int_{\g_{R\vee\L^*}}\exp\{-n(az-G_\L^{N/n}(z))\}dz.
$$
Recall also that we set
$$
G_T(z)=\int_\R\frac{\mu_T(dx)}{z-x}
$$
and, for $a,b>0$ with $a+b=T$,
$$
I_n^a=I_n^b=\frac2{n\pi}\int_0^\infty\cosh\left\{(a-b)nx/2\right\}\sin\{n\pi\rho_T(x)\}dx
$$
and that, by Proposition \ref{formula sin}, 
$$
I_n^a=\frac1{2\pi in}\int_{\g_R}\exp\{-n(az-G_T(z))\}dz.
$$

In Proposition \ref{Representation simple loops beta ensemble} we showed that, for any simple loop $l\in L(\Sbb_T)$ which divides $\Sbb_T$ into components of areas $a$ and $b$,
$$
\E(\tr(H_l^n))=\E(I_n^a(\L))=\E(I_n^b(\L))
$$
and
$$
\E(|\tr(H_l^n)|^2)=\E(\tr(H_l^{-n})\tr(H_l^n))=\E(I_n^a(\L)I_n^b(\L)).
$$
We will show that, for all $n\in\N$, in the limit $N\to\infty$, uniformly in $a\in(0,T)$,
\begin{equation}\label{EIN}
\E(I_n^a(\L))\to I_n^a,\q\E(I_n^a(\L)I_n^b(\L))\to I_n^aI_n^b.
\end{equation}
Then
$$
\E(\tr(H_l^n))\to I_n^a,\q\E(|\tr(H_l^n)|^2)\to|I_n^a|^2
$$
so
$$
\E(|\tr(H_l^n)-I_n^a|^2)=\E(|\tr(H_l^n)|^2)-2\E(\tr(H_l^n))I_n^a+|I_n^a|^2\to0
$$
as required.

The following estimates hold for $|w|\le1/2$
$$
\left|\Log(1+w)\right|\le2|w|,\q 
\left|\Log(1+w)-w\right|\le|w|^2.
$$
We apply these estimates with $w=n/(N(z-x))$, for $N\ge2n$ and for points $z$ on the contour $\g_{R\vee\L^*}$ and $x$ in the support of $\mu_\L$, to obtain
$$
|G_\L^{N/n}(z)|\le2,\q |G_\L^{N/n}(z)-G_\L(z)|\le n/N
$$
where
$$
G_\L(z)=\int_\R\frac{\mu_\L(dx)}{z-x}.
$$
Note that $\g_R$ has length $4R+2\pi$.
By some straightforward estimation, on $\O_R^c$,
$$
|I_n^a(\L)|\le\frac1{2\pi n}(4\L^*+2\pi)e^{nT(\L^*+1)+2n}
$$
while, on $\O_R$, 
$$
|I_n^a(\L)|\le\frac1{2\pi n}(4R+2\pi)e^{nT(R+1)+2n}.
$$
Then, by the estimate \eqref{ENT}, uniformly in $a\in(0,T)$,
$$
\E(|I_n^a(\L)|1_{\O_R^c})\to0,\q \E(|I_n^a(\L)I_n^b(\L)|1_{\O_R^c})\to0
$$
while, by the weak limit \eqref{WPR}, also uniformly in $a\in(0,T)$,
\begin{align*}
I_n^a(\L)1_{\O_R}
&=1_{\O_R}\frac{e^{-an^2/(2N)}}{2\pi in}\int_{\g_R}\exp\{-n(az-G_\L^{N/n}(z))\}dz\\
&\q\q\to\frac1{2\pi in}\int_{\g_R}\exp\{-n(az-G_T(z))\}dz=I_n^a.
\end{align*}
in probability, and so
$$
\E(I_n^a(\L)1_{\O_R})\to I_n^a,\q \E(I_n^a(\L)I^b_n(\L)1_{\O_R})\to I_n^aI_n^b.
$$
The desired limits \eqref{EIN} now follow.

\section{Makeenko--Migdal equations}\label{MM}
Our aim in this section is to prove Proposition \ref{REG}.
For this, our main tool will be the the Makeenko--Migdal equations.
In order to formulate these precisely, we first give a description of the set of regular loops modulo area-preserving homeomorphisms of $\Sbb_T$. 
This allows to reduce our analysis to a series of finite-dimensional simplices, each representing the possible vectors of face-areas for a given combinatorial graph.
We show that the Makeenko--Migdal equations allow us to move area between faces of a regular loop provided only that the total area and the total winding number are conserved.
This finally allows an inductive scheme to bootstrap the convergence we have shown for simple loops to all regular loops.

%TAKING OUT THIS SUBSECTION%%%%%%%%%%%%%%%%%%%%%
\def\j{
\subsection{Wilson loops}
 and let $l\in L(\Sbb_T)$.
The random variable
$$  
W_l= \tr(H_l)
$$
is called a {\em Wilson loop variable}. 
We will analyse functions of the form
$$
w(l_1,\dots,l_n)=\E(W_{l_1}\dots W_{l_n})
$$
for loops $l_1,\dots,l_n\in L(\Sbb_T)$ using a system of differential equations.
In this, we exploit the fact that the Yang--Mills measure is invariant under area-preserving homeomorphisms to reduce to finite-dimensional parameter spaces which encode the areas of faces.
To make this precise, it is convenient to work with {\em labelled embedded graphs} in $\Sbb_T$, that is to say, sequences $\Gbb=(e_1,\dots,e_m)$ in $P(\Sbb_T)$ such that $\{e_1,\dots,e_m\}$ is an embedded graph in $\Sbb_T$.
}
%%%%%%%%%%%%%%%%%%%%%%%%%%%%%%%%%%%%%%%%%%%%

\subsection{Combinatorial planar graphs and loops}
Given two labelled embedded graphs $\Gbb=(e_1,\dots,e_m)$ and $\Gbb'=(e_1',\dots,e_m')$, 
let us write $\Gbb\sim\Gbb'$ if there is an orientation-preserving homeomorphism $\th$ of $\Sbb_T$ such that $e_j'=\th\circ e_j$ for all $j$.
Further, let us write $\Gbb\approx\Gbb'$ if $\th$ may be chosen to be area-preserving.
Then $\sim$ and $\approx$ are equivalence relations on the set of labelled embedded graphs.
We will call the equivalence class of $\Gbb$ under $\sim$ the {\em combinatorial graph} associated to $\Gbb$.

We define a {\em standard labelling} of the vertices and faces of $\Gbb$ as follows.
Consider the sequence of vertices $(\underline e_1,\overline e_1,\dots,\underline e_m,\overline e_m)$ and 
write $V=(v_1,\dots,v_q)$ for the subsequence obtained by dropping any vertex which has already appeared.
Similarly consider the sequence of faces $(l(e_1),r(e_1),\dots,l(e_m),r(e_m))$, where $l(e_j)$ and $r(e_j)$ are the connected components of $\Sbb_T\sm\{e_1^*,\dots,e_m^*\}$ to the left and right of $e_j$.
Then write $F=(f_1,\dots,f_p)$ for the subsequence obtained by dropping any face which has already appeared.
Set
$$
\cV=\{1,\dots,q\},\q \cE=\{1,\dots,m\},\q \cF=\{1,\dots,p\}.
$$
The combinatorial graph associated to $\Gbb$ is then characterized\footnote{To see this, given $\Gbb'$ with the same combinatorial data, we can first define homeomorphisms $e_j^*\to{e_j'}^*$ by parametrization at constant speed, 
then extend the resulting homeomorphisms of face-boundaries to homeomorphisms of closed faces to obtain a homeomorphism of $\Sbb_T$.} by the integers $q,m,p$ and the functions $s,t:\cE\to\cV$ and $l,r:\cE\to\cF$ given by
\begin{itemize}
\item[(a)]
$s(j)=i$ if $v_i$ is the starting point of $e_j$, 
\item[(b)]
$t(j)=i$ if $v_i$ is the terminal point of $e_j$, 
\item[(c)]
$l(j)=k$ if $f_k$ is the face to the left of $e_j$,
\item[(d)]
$r(j)=k$ if $f_k$ is the face to the right of $e_j$.
\end{itemize}
We call any quadruple $\cG=(s,t,l,r)$ which arises in this way a {\em combinatorial planar graph}.
We freely identify $\cG$ with the corresponding equivalence class of labelled embedded graphs.

Given a combinatorial planar graph $\cG$, consider the simplex
$$
\Delta_\cG(T)=\{(a_1,\dots,a_p):a_k>0\text{ for all }k\text{ and }a_1+\dots+a_p=T\}.
$$
Given a labelled embedded graph $\Gbb\in\cG$, define the {\em face-area vector} $a(\Gbb)=(a_1,\dots,a_p)$ by
$$
a_k=\area(f_k).
$$
Then $a(\Gbb)\in\Delta_\cG(T)$.
For $a\in\Delta_\cG(T)$, set
$$
\cG(a)=\{\Gbb\in\cG:a(\Gbb)=a\}.
$$
The sets $\cG(a)$ are then the equivalence classes of the relation $\approx$.
There is a universal constant $C<\infty$ such that, for all $l\in\cG(a)$ and $l'\in\cG(a')$, 
\begin{equation}\label{MC}
\sum_{k=1}^p|a_k-a_k'|\le C(\ell(l)+\ell(l'))d(l,l')
\end{equation}
where $d$ is the length metric \eqref{metricP}.

We call a sequence $\mfl_0=((j_1,\ve_1),\dots,(j_r,\ve_r))$ in $\cE\times\{-1,1\}$ a {\em loop in $\cG$} if
\begin{equation}\label{COMBL}
t(j_k,\ve_k)=s(j_{k+1},\ve_{k+1})
\end{equation}
for $k=1,\dots,r$, where $j_{r+1}=j_1$ and $\ve_{r+1}=\ve_1$ and where
$$
s(\ve,j)=t(-\ve,j)=\begin{cases}s(j),&\text{if $\ve=1$},\\t(j),&\text{if $\ve=-1$}.\end{cases}
$$
The condition \eqref{COMBL} means that, in any labelled embedded graph $\Gbb=(e_1,\dots,e_m)\in\cG$, 
we can concatenate the sequence of edges $(e_{j_1}^{\ve_1},e_{j_2}^{\ve_2},\dots,e_{j_r}^{\ve_r})$ to form a loop 
$$
l_0=e_{j_1}^{\ve_1}e_{j_2}^{\ve_2}\dots e_{j_r}^{\ve_r}.
$$
Then we call $l_0$ the drawing of $\mfl_0$ in $\Gbb$.
Note that the sequence
$$
\mfl^{-1}=((j_r,-\ve_r),\dots,(j_1,-\ve_1))
$$
is then also a loop in $\cG$, whose drawing in $\Gbb$ is the reversal $l^{-1}$ of $l$.
Note also the obvious notion of concatenation for loops in $\cG$.

In the case of interest to us, $\cG$ will be the combinatorial graph of the labelled embedded graph $\Gbb=(e_1,\dots,e_m)$ of a regular loop $l$.
Then, if $l$ has $n$ self-intersections, we have $q=n+1$, $m=2n+1$ and, by Euler's relation, $p=n+2$.
Note that the set of self-intersections is given in the standard labelling by $\{v_i:i\in\cI\}$, where $\cI=\{2,3,\dots,n+1\}$.
We recover $l$ as the drawing in $\Gbb$ of the loop 
$$
\mfl=((1,1),\dots,(2n+1,1))
$$
in $\cG$.
We call the pair $(\cG,\mfl)$ a {\em combinatorial planar loop}.
For each $n\ge0$, there are only finitely many combinatorial loops with $n$ self-intersections.
We will write abusively $\mfl$ for $(\cG,\mfl)$, $\Delta_\mfl(T)$ for $\Delta_\cG(T)$ and $\mfl(a)$ for $\cG(a)$.
Given a loop $\mfl_0$ in $\cG$, it may be that the drawing $l_0$ of $\mfl_0$ in $\Gbb$ is a regular loop.
We could then consider the combinatorial loop associated to $l_0$, without
reference to its relation to $l$. 
We will therefore need to make clear when such a combinatorial loop is to
be considered in the context of a larger combinatorial graph.

\subsection{Generalized Makeenko--Migdal equations} 
Let $\mfl$ be a combinatorial planar loop.
Write $m$ and $p$ for the numbers of edges and faces in the associated combinatorial graph.
Let $H=(H_\g:\g\in P(\Sbb_T))$ be a Yang--Mills holonomy field in $U(N)$. 

\begin{prop}\label{DTC}
Let $f:U(N)^m\to\C$ be a continuous bounded function.
Then there is a uniformly continuous bounded function $E(f):\Delta_\mfl(T)\to\C$ such that
$$
E(f)(a)=\E(f(H_{e_1},\dots,H_{e_m}))
$$
for $a\in\Delta_\mfl(T)$, whenever $\Gbb=(e_1,\dots,e_m)$ is a labelled embedded graph with $\Gbb\in\mfl(a)$.
\end{prop}
\begin{proof}
Write $\overline{\Delta_\mfl(T)}$ for the closure of $\Delta_\mfl(T)$ in $\R^p$.
There is a sequence of continuous maps
$$
p_j:\overline{\Delta_\mfl(T)}\to P(\Sbb_T),\q j=1,\dots,m
$$
such that, for all $a\in\Delta_\mfl(T)$, the endpoints of the paths $p_1(a),\dots,p_m(a)$ do not depend on $a$
and we may concatenate these paths to form a regular loop $l(a)$ with $\Gbb_{l(a)}\in\mfl(a)$.
Define $E(f):\overline{\Delta_\mfl(T)}\to\C$ by
$$
E(f)(a)=\E(f(H_{p_1(a)},\dots,H_{p_m(a)})).
$$
Since $H$ is continuous in probability for convergence in length with fixed endpoints, 
we see by bounded convergence that $E(f)$ is continuous on $\overline{\Delta_\mfl(T)}$, and hence uniformly continuous.
On the other hand, for all $a\in\Delta_\mfl(T)$ and any embedded graph $\Gbb=(e_1,\dots,e_m)\in\mfl(a)$, we see from \eqref{Discrete Yang Mills} that 
$$
\E(f(H_{e_1},\dots,H_{e_m}))=\E(f(H_{p_1(a)},\dots,H_{p_m(a)})).
$$
\end{proof}
For $i\in\{1,\dots,m\}$ and $g\in U(N)$, define maps $R_{i,g}$ and $\hat R_{i,g}$ on $U(N)^m$ by
\begin{align*}
R_{i,g}(h_1,\dots,h_m)&=(h_1,\dots,h_ig,\dots,h_m),\\
\hat R_{i,g}(h_1,\dots,h_m)&=(h_1,\dots,g^{-1}h_i,\dots,h_m).
\end{align*}
A function $f:U(N)^m\to\C$ is said to have {\em extended gauge invariance} if, 
for all $g\in U(N)$ and for $i=1,\dots,m-1$, 
$$
f\circ\hat R_{i,g}\circ R_{i+1,g}=f.
$$
Thus we require
$$
f(h_1,\dots,g^{-1}h_i,h_{i+1}g,\dots,h_m)=f(h_1,\dots,h_i,h_{i+1},\dots,h_m).
$$
For $i\in\{1,\dots,m\}$ and $X\in\mfu(N)$, define a differential operator $\cL^i_X$ on $U(N)^m$ by 
$$
\cL^i_X(f)=\left.\frac{d}{dt}\right|_{t=0}f\circ R_{i,e^{tX}}.
$$
Choose an orthonormal basis $(X_n:n\in\N)$ for $\mfu(N)$ (with inner product \eqref{metric}) and, for $i,j\in\{1,\dots,m\}$, define
$$
\Delta_{i,j}(f)=\sum_n\cL^i_{X_n}\circ\cL^j_{X_n}(f).
$$
The operator $\Delta_{i,j}$ does not depend on the choice of orthonormal basis.  

Write $\cI$ of the set of intersection labels and $\cF$ for the set of face labels in the combinatorial graph $\cG$ of $\mfl$, as usual.
For $i\in\cI$, define a (constant) vector field $\mu_i$ on $\Delta_\mfl(T)$ as follows.
Choose $\Gbb\in\cG$ and write $l$ for the drawing of $\mfl$ in $\Gbb$.
In the standard labelling of $\Gbb$, the vertex $v_i$ is a self-intersection of $l$, 
so there is a unique sequence $(k_1,k_2,k_3,k_4)$ in $\cF$ such that $(f_{k_1},f_{k_2},f_{k_3},f_{k_4})$ is an anti-clockwise circuit of the faces of $\Gbb$ around $v_i$, 
starting from the unique face $f_{k_1}$ adjacent to both outgoing strands of $l$.  
This sequence does not depend on the choice of $\Gbb$.
Set
\begin{equation}\label{MMCP}
\mu_i=\pd_{k_1}-\pd_{k_2}+\pd_{k_3}-\pd_{k_4}
\end{equation}
where $\pd_k$ denotes the elementary vector field in direction $k$.

The following theorem is a specialization of a result of Driver, Gabriel, Hall and Kemp \cite[Theorem 2]{MR3631396}, which generalizes a formulation of L\'evy \cite{MR3636410}.

\begin{thm}\label{GMM}
Let $f:U(N)^m\to\C$ be a smooth function having extended gauge invariance.
Then, for all $i\in\cI$, the function $E(f)$ has directional derivative on $\Delta_\mfl(T)$ in direction $\mu_i$ given by 
$$
\mu_iE(f)=-E(\Delta_{j_1,j_2}(f))
$$
where $j_1,j_2$ are determined by $s(j_1)=s(j_2)=i$.
\end{thm}

\subsection{Makeenko--Migdal equations for Wilson loops} 
Given a loop $\mfl_0=((j_1,\ve_1),\dots,(j_r,\ve_r))$ in $\cG$, we can define a continuous bounded function $W_{\mfl_0}:U(N)^m\to\C$ by
$$
W_{\mfl_0}(h_1,\dots,h_m)=\tr(h_{j_r}^{\ve_r}\dots h_{j_1}^{\ve_1}).
$$
Given a sequence of loops $(\mfl_1,\dots,\mfl_k)$ in $\cG$, define the {\em Wilson loop function} 
$$
\phi^N_{\mfl_1,\dots,\mfl_k}:\Delta_\mfl(T)\to\C
$$ 
by
$$
\phi^N_{\mfl_1,\dots,\mfl_k}=E(W_{\mfl_1}\dots W_{\mfl_k}).
$$
Then $\phi^N_{\mfl_1,\dots,\mfl_k}$ is uniformly continuous and, for all $a\in\Delta_\mfl(T)$ and all $\Gbb\in\mfl(a)$, 
\begin{equation}\label{WILH}
\phi^N_{\mfl_1,\dots,\mfl_k}(a)=\E(\tr(H_{l_1})\dots\tr(H_{l_k}))
\end{equation}
where $l_1,\dots,l_k$ are the drawings of $\mfl_1,\dots,\mfl_k$ in $\Gbb$.
We will write $\phi^N_{\mfl_1,\dots,\mfl_k}$ also for the continuous extension
to $\overline{\Delta_\mfl(T)}$.

For $i\in\cI$, we obtain two regular loops $l_i$ and $\hat l_i$ by splitting $l$ at $v_i$, that is, by following the two outgoing strands of $l$ from $v_i$ until their first return to $v_i$.
In one case we will pass through the endpoint of $l$ and begin another circuit of $l$ until we reach $v_i$.
Write $\mfl_i$ and $\hat\mfl_i$ for the loops in $\cG$ whose drawings in $\Gbb$ are $l_i$ and $\hat l_i$, which do not depend on the choice of $\Gbb$.
Then set
$$
[\mfl]_i=\mfl_i\hat\mfl_i\mfl_i^{-1}\hat\mfl_i^{-1},\q
[\hat\mfl]_i=\hat\mfl_i\mfl_i\hat\mfl_i^{-1}\mfl_i^{-1}
$$
where $\mfl_i^{-1},\hat\mfl_i^{-1}$ denote the reversals of $\mfl_i,\hat\mfl_i$ and the right-hand sides are understood as concatenations.

\begin{prop}[Makeenko--Migdal equations for Wilson loops]
\label{MM equations} 
The functions $\phi^N_\mfl$ and $\phi^N_{\mfl,\mfl^{-1}}$ have directional derivatives in $\Delta_\mfl(T)$ in direction $\mu_i$ given by
$$
\mu_i\phi^N_\mfl=\phi^N_{\mfl_i,\hat\mfl_i},\q
\mu_i\phi^N_{\mfl,\mfl^{-1}}=\phi^N_{\mfl_i,\hat\mfl_i,\mfl^{-1}}+\phi^N_{\mfl,\mfl_i^{-1},\hat\mfl_i^{-1}}-N^{-2}(\phi^N_{[\mfl]_i}+\phi^N_{[\hat\mfl]_i}).
$$
\end{prop}
\begin{proof}
We give details only for $\phi^N_{\mfl,\mfl^{-1}}$.
The simpler argument for $\phi^N_\mfl$ will then be obvious.
The argument for $\phi^N_\mfl$ already appeared after Theorem 2.6 in \cite{MR3613519} and in Section 9.2 of \cite{MR3636410}.
Given $\Gbb=(e_1,\dots,e_m)\in\cG$, set $l=e_1\dots e_m$, so $l$ is the drawing of $\mfl$ in $\Gbb$.
Given $h=(h_1,\dots,h_m)\in U(N)^m$, there is a unique multiplicative function $(h_\g:\g\in P(\Gbb))\in\cM(P(\Gbb),U(N))$ such that $h_{e_j}=h_j$ for all $j$.
Then $\phi^N_{\mfl,\mfl^{-1}}=E(f)$, where $f=|W_\mfl|^2$ and
$$
W_\mfl(h_1,\dots,h_m)=\tr(h_l)=\tr(h_m\dots h_1).
$$
Note that $W_\mfl$ has extended gauge invariance and so also does $f$.
We can write $l_i=e\g$ and $\hat l_i=\hat e\hat\g$, where $e=e_{j_1},\hat e=e_{j_2}$, $s(j_1)=s(j_2)=i$ and $\g,\hat\g\in P(\Gbb)$.
Then
$$
f(h)=\tr(h_l)\tr(h_l^{-1})=\tr(h_{\hat\g}h_{\hat e}h_\g h_e)\tr(h_e^{-1}h_\g^{-1}h_{\hat e}^{-1}h_{\hat\g}^{-1}).
$$
For $X\in\mfu(N)$,
\begin{align*}
\cL_X^{j_1}\circ\cL_X^{j_2}(f)(h)&=\tr(h_{\hat\g}h_{\hat e}Xh_\g h_eX)\tr(h_l^{-1})+\tr(h_l)\tr(Xh_e^{-1}h_\g^{-1}Xh_{\hat e}^{-1}h_{\hat\g}^{-1})\\
&\q\q\q\q-\tr(h_{\hat\g}h_{\hat e}h_\g h_eX)\tr(h_e^{-1}h_\g^{-1}X h_{\hat e}^{-1}h_{\hat\g}^{-1})\\
&\q\q\q\q-\tr(h_{\hat\g}h_{\hat e}Xh_\g h_e)\tr(Xh_e^{-1}h_\g^{-1}h_{\hat e}^{-1}h_{\hat\g}^{-1}).
\end{align*}
Write $E_{j,k}$ for the elementary matrix with a $1$ in the $(j,k)$-entry. 
Set
$$
X_{j,j}=iE_{j,j}/{\sqrt N},\q
X_{j,k}=\begin{cases}
(E_{j,k}-E_{k,j})/{\sqrt{2N}},&\text{ for $j<k$},\\
i(E_{j,k}+E_{k,j})/{\sqrt{2N}},&\text{ for $j>k$}.
\end{cases}
$$
Then $\{X_{j,k}:j,k=1,\dots,N\}$ is an orthonormal basis in $\mfu(N)$. 
A simple calculation gives the standard identity
$$ 
\sum_{j,k=1}^NX_{j,k}\otimes X_{j,k}=-\frac1N\sum_{j,k=1}^NE_{j,k}\otimes E_{k,j}.
$$
We sum to obtain
\begin{align*}
-\Delta_{j_1,j_2}(f)(h)&=\tr(h_{\hat\g}h_{\hat e})\tr(h_\g h_e)\tr(h_l^{-1})+\tr(h_l)\tr(h_e^{-1}h_\g^{-1})\tr(h_{\hat e}^{-1}h_{\hat\g}^{-1})\\
&\q\q\q\q-\frac1{N^2}\tr(h_{\hat\g}h_{\hat e}h_\g h_eh_{\hat e}^{-1}h_{\hat\g}^{-1}h_e^{-1}h_\g^{-1})\\
&\q\q\q\q-\frac1{N^2}\tr(h_\g h_eh_{\hat\g}h_{\hat e}h_e^{-1}h_\g^{-1}h_{\hat e}^{-1}h_{\hat\g}^{-1})
\end{align*}
and hence, by Theorem \ref{GMM},
$$
\mu_i\phi^N_{\mfl,\mfl^{-1}}=-E(\Delta_{j_1,j_2}(f))=\phi^N_{\mfl_i,\hat\mfl_i,\mfl^{-1}}+\phi^N_{\mfl,\mfl_i^{-1},\hat\mfl_i^{-1}}-N^{-2}(\phi^N_{[\mfl]_i}+\phi^N_{[\hat\mfl]_i}).
$$
\end{proof}

%%%%%%%%%%%%%%%%%%%%%%%%%%%%%%%%%%%%%%%%%%%%%%%%%%%%%%%%%%
\def\j{
Remark that $f(h)=\tr(h_l)\tr(h_{\hat l})$ and denote by $\a$ and $\b$ the two paths such that $l=e_1\a e_4^{-1}e_2\b e_3^{-1}$.  
For any $X\in\mfg$,
$$
\begin{aligned}
\cL_X^{e_1}\circ\cL_X^{e_2}(f)&=\tr(h_{e_3}^{-1}h_\b h_{e_2}Xh_{e_4}^{-1}h_\a h_{e_1}X)\tr(h_{\hat l})+\tr(h_l)\tr(Xh_{e_1}^{-1}h_\a^{-1}h_{e_4}Xh_{e_2}^{-1}h_\b^{-1}h_{e_3})\\
&\q\q\q\q-\tr(h_{e_3}^{-1}h_\b h_{e_2}h_{e_4}^{-1}h_\a h_{e_1}X )\tr(h_{e_1}^{-1}h_\a^{-1}h_{e_4}X h_{e_2}^{-1}h_\b^{-1}h_{e_3})\\
&\q\q\q\q-\tr(h_{e_3}^{-1}h_\b h_{e_2}Xh_{e_4}^{-1}h_\a h_{e_1})\tr(X h_{e_1}^{-1}h_\a^{-1}h_{e_4} h_{e_2}^{-1}h_\b^{-1}h_{e_3})
\end{aligned}
$$
We use the identity%
%%%%%%%%%%%%%%%%%%%%%%%%%%%%%%%%%%%%%%%%%%%%%%%%%%%%%%%%%%%%%%%%%%%%%%%%%%%%%%%
\footnote{See for example \cite[Lemma 1.2.1]{MR3636410}}
%%%%%%%%%%%%%%%%%%%%%%%%%%%%%%%%%%%%%%%%%%%%%%%%%%%%%%%%%%%%%%%%%%%%%%%%%%%%%%%
$$ 
\sum_nX_n\otimes X_n=-\frac1N\sum_{i,j=1}^NE_{i,j}\otimes E_{j,i}
$$
where $E_{i,j}$ is the elementary matrix with a $1$ in the $(i,j)$-entry, to obtain
\begin{align*}
-\Delta_{e_1,e_2}(f)&=\tr(h_{e_3}^{-1}h_\b h_{e_2})\tr(h_{e_4}^{-1}h_\a h_{e_1})\tr(h_{\hat l})+\tr(h_l)\tr(h_{e_1}^{-1}h_\a^{-1}h_{e_4})\tr(h_{e_2}^{-1}h_\b^{-1}h_{e_3})\\
&-\frac1{N^2}\tr(h_{e_3}^{-1}h_\b h_{e_2}h_{e_4}^{-1}h_\a h_{e_1}h_{e_2}^{-1}h_\b^{-1}h_{e_3}h_{e_1}^{-1}h_\a^{-1}h_{e_4})\\
&-\frac1{N^2}\tr(h_{e_4}^{-1}h_\a h_{e_1}h_{e_3}^{-1}h_\b h_{e_2}h_{e_1}^{-1}h_\a^{-1}h_{e_4}h_{e_2}^{-1}h_\b^{-1}h_{e_3}).
\end{align*}
To conclude, note that $\mfl_i$ and $\mfl_i'$ are given by $e_2\b e_3^{-1}$ and $e_1\a e_4^{-1}$. 
}
%%%%%%%%%%%%%%%%%%%%%%%%%%%%%%%%%%%%%%%%%%%%%%%%%%%%%%%%%%

\subsection{Makeenko--Migdal vectors and the winding number\label{MM section}}
Let $l\in L(\Sbb_T)$ be a regular loop and let $\Gbb=(V,E,F)$ be the associated labelled embedded graph.
The winding number of $l$ is a function 
$$
n_l:F\to\Z
$$
defined up to an additive constant, which may be computed as follows.
Fix a reference face $f_0\in F$. 
For each face $f\in F$, there is a non-negative integer $k$ and a {\em track} from $f_0$ to $f$, comprising edges $e_1,\dots,e_k$ and faces $f_1,\dots,f_k$ such that $f_k=f$ and $e_j$ is adjacent to both $f_{j-1}$ and $f_j$ for all $j$.
(The notation here does not refer to the standard labelling of $\Gbb$.)
Set
$$
n_l(f)=L(f)-R(f)
$$
where $L(f)$ and $R(f)$ are the numbers of edges $e_j$ with $f_j$ on the left and right respectively.
Then $L(f)$ and $R(f)$ are well-defined functions of the track, and $n_l(f)$ does not depend on the choice of track.
Moreover, the function $n_l$ depends on the choice of reference face only by the addition of a constant.
The winding number is invariant under orientation-preserving homeomorphisms of $\Sbb_T$, so we obtain also a function
$$
n_\mfl:\cF\to\Z
$$
determined by the associated combinatorial loop $\mfl$, also defined up to an additive constant, by setting 
$$
n_\mfl(k)=n_l(f)
$$ 
where $f$ is the $k$th face in the standard labelling of $\Gbb$.

The following lemma is a reformulation of a lemma of L\'evy \cite[Lemma 6.28]{MR3636410}. 
See also Dahlqvist \cite[Lemma 21]{MR3554890}. 
We give a slightly different proof, relying on properties of the winding number in place of a dimension-counting argument.
The prior results were stated for the whole plane, while ours applied to the sphere, but this make little difference to the argument.

\begin{lem}
There is an orthogonal direct sum decomposition
$$
\R^\cF=\mfm_\mfl\oplus\mfn_\mfl
$$
where
$$
\mfm_\mfl=\spann\{\mu_i:i\in\cI\},\q \mfn_\mfl=\spann\{1,n_\mfl\}.
$$
\end{lem}
\begin{proof}
Note first that $1^T\mu_i=1-1+1-1=0$ for all $i$.
Let $i\in\cI$.
Write $k_1,k_2,k_3,k_4$ for the faces at $i$, listed anticlockwise starting from the face $k_1$ adjacent to both outgoing edges.
Then the values of $n_\mfl$ at $k_1,k_2,k_3,k_4$ are given respectively by $n,n+1,n,n-1$ for some $n$, so
$$
n_\mfl^T\mu_i=n_\mfl(k_1)-n_\mfl(k_2)+n_\mfl(k_3)-n_\mfl(k_4)=0.
$$
Hence, if $\a\in\mfm_\mfl$, then $1^T\a=0$ and $n_\mfl^T\a=0$.

Suppose on the other hand that $\a\in\mfm_\mfl^\perp$. 
Consider the $1$-forms (of the dual graph) $d\a$ and $d\nu$, given by
$$
d\a(j)=\a(l(j))-\a(r(j)),\q dn_\mfl(j)=n_\mfl(l(j))-n_\mfl(r(j)),\q j\in\cE.
$$
Then $dn_\mfl(j)=1$ for all $j$.
On the other hand, for $j=1,\dots,m-1$, there is an $i_j\in\cI$ such that $t(j)=i_j=s(j+1)$, so
$$
d\a(j)-d\a(j+1)=\pm\mu_{i_j}^T\a=0.
$$
Hence $d\a=c_1dn_\mfl$ and so $\a=c_1n_\mfl+c_2$ for some constants $c_1,c_2$.
\end{proof}

Note that $\Delta_\mfl(T)$ is convex, and that, by counting dimensions, the vectors $\{\mu_i:i\in\cI\}$ are linearly independent.
We deduce from these facts, and the preceding lemma the following proposition.

\begin{prop}\label{DELTA}
Let $a\in\Delta_\mfl(T)$ and $a'\in\overline{\Delta_\mfl(T)}$.
Set $v=a'-a$.
Then $a+tv\in\Delta_\mfl(T)$ for all $t\in[0,1)$.
Moreover, there exists $\a\in\R^\cI$ such that
$$
v=\sum_{i\in\cI}\a_i\mu_i
$$
if and only if 
$$
\sum_{k\in\cF}a_kn_\mfl(k)=\sum_{k\in\cF}a'_kn_\mfl(k).
$$
Moreover, in this case, $\a$ is uniquely determined by $v$ and
\begin{equation}
\sum_{i\in\cI}|\a_i|\le C_\mfl T
\end{equation}
for some constant $C_\mfl<\infty$ depending only on $\mfl$.
\end{prop}

\def\j{
Set
$$
n_*=\max\{n_\mfl(k)-n_\mfl(k'):k,k'\in\cF\}.
$$
Choose $k_0,k_*\in\cF$ so that $n_\mfl(k_*)-n_\mfl(k_0)=n_*$ 
and then choose the additive constant for the winding number so that $n_\mfl(k_0)=0$ and $n_\mfl(k_*)=n_*$.
For $k\in\cF$, set $p_k=n_\mfl(k)/n_*$.
Then, for $k\not=k_0,k_*$, define a vector $v_k\in\R^\cF$ by
$$
v_k(k')=
\begin{cases}
1-p_k,&\text{if $k'=k_0$,}\\
p_k,&\text{if $k'=k_*$,}\\
-1,&\text{if $k'=k$,}\\
0,&\text{otherwise.}
\end{cases}
$$
Then $1^Tv_k=0$ and $n_\mfl^Tv_k=n_*p_k-n_\mfl(k)=0$, so $v_k\in\mfm_\mfl$.
So we can write
$$
v_k=\sum_{i\in\cI}\a_{ki}\mu_i
$$
for some $\a_{ki}\in\R$. 
By dimension counting, the vectors $\mu_i$ are linearly independent, so there is a constant $C<\infty$, 
depending only on $n$, such that 
\begin{equation}\label{ALPF}
|\a_{ki}|\le C\q\text{for all $k$ and $i$}.
\end{equation}

The following proposition will allow us to move all the area to the chosen faces of extreme winding number,
using the Makeenko--Migdal equations, while staying within the simplex $\Delta_\mfl(T)$ of positive face-areas.

\begin{prop}
\label{Reduction two faces} 
For $a\in\Delta_\mfl(T)$ and $t\in[0,1]$, set
$$
v=\sum_{k\not=k_0,k_*}a_kv_k,\q a(t)=a+tv.
$$
Then there exist $\a_i\in\R$ such that
$$
v=\sum_{i\in\cI}\a_i\mu_i.
$$
Moreover, $a(t)\in\Delta_\mfl(T)$ for all $t\in[0,1)$, while
$a_k(1)=0$ for all $k\not\in\{k_0,k_*\}$ \bb and
$$
a_{k_0}(1)=\sum_{k\in\cF}(1-p_k)a_k\ge a_{k_0},\q
a_{k_*}(1)=\sum_{k\in\cF}p_ka_k\ge a_{k_*}.
$$
Morever, there is a constant $C<\infty$, depending only on $n$, such that 
\begin{equation}\label{ALPE}
\sum_{i\in\cI}|\a_i|\le C(T-a_{k_0}-a_{k_*}).
\end{equation}
UNCLEAR WE NEED MORE THAN $|\a|\le C(T,\mfl)$.
\eb
\end{prop}
\begin{proof}
It is straightforward to check that we can take
$$
\a_i=\sum_{k\not=k_0,k_*}a_k\a_{ki}.
$$
\bb
Then \eqref{ALPF} implies
$$
|\a_i|\le C\sum_{k\not=k_0,k_*}a_k=C(T-a_{k_0}-a_{k_*})
$$
for all $i$, which in turn implies \eqref{ALPE}.
\eb
\end{proof}

\bb [A:   the following is needed to consider the limit of the master field on the sphere to  the one on the plane, as one face area goes to infinity, sketched in Proposition \ref{Convergence Sphere Plane}. To be checked]   The last Lemma of this section completes the previous Proposition in the case of the plane. It further shows that up to a "Makeenko-Migdal" move,  one can fix an arbitrary face $k_\infty$ to belong 
to 
$\{k_0,k_*\},$ where $k_0,k_*$ are some  faces labels satisfying  $n_\mfl(k_*)-n_\mfl(k_0)=n.$

\begin{lem} \label{lem:MM fixing one face}  Let  us fix $k_\infty \in \mathcal{F}$ and $n_\mfl $ a winding number function such that $n_\mfl(k_\infty)=0.$ For any $a\in\Delta_\cG(T), $ there is $a'\in \Delta_\cG(T),$ such that 
$ a'. n_\mfl  $ has a constant sign and $a'-a \in \mfm_\mfl. $
\end{lem}

\begin{proof}   Let us set $r(a)=\min \{ \# \{k\in \mathcal{F}: a(k) n_\mfl(k)>0\},\# \{k\in \mathcal{F}: a(k)n_\mfl(k)<0\} \}.$ If $r=0,$ $a$ satisfies the conclusion. If $r(a)>0,$ there is $k_+,k_-\in \mathcal{F},$ with $n_{\mfl} (k_+)>0,n_{\mfl}(k_-)<0$  and $a(k_+)a(k_-)>0.$  Then, setting $v(k_-)= -n_{\mfl}(k_+),v(k_+)= n_{\mfl}(k_-), v(k_\infty)=n_{\mfl}(k_+)-n_\mfl(k_-) $ and $n_\mfl(k)=0$ for all $k\in \mathcal{F}\setminus \{k_\infty,k_-,k_+\},$ defines an element of $\mfm_\mfl.$   Moreover, denoting $s= \min\{ - \frac{a_+}{n(k_-)}, \frac{a_-}{n(k_+)}\},$ $a'= a + s v\in \Delta_\cG(T)$  with $r(a')<r(a).$ The conclusion follows by induction.  
\end{proof}
\eb
%%%%%%%%%%%%%%%%%%%%%%%%%%%%%%%%%%%%%%%%%%%%%%%%%%COMMENT OUT FOR NOW
\def\jj{

Consider a loop $l\in L(\Gbb,a)$ with $a\in\Delta_\Gbb(T)$ supported on two faces $F_0,F_\infty\in\Fbb$ with $a_{F_0},a_{F_\infty}>0$.
Then $l^*$ is the boundary of a simply connected domain and this boundary has finite length.
So there exists $n\in\N$ and a simple loop $\g\in L(\Sbb_T)$ such that $\Sbb_T\sm\g^*$ has components of area $a_{F_0}$ and $a_{F_\infty}$ and $\g^n$ is a reduction of $l$.
To conclude this section, we shall use another 'move' to transform such a loop $l$ into a simple loop. 

For $n\in\N$, denote by $\mathrm{wl}_n$ the combinatorial loop winding $n$ times around a fixed point (see figure \ref{Max Winding}), 
such that the dual graph of $\Gbb_{\mathrm{wl}_n}$ is a segment.  
Let $F'_0,F'_\infty\in\Gbb_{\mathrm{wl}_n}$ be as in (\ref{max winding}).  
Then $F'_0$ and $F'_\infty)$ are the endpoints of the segment $\hat\Gbb_{\mathrm{wl}_n}$.   
A  loop $l\in L(\Sbb_T)$, with $\Gbb^l\in\Gbb_{\mathrm{wl}_n}(b)$ for some $n\in\N$ and $b\in\Delta_{\Gbb_{\mathrm{wl}_n}}(T)$ is called {\em maximally winding}. 
With this notation, any $l\in L(\Gbb,a)$, parametrized by a smooth curve, with $a$ as above, belongs also to $L(\Gbb_{\mathrm{wl}_n},a')$, 
where $n=n_{F_\infty,l}(F_0)$, $a'(F'_0)=a(F_0)$, $a'(F'_\infty)=a(F_\infty)$ and $a'(F)=0$, if $F\in\Fbb_{\mathrm{wl}_n}\sm\{F_0,F_\infty\}.$ 

\begin{prop}
\label{onions pealing} 
Let $n\in \Z$ with $n\ge1$ and let $a\in\Delta_{\mathrm{wl}_{n+1}}(T)$.
Then there exists $v\in\mfm_{\mathrm{wl}_{n+1}}$ such that $a+tv\in\Delta_{\mathrm{wl}_{n+1}}(T)$ for all $t\in[0,1]$ and $L(\Gbb_{\mathrm{wl}_{n+1}},a+v)\subset L(\Gbb_{\mathrm{wl}_n},a')$, for some $a'\in\Delta_{\mathrm{wl}_{n}}(T)$.
\end{prop}
\begin{proof} 
Let $v_1\in\Vbb_{\mathrm{wl}_{n+1}}$ be the vertex belonging to the boundary of $F_0$ and $v_2,\dots,v_n$ be the other vertices of $\Gbb_{\mathrm{wl}_{n+1}}$,
ordered by their time of first visit by $\mathrm{wl}_n$ starting from $v_1$. 
Then $\tilde{v}=\mu_{v_1}+\dots+\mu_{v_n}\in\mfm_{\mathrm{wl}_{n+1}}$ is such that
$$  
\tilde{v}(F)=
\left\{
\begin{array}{ll}
-1&,\text{if }F\in\{F_0,F_\infty\},\\
\;\;\; 1&   ,\text{if }\overline{F}\cap(\overline{F}_0\cup\overline{F}_\infty)\not=\emptyset,\\
\;\;\; 0& ,\text{otherwise}.
\end{array}
\right.
$$
Then $v=\min(a(F_0),a(F_\infty))\tilde{v}$ has the claimed properties.
\end{proof}

Given a combinatorial loop $\mfl$, consider the {\em double embedding} $\tilde\Gbb_\mfl$ obtained as follows.
Embed $\mfl$ in $\Sbb_T$ and replace each embedded vertex $v$ with four vertices $v_1,v_2,v_3,v_4$ lying in the four corners around $v$ in a cyclic order, 
and add four edges forming a square around $v$.
Replace each embedded edge $(v,v')$ by two edges joining the two pairs of new vertices lying in the two faces neighboring $(v,v')$ (see figure \ref{Doubling graph}).
\begin{figure}\centering 
\begin{minipage}{.5\textwidth}
  \centering
 \includegraphics[width=50 mm,height=44 mm]{Doublinggraph}
%  \captionof{figure}{A figure}
\end{minipage}%
\begin{minipage}{.5\textwidth}
  \centering
 \includegraphics[width=50 mm,height=44 mm]{DoublingMM}
%  \captionof{figure}{Another figure}
\end{minipage}
\caption{\label{Doubling graph}  On the left-hand-side, the doubling of the graph $\Gbb_\mfl$ of a combinatorial loop $\mfl$ drawn in plane lines; $\tilde\mfl^{-1}$ is drawn in dashed line. 
The right-hand-side pictures the decomposition of $\cD(\mu_v)$ used in lemma \ref{lemma DoublingMM}, for the vertex $v$ of $\Gbb_\mfl$ represented by a dot.}
\end{figure}

Then $\tilde\Gbb_\mfl$ is a $4$-regular graph containing two embeddings of $\mfl$, which we will denote by $\mfl$ and $\tilde{\mfl}$, with $\tilde{\mfl}$ to the right of $\mfl$. 
By Euler's relation, 
$$
\#\tilde{\Fbb}_{\mathfrak{l}}= 2+ 4 \#\Vbb_\mathfrak{l}= 4\# \Fbb_\mathfrak{l}-6. 
$$  
There is an  injection $\iota:\Fbb_\mfl\to\tilde\Fbb_\mfl$ that identifies faces of $\Gbb_\mfl$ with the faces of $\tilde\Gbb_\mfl$ that are not fully contracted, 
when retracting $\tilde\mfl$ to $\mfl$. 
This latter induces an injective map 
$$
\cD:\R^{\Fbb_\mfl}\to\R^{\tilde\Fbb_\mfl},
$$
such that for any $h\in\R^{\Fbb_\mfl}$, and $F'\in\tilde\Fbb_\mfl$, $\cD(h)(F')=h(F)$, if $F'=\iota(F)$ for some $F\in\Fbb_\mfl$, and $0$ otherwise. 
It satisfies $\cD(X_{\Gbb_\mfl})\subset X_{\tilde\Gbb_\mfl}$ and, for any $T>0$, $\cD(\Delta_\Gbb(T))\subset\Delta_{\tilde\Gbb}(T)$. 

\begin{lem} 
\label{lemma DoublingMM}
The map $\cD$ satisfies 
$$
\cD:\mfm_l\to\mfm_{\mfl,\tilde\mfl}\cap\mfm_{\mfl,\tilde\mfl^{-1}}.
$$
\end{lem}
\begin{proof} 
For $v\in\Vbb_\mfl$, let $(v_i)_{1\le i\le4}$ be the associated vertices of $\tilde\Gbb_\mfl$ in cyclic order, with $v_1$ being an self-intersection point of $\mfl$. 
Then (see figure \ref{Doubling graph}),  
$$  
\cD(\mu_v^\mfl)=\mu_{v_1}^{\mfl,\tilde\mfl}+\mu_{v_2}^{\mfl,\tilde\mfl}+\mu_{v_3}^{\mfl,\tilde\mfl}+\mu_{v_4}^{\mfl,\tilde\mfl}
=\mu_{v_1}^{\mfl,\tilde\mfl^{-1}}-\mu_{v_2}^{\mfl,\tilde\mfl^{-1}}+\mu_{v_3}^{\mfl,\tilde\mfl^{-1}}-\mu_{v_4}^{\mfl,\tilde\mfl^{-1}}.
$$
\end{proof}
}%%%%%%%%%%%%%%%%%%%%%%%%%%%%%%%%%%%%%%%%%%%%%%%%%%%%%%%%%%%%%%%%%%%%%%%%%%
}

\subsection{Proof of Proposition \ref{REG}}\label{PRREG}
We will show that the following statements hold for all $n\ge0$.
Firstly, {\em for all combinatorial planar loops $\mfl$ with no more than $n$ self-intersections, there is a uniformly continuous function
$$
\phi_\mfl:\Delta_\mfl(T)\to\R
$$
such that, uniformly on $\Delta_\mfl(T)$ as $N\to\infty$,}
$$
\phi_\mfl^N\to\phi_\mfl,\q\phi_{\mfl,\mfl^{-1}}^N\to(\phi_\mfl)^2.
$$
Secondly, {\em the restriction of the master field $\Phi_T$ to $L_n(\Sbb_T)$ is the unique function $L_n(\Sbb_T)\to\C$ 
with the following properties:
it is uniformly continuous, 
invariant under reduction and under under area-preserving homeomorphisms, 
satisfies the Makeenko--Migdal equations \eqref{SLMM},
and satisfies, for all simple loops $s$ and all $k\le n+1$,}
$$
\bar\Phi_T(s^k)=\phi_T(k,a_1(s),a_2(s)).
$$ 
For $a\in\Delta_\mfl(T)$ and $l\in\mfl(a)$, 
$$
\E(|\tr(H_l)-\phi_\mfl(a)|^2)=\phi_{\mfl,\mfl^{-1}}^N(a)-\phi^N_\mfl(a)^2+(\phi_\mfl^N(a)-\phi_\mfl(a))^2
$$
so the first statement implies that 
$$
\tr(H_l)\to\Phi_T(l)=\phi_\mfl(a)
$$
in $L^2$, uniformly in $l\in L_n(\Sbb_T)$.
So the two statements suffice to prove Proposition \ref{REG}.

For the simple combinatorial loop $\mfs$, set
$$
\phi_\mfs=\phi_T(1,.,.)
$$
then $\phi_\mfs$ is uniformly continuous on $\Delta_\mfs(T)$ and, by Proposition \ref{NFL},
$\phi^N_\mfs\to\phi_\mfs$ and $\phi_{\mfs,\mfs^{-1}}^N\to(\phi_\mfs)^2$ uniformly on $\Delta_\mfs(T)$. 
There are no self-intersections, so no Makeenko--Migdal equations.
For $a\in\Delta_\mfs(T)$ and $s\in\mfs(a)$,
$$
\Phi_T(s)=\phi_\mfs(a)=\phi_T(1,a).
$$
Hence the desired statements hold for $n=0$.

Let $n\ge1$ and suppose inductively that the desired statements hold for $n-1$.
Let $\mfl$ be a combinatorial planar loop with $n$ self-intersections.
Choose faces $k_0$ and $k_*$ of minimal and maximal winding number and set
$$
n_*=n_\mfl(k_*)-n_\mfl(k_0).
$$
Let $a\in\Delta_\mfl(T)$.
There exist uniquely $a_0,a_*\in[0,T]$ such that
$$
a_0+a_*=T,\q a_0n_\mfl(k_0)+a_*n_\mfl(k_*)=n_\mfl^Ta.
$$
Then, by Proposition \ref{DELTA}, there exists a unique $\a\in\R^\cI$, with 
$$
\sum_{i\in\cI}|\a_i|\le C_\mfl T
$$
such that, for
\begin{equation}\label{AOM}
a(t)=a+t\sum_{i\in\cI}\a_i\mu_i
\end{equation}
we have $a(t)\in\Delta_\mfl(T)$ for all $t\in[0,1)$ and 
$$
a_{k_0}(1)=a_0,\q a_{k_*}(1)=a_*.
$$
By Proposition \ref{MM equations}, the maps 
$$
t\mapsto\phi_\mfl^N(a(t)),\q t\mapsto\phi_{\mfl,\mfl^{-1}}^N(a(t))
$$ 
are differentiable on $[0,1)$, with
$$
\frac d{dt}\phi_\mfl^N(a(t))=\sum_{i\in\cI}\a_i\phi^N_{\mfl_i,\hat\mfl_i}(a(t))
$$
and
$$
\frac d{dt}\phi_{\mfl,\mfl^{-1}}^N(a(t))=\sum_{i\in\cI}\a_i\left(\phi^N_{\mfl_i,\hat\mfl_i,\mfl^{-1}}+\phi^N_{\mfl,\mfl_i^{-1},\hat\mfl_i^{-1}}-N^{-2}(\phi^N_{[\mfl]_i}+\phi^N_{[\hat\mfl]_i})\right)(a(t)).
$$
Here we have used the fact that the directional derivatives given by Proposition \ref{MM equations} are continuous on $\Delta_\mfl(T)$ to guarantee differentiability in any linear combination of those directions.
We integrate to obtain, for all $t\in[0,1)$,
\begin{equation}\label{LOM}
\phi_\mfl^N(a(t))=\phi_\mfl^N(a)+\sum_{i\in\cI}\int_0^t\a_i\phi^N_{\mfl_i,\hat\mfl_i}(a(s))ds
\end{equation}
and
\begin{align}
\notag
&\phi_{\mfl,\mfl^{-1}}^N(a(t))\\
\label{LON}
&=\phi_{\mfl,\mfl^{-1}}^N(a)
+\sum_{i\in\cI}\a_i\int_0^t\left(
\phi^N_{\mfl_i,\hat\mfl_i,\mfl^{-1}}+\phi^N_{\mfl,\mfl_i^{-1},\hat\mfl_i^{-1}}
-N^{-2}(\phi^N_{[\mfl]_i}+\phi^N_{[\hat\mfl]_i})\right)(a(s))ds.
\end{align}
Since $\phi_\mfl^N$ and $\phi_{\mfl,\mfl^{-1}}^N$ extend continuously to $\overline{\Delta_\mfl(T)}$
and the integrands on the right are bounded, these equations hold also for $t=1$.

Define $l:\overline{\Delta_\mfl(T)}\to L(\Sbb_T)$ as in the proof of Proposition \ref{DTC}.
Then $l(a(1))\sim s^{n_*}$ for some $s\in\mfs(a_0,a_*)$, so $H_{l(a(1))}=H_s^{n_*}$, and so
\begin{equation}\label{EDGE}
\phi_\mfl^N(a(1))=\E(\tr(H_{l(a(1))}))=\E(\tr(H_s^{n_*}))=\phi^N_{\mfs^{n_*}}(a_0,a_*).
\end{equation}
By Proposition \ref{NFL},
$$
\phi^N_{\mfs^{n_*}}(a_1,a_2)\to\phi_T(n_*,a_1,a_2)
$$
uniformly in $(a_1,a_2)\in\Delta_\mfs(T)$.
Write $l_i$ and $\hat l_i$ for the drawings of $\mfl_i$ and $\hat\mfl_i$ in $\Gbb$ for some $\Gbb\in\mfl(a)$.
Then
$$
\phi^N_{\mfl_i,\hat\mfl_i}(a)=\E(\tr(H_{l_i})\tr(H_{\hat l_i})).
$$
Since both $l_i$ and $\hat l_i$ have no more than $n-1$ self-intersections, by the inductive hypothesis, 
\begin{equation}\label{LOY}
\tr(H_{l_i})\to\phi_{\mfl_i}(a),\q \tr(H_{\hat l_i})\to\phi_{\hat\mfl_i}(a)
\end{equation}
in $L^2$, uniformly in $a\in\Delta_\mfl(T)$.
Hence
$$
\phi^N_{\mfl_i,\hat\mfl_i}\to\phi_{\mfl_i}\phi_{\hat\mfl_i}
$$
uniformly on $\Delta_\mfl(T)$.
Here we used the obvious submersions $\Delta_\mfl(T)\to\Delta_{\mfl_i}(T)$ and $\Delta_\mfl(T)\to\Delta_{\hat\mfl_i}(T)$
in evaluating $\phi_{\mfl_i}$ and $\phi_{\hat\mfl_i}$ on $\Delta_\mfl(T)$.
We let $N\to\infty$ in \eqref{LOM}, first in the case $t=1$ and then for $t\in(0,1)$ to see that 
$\phi_\mfl^N$ converges uniformly on $\Delta_\mfl(T)$ with uniformly continuous limit, $\phi_\mfl$ say, satisfying,
for all $t\in[0,1]$,
\begin{equation}\label{LOQ}
\phi_\mfl(a(t))=\phi_\mfl(a)+\sum_{i\in\cI}\int_0^t\a_i\phi_{\mfl_i}(a(s))\phi_{\hat\mfl_i}(a(s))ds.
\end{equation}
Now, by the argument leading to \eqref{EDGE},
$$
\phi_{\mfl,\mfl^{-1}}^N(a(1))=\phi^N_{\mfs^n,\mfs^{-n}}(a_0,a_*)
$$
and, by Proposition \ref{NFL}, for $s\in\mfs(a_1,a_2)$,
$$
\phi^N_{\mfs^n,\mfs^{-n}}(a_1,a_2)=\E(|\tr(H^n_s)|^2)\to\phi_T(n,a_1,a_2)^2
$$
uniformly in $(a_1,a_2)\in\Delta_\mfs(T)$ as $N\to\infty$.
We have
$$
\phi^N_{\mfl_i,\hat\mfl_i,\mfl^{-1}}(a)=\phi^N_{\mfl,\mfl_i^{-1},\hat\mfl_i^{-1}}(a)
=\E(\tr(H_{l_i})\tr(H_{\hat l_i})\tr(H_{l^{-1}}))
$$
and we have just shown that
$$
\E(\tr(H_{l^{-1}}))=\E(\tr(H_l))\to\phi_\mfl(a)
$$
uniformly in $a\in\Delta_\mfl(T)$. 
In combination with \eqref{LOY}, we deduce that, uniformly on $\Delta_\mfl(T)$,
$$
\phi^N_{\mfl_i,\hat\mfl_i,\mfl^{-1}}\to\phi_{\mfl_i}\phi_{\hat\mfl_i}\phi_\mfl.
$$
Hence, on letting $N\to\infty$ in \eqref{LON}, first in the case $t=1$ and then for $t\in(0,1)$, 
we see that $\phi^N_{\mfl,\mfl^{-1}}$ converges uniformly on $\Delta_\mfl(T)$ with uniformly continuous limit, 
$\phi_{\mfl,\mfl^{-1}}$ say, satisfying, for all $t\in[0,1]$,
\begin{equation}
\label{LOT}
\phi_{\mfl,\mfl^{-1}}(a(t))=\phi_{\mfl,\mfl^{-1}}(a)
+2\sum_{i\in\cI}\a_i\int_0^t\phi_{\mfl_i}(a(s))\phi_{\hat\mfl_i}(a(s))\phi_\mfl(a(s))ds.
\end{equation}
By differentiating \eqref{LOQ} and \eqref{LOT}, we see that 
$$
\frac d{dt}\left(\phi_{\mfl,\mfl^{-1}}(a(t))-\phi_\mfl(a(t))^2\right)=0
$$
so
$$
\phi_{\mfl,\mfl^{-1}}(a)-\phi_\mfl(a)^2=\phi_{\mfl,\mfl^{-1}}(a(1))-\phi_\mfl(a(1))^2=0.
$$
Thus the first of the desired statements holds for $n$.

We turn to the second statement.
First we will show the claimed properties of the master field $\Phi_T$ on $L_n(\Sbb_T)$.
By the first statement, for all $a\in\Delta_\mfl(T)$ and $l\in\mfl(a)$,
$$
\Phi_T(l)=\phi_\mfl(a).
$$
Hence $\Phi_T$ is invariant under area-preserving homeomorphisms.
Since $\phi_\mfl$ is uniformly continuous on $\Delta_\mfl(T)$, the inequality \eqref{MC} ensures that $\Phi_T$ is uniformly continuous on $L_n(\Sbb_T)$.
For $l_1,l_2\in L(\Sbb_T)$ with $l_1\sim l_2$, we have $H_{l_1}=H_{l_2}$, so $\Phi_T^N(l_1)=\Phi_T^N(l_2)$, 
and so, if $l_1,l_2\in\overline{L_n(\Sbb_T)}$, then
$$
\bar\Phi_T(l_1)=\lim_{N\to\infty}\Phi_T^N(l_1)=\lim_{N\to\infty}\Phi_T^N(l_2)=\bar\Phi_T(l_2).
$$
We used here the fact that $\Phi_T^N\to\Phi_T$ uniformly on $L_n(\Sbb_T)$.
Hence $\Phi_T$ is invariant under reduction.
By Proposition \ref{NFL}, for $a\in\Delta_\mfs(T)$, $s\in\mfs(a)$ and $k\le n+1$, 
$$
\bar\Phi_T(s^k)=\lim_{N\to\infty}\Phi_T^N(s^k)=\phi_T(k,a).
$$
It remains to show that $\Phi_T$ satisfies the Makeenko--Migdal equations \eqref{SLMM} on $L_n(\Sbb_T)$.
Let $l$ be a regular loop with $n$ self-intersections.
Let $i\in\cI$ and let $\th:[0,\eta)\times\Sbb_T\to\Sbb_T$ be a Makeenko--Migdal flow at $(l,v_i)$.
Write $a_\th(t)$ for the face-area vector of $l(t)=\th(t,l)$. 
Then
$$
a_\th(t)=a+t\mu_i
$$
so, by the argument leading to \eqref{LOQ},
$$
\E(\tr(H_{l(t)}))=\E(\tr(H_l))+\int_0^t\E(\tr(H_{l_i(s)})\tr(H_{\hat l_i(s)}))ds.
$$
By bounded convergence, on letting $N\to\infty$, we obtain
$$
\Phi_T(l(t))=\Phi_T(l)+\int_0^t\Phi_T(l_i(s))\Phi_T(\hat l_i(s)))ds
$$
as required.

Suppose finally that $\Psi:L_n(\Sbb_T)\to\C$ is another function with the same properties.
We have to show that $\Psi=\Phi_T$ on $L_n(\Sbb_T)$.
Given a combinatorial planar loop $\mfl$ with at most $n$ self-intersections, define a function
$$
\psi_\mfl:\Delta_\mfl(T)\to\C
$$
by
$$
\psi_\mfl(a)=\Psi(l(a))
$$
where $l(a)$ is constructed as in the proof of Proposition \ref{DTC}.
Then $\psi_\mfl$ is uniformly continuous and $\Psi(l)=\psi_\mfl(a)$ for all $a\in\Delta_\mfl(T)$ and all $l\in\mfl(a)$.
Given $a\in\Delta_\mfl(T)$ and a self-intersection $i$ of $\mfl$, choose $l\in\mfl(a)$ let $\th$ be a Makeenko--Migdal flow at $(l,v_i)$.
Then
$$
\psi_\mfl(a+t\mu_i)=\Psi(\th(t,l))
$$
so, since $\Psi$ satisfies the Makeenko--Migdal equations, $\psi_\mfl$ has a directional derivative given by
$$
\mu_i\psi_\mfl(a)=\left.\frac d{dt}\right|_{t=0}\Psi(\th(t,l))=\Psi(l_i)\Psi(\hat l_i)=\psi_{\mfl_i}(a)\psi_{\hat\mfl_i}(a)
$$
where $l_i,\hat l_i$ are the loops obtained by splitting $l$ at $v_i$, and $\mfl_i,\hat\mfl_i$ are the associated combinatorial loops.

Given $a\in\Delta_\mfl(T)$, define $a(t)$ as at \eqref{AOM}.
Then, by the argument leading to \eqref{LOM}, for all $t\in[0,1)$,
$$
\psi_\mfl(a(t))=\psi_\mfl(a)+\sum_{i\in\cI}\a_i\int_0^t\psi_{\mfl_i}(a(s))\psi_{\hat\mfl_i}(a(s))ds
$$
and, on letting $t\to1$, we obtain
$$
\bar\Psi(l(a(1)))=\Psi(l(a))+\sum_{i\in\cI}\a_i\int_0^1\Psi(l_i(a(s)))\Psi(\hat l_i(a(s)))ds.
$$
Now the same equation holds for $\Phi_T$ and
$$
\bar\Psi(l(a(1)))=\bar\Psi(s^{n_*})=\phi_T(n_*,a(1))=\bar\Phi_T(l(a(1)))
$$
and, by the inductive hypothesis, since $\mfl_i$ and $\hat\mfl_i$ have no more that $n-1$ self-intersections,
$$
\Psi(l_i(a(s)))=\Phi_T(l_i(a(s))),\q \Psi(\hat l_i(a(s)))=\Phi_T(\hat l_i(a(s))).
$$
Hence $\Psi(l(a))=\Phi_T(l(a))$, showing that $\Psi=\Phi_T$ on $L_n(\Sbb_T)$, as required.
Hence both statements hold for $n$ and the induction proceeds.

\section{Extension to loops of finite length}\label{RL}
\subsection{Some estimates for piecewise geodesic loops}
Our aim in this section is to prove Proposition \ref{ALL}, which is the final step in the proof of our main result Theorem \ref{main}.
For this it is convenient to work with piecewise geodesics.
We will need some associated estimates for the master field and its approximations, which we now develop.
Write $P_*(\Sbb_T)$ and $L_*(\Sbb_T)$ for the sets of piecewise geodesic paths and loops in $\Sbb_T$.
The sphere $\Sbb_T$ has positive injectivity radius $\k=\sqrt{\pi T}/2$.
For $\a\in P(\Sbb_T)$, write $n_0(\a)$ for the smallest integer such that $2^{-n_0(\a)}\ell(\a)<\k$.
For $n\ge n_0(\a)$, we define $D_n(\a)\in P_*(\Sbb_T)$ by parametrizing $\a$ by $[0,1]$ at constant speed and then interpolating the points $(\a(k2^{-n}):k=0,1,\dots,2^n)$ by geodesics.
Then $D_n(\a)\to\a$ in length as $n\to\infty$, so $P_*(\Sbb_T)$ is dense in $P(\Sbb_T)$ for the topology of convergence in length with fixed endpoints.
Note in particular that, when $\ell(\a)<\k$, we use the notation $D_0(\a)$ for the unique geodesic with the same endpoints as $\a$.
Define, for loops $\a\in L(\Sbb_T)$,
$$
\Psi_N(\a)=\sqrt{1-\Phi_T^N(\a)},\q \Phi_T^N(\a)=\E(\tr(H_\a))
$$
where $H=(H_\g:\g\in P(\Sbb_T))$ is a Yang--Mills holonomy field in $U(N)$.
It is straightforward to check, using the Cauchy--Schwarz inequality, that
\begin{equation}\label{PSA}
\Psi_N(\a\b)\le\Psi_N(\a)+\Psi_N(\b)
\end{equation}
whenever $\a$ and $\b$ have the same base point.
Also, by a standard estimate for the Brownian bridge, there is a constant $K_1<\infty$ such that, for all $N$ and all $a\in[0,T]$, for all simple loops $\a$ bounding a domain of area $a$,
\begin{equation}\label{PSB}
\Psi_N(\a)\le K_1\sqrt a.
\end{equation}
Moreover $\Psi_N$ inherits from $H$ the following properties: for all $\a,\b\in L(\Sbb_T)$ with $\a\sim\b$,
\begin{equation}\label{PSC}
\Psi_N(\a)=\Psi_N(\a^{-1}),\q\Psi_N(\a)=\Psi_N(\b)
\end{equation}
and, for all pairs of paths $\g_1,\g_2\in P(\Sbb_T)$ which concatenate to form a loop,
\begin{equation}\label{PSD}
\Psi_N(\g_1\g_2)=\Psi_N(\g_2\g_1).
\end{equation}
For $\a\in L_*(\Sbb_T)$, we have $\a\in\overline{L_n(\Sbb_T)}$ for some $n$ so, by Proposition \ref{REG}, we can define
$$
\Psi(\a)=\lim_{N\to\infty}\Psi_N(\a)=\sqrt{1-\Phi_T(\a)}.
$$
On letting $N\to\infty$, we see that the properties \eqref{PSA},\eqref{PSB},\eqref{PSC},\eqref{PSD} hold also for $\Psi$ on $L_*(\Sbb_T)$.
We note for later use a further inequality which follows from \eqref{PSA},\eqref{PSC},\eqref{PSD}: for all $\a,\b\in L_*(\Sbb_T)$,
\begin{equation}\label{PSE}
|\Psi(\a)-\Psi(\b)|\le\Psi(\a\b^{-1}).
\end{equation}

We now follow a line of argument which is adapted from \cite[Section 3.3]{MR2667871} where it is presented in more detail.
See also \cite[Theorem 4.1]{1601.00214}.
In particular, we will use the following isoperimetric inequality \cite[Lemma 3.3.5]{MR2667871}:
{\em there is a constant $K_2\in[\kappa^{-1},\infty)$ such that, for all $a\in[0,T]$ and all $\a\in P_*(\Sbb_T)$ of length $\ell(\a)<K_2^{-1}$ and such that the loop $s=\a^{-1}D_0(\a)$ is simple, we have
\begin{equation}\label{IPJ}
\sqrt a\le K_2\ell(\a)^{3/4}(\ell(\a)-\ell(D_0(\a)))^{1/4}
\end{equation}
where $a$ is the smaller of the areas of the connected components of $\Sbb_T\sm s^*$.}
The next proposition is a reformulation of \cite[Lemma 3.3.4]{MR2667871}.

\begin{prop}\label{CL}
There is a constant $K\in[\k^{-1},\infty)$ such that, for all $N\in\N$, all $n\ge0$ and 
all $\a\in P(\Sbb_T)$ with $2^{-n}\ell(\a)<K^{-1}$, we have
$$
\Psi_N(\a D_n(\a)^{-1})\le K\ell(\a)^{3/4}(\ell(\a)-\ell(D_n(\a)))^{1/4}.
$$
Moreover the same estimate holds for $\Psi$ whenever $\a\in P_*(\Sbb_T)$.
\end{prop}
\begin{proof}
The argument relies only on the properties \eqref{PSA},\eqref{PSB},\eqref{PSC},\eqref{PSD}, which hold for both $\Psi_N$ and $\Psi$, and the continuity of $\Psi_N$ on $L(\Sbb_T)$, 
which allows us to reduce to the case $\a\in P_*(\Sbb_T)$.
We will write it out for $\Psi$.
Consider first the case where $\a$ is injective, with $\ell(\a)<\k$.
Then (see \cite[Lemma 3.3.5]{MR2667871}) there is a lasso decomposition 
$$
\a D_0(\a)^{-1}\sim l_1\dots l_p,\q l_i=\g_is_i\g_i^{-1},\q s_i=\a_iD_0(\a_i)^{-1}
$$ 
where $s_i\in L_*(\Sbb_T)$ and $\g_i\in P_*(\Sbb_T)$ for all $i$, and where either $s_i$ is simple or $\a_i=D_0(\a_i)$, and such that
$$
\ell(\a)=\ell(\a_1)+\dots+\ell(\a_p),\q \ell(D_0(\a))=\ell(D_0(\a_1))+\dots+\ell(D_0(\a_p)).
$$
Write $a_i$ for the smaller of the areas of the connected components of $\Sbb_T\sm s_i^*$.
In the case $\a_i=D_0(\a_i)$, when there is only one such component, set $a_i=0$.
Note that $\ell(s_i)\le2\ell(\a_i)\le2\ell(\a)$.
Take $K=\max\{K_1K_2,2K_2\}$ and suppose that $\ell(\a)<K^{-1}$.
Then
\begin{align*}
\Psi(\a D_0(\a)^{-1})
&=\Psi(l_1\dots l_p)\le\Psi(l_1)+\dots+\Psi(l_p)=\Psi(s_1)+\dots+\Psi(s_p)\\
&\le K_1({\sqrt a_1}+\dots+{\sqrt a_p})\\
&\le K_1K_2\sum_i\ell(\a_i)^{3/4}(\ell(\a_i)-\ell(D_0(\a_i)))^{1/4}\\
&\le K\ell(\a)^{3/4}(\ell(\a)-\ell(D_0(\a)))^{1/4}
\end{align*}
where we used H\"older's inequality for the last step.

Now, for general $\a\in P_*(\Sbb_T)$ with $\ell(\a)<K^{-1}$, there is a lasso decomposition 
$$
\a\sim l_1\dots l_p\g,\q l_i=\g_i s_i\g_i^{-1},\q \ell(\a)=\ell(s_1)+\dots+\ell(s_p)+\ell(\g)
$$
where $s_i\in L_*(\Sbb_T)$ is simple and $\g_i\in P_*(\Sbb_T)$ for all $i$, and where $\g\in P_*(\Sbb_T)$ is injective.
Write $a_i$ for the smaller of the areas of the connected components of $\Sbb_T\sm s_i^*$.
Then 
$$
\Psi(l_i)=\Psi(s_i)\le K_1\sqrt{a_i}\le K_1K_2\ell(s_i)
$$
so
\begin{align*}
\Psi(l_1)+\dots+\Psi(l_p)\le K(\ell(s_1)+\dots+\ell(s_p))=K(\ell(\a)-\ell(\g)).
\end{align*}
On the other hand, by the first part,
$$
\Psi(\g D_0(\g)^{-1})\le K\ell(\g)^{3/4}(\ell(\g)-\ell(D_0(\g)))^{1/4}.
$$
But $D_0(\g)=D_0(\a)$, so
\begin{align*}
\Psi(\a D_0(\a)^{-1})
&=\Psi(l_1\dots l_p\g D_0(\g)^{-1})\le\Psi(l_1)+\dots+\Psi(l_p)+\Psi(\g D_0(\g)^{-1})\\
&\le K(\ell(\a)-\ell(\g))+K\ell(\g)^{3/4}(\ell(\g)-\ell(D_0(\g)))^{1/4}\\
&\le K\ell(\a)^{3/4}(\ell(\a)-\ell(D_0(\a)))^{1/4}.
\end{align*}

Finally, for $n\ge0$ and $\a\in P_*(\Sbb_T)$ with $2^{-n}\ell(\a)<K^{-1}$, we can write $\a$ as a concatenation $\a_1\dots\a_{2^n}$ such that
$$
D_n(\a)=D_0(\a_1)\dots D_0(\a_{2^n}),\q \ell(\a_i)=2^{-n}\ell(\a).
$$
Then there is a lasso decomposition
$$
\a D_n(\a)^{-1}\sim l_1\dots l_{2^n},\q l_i=\g_i\a_i D_0(\a_i)^{-1}\g_i^{-1}
$$
where $\g_i\in L_*(\Sbb_T)$ for all $i$.
Then, by the second part,
\begin{align*}
\Psi(\a D_n(\a)^{-1})&=\Psi(l_1\dots l_{2^n})\le\sum_i\Psi(\a_iD_0(\a_i)^{-1})\\
&\le\sum_iK\ell(\a_i)^{3/4}(\ell(\a_i)-\ell(D_0(\a_i))^{1/4}\\
&\le K\ell(\a)^{3/4}(\ell(\a)-\ell(D_n(\a))^{1/4}.
\end{align*}
\end{proof}

\subsection{Proof of Proposition \ref{ALL}}
Let $(H_\g:\g\in P(\Sbb_T))$ be a Yang--Mills holonomy field in $U(N)$.
We have to show that $\tr(H_l)$ converges in probability as $N\to\infty$ for all $l\in L(\Sbb_T)$.
We have to show further that the master field 
$$
\Phi_T(l)=\lim_{N\to\infty}\E(\tr(H_l)),\q l\in L(\Sbb_T)
$$ 
is the unique continuous function $L(\Sbb_T)\to\C$ which is invariant under reduction and under area-preserving homeomorphisms, satisfies the Makeenko--Migdal equations \eqref{SLMM} on regular loops, and satisfies \eqref{SLD} for simple loops.

Let $l\in L(\Sbb_T)$ and set $l_n=D_n(l)$.
Note that $D_n(l_m)=l_n$ when $m\ge n$.
By \eqref{PSE} and Proposition \ref{CL}, for $n\in\N$ sufficiently large and $m\ge n$, 
$$
|\Psi(l_m)-\Psi(l_n)|\le\Psi(l_ml_n^{-1})\le K\ell(l)^{3/4}(\ell(l)-\ell(l_n))^{1/4}.
$$
Also
$$
\E(|\tr(H_{l_n})-\tr(H_l)|^2)\le\E(\tr((H_{l_n}-H_l)(H_{l_n}-H_l)^*))=2\Psi_N(ll_n^{-1})^2
$$
so
$$
\|\tr(H_{l_n})-\tr(H_l)\|_2\le\sqrt2\Psi_N(ll_n^{-1})\le\sqrt2K\ell(l)^{3/4}(\ell(l)-\ell(l_n))^{1/4}.
$$
Since $\ell(l_n)\to\ell(l)$ as $n\to\infty$, we see that $\Psi(l_n)$ and $\Phi_T(l_n)=1-\Psi(l_n)^2$ must converge as $n\to\infty$.
Define
$$
\tilde\Phi(l)=\lim_{n\to\infty}\Phi_T(l_n).
$$
Let $n\to\infty$ and then $N\to\infty$ in the inequality
$$
\|\tr(H_l)-\tilde\Phi(l)\|_1\le\|\tr(H_l)-\tr(H_{l_n})\|_1+\|\tr(H_{l_n})-\Phi_T(l_n)\|_1+|\Phi_T(l_n)-\tilde\Phi_T(l)|
$$
to see that $\tr(H_l)\to\tilde\Phi(l)$ in probability and $\Phi_T(l)=\lim_{N\to\infty}\E(\tr(H_l))=\tilde\Phi(l)$.

The invariance of $\Phi_T$ on $L(\Sbb_T)$ under reduction and area-preserving homeomorphisms
follows from the corresponding invariance properties of $\Phi_T^N$.
The claimed properties of $\Phi_T$ on simple and regular loops were shown in Propositions \ref{NFL} and \ref{REG}.
We now show that $\Phi_T$ is continuous on $L(\Sbb_T)$.
For this, we translate to our context the argument of \cite[Proposition 3.3.9]{MR2667871}.
Let $\a\in L(\Sbb_T)$ and let $(\a_n:n\in\N)$ be a sequence in $L(\Sbb_T)$ which converges to $\a$ in length.
We have to show that $\Phi_T(\a_n)\to\Phi_T(\a)$.
There exist area-preserving homeomorphisms $\th_n$ on $\Sbb_T$ such that $\th_n(\a_n)$ converges to $\a$ in length with fixed endpoints.
We have $\Phi_T(\a)=1-\Psi(\a)^2$ and we know that $\Psi(D_m(\a_n))\to\Psi(\a_n)$ as $m\to\infty$. 
Hence it will suffice to consider the case where $\a_n$ is piecewise geodesic for all $n$ and $\a_n$ converges to $\a$ in length with fixed endpoints,
and to show then that $\Psi(\a_n)\to\Psi(\a)$ as $n\to\infty$.
Parametrize $\a$ at constant speed and choose parametrizations for the loops $\a_n$ so that
$$
\|\a_n-\a\|_\infty=\sup_{t\in[0,1]}|\a_n(t)-\a(t)|\to0.
$$
Fix $m\ge0$ and write $D_m(\a)$ and $\a_n$ as concatenations
$$
D_m(\a)=\s_1\dots\s_{2^m},\q \a_n=\a_{n,1}\dots\a_{n,2^m}
$$
where $\s_i$ is the geodesic from $\a((i-1)2^{-m})$ to $\a(i2^{-m})$ and $\a_{n,i}$ is the restriction of $\a_n$ to $[(i-1)2^{-m},i2^{-m}]$.
For $i=0,1,\dots,2^m$, denote by $\eta_{n,i}$ the geodesic from $\a(i2^{-m})$ to $\a_n(i2^{-m})$.
Then $\ell(\eta_{n,0})=\ell(\eta_{n,2^m})=0$ and, for $i=1,\dots,2^m-1$,
$$
\ell(\eta_{n,i})\le\|\a_n-\a\|_\infty.
$$
Set
$$
\b_n=\b_{n,1}\dots\b_{n,2^m},\q \b_{n,i}=\eta_{n,i-1}\a_{n,i}\eta_{n,i}^{-1}.
$$
Then $\a_n\sim\b_n$ and $D_0(\b_{n,i})=\s_i$ for all $i$.
So
$$
\Psi(\a_nD_m(\a)^{-1})=\Psi(\b_nD_m(\a)^{-1})
$$
and, by the argument used in the last part of the proof of Proposition \ref{CL},
$$
\Psi(\b_nD_m(\a)^{-1})\le K\ell(\b_n)^{3/4}(\ell(\b_n)-\ell(D_m(\a))^{1/4}.
$$
Now
$$
\ell(\b_n)\le\ell(\a_n)+2^{m+1}\|\a_n-\a\|_\infty
$$
so
\begin{align*}
&|\Psi(\a_n)-\Psi(\a)|\le\Psi(\a_nD_m(\a)^{-1})
+|\Psi(D_m(\a))-\Psi(\a)|\\
&\q\q\le K(\ell(\a_n)+2^{m+1}\|\a_n-\a\|_\infty)^{3/4}(\ell(\a_n)-\ell(D_m(\a))+2^{m+1}\|\a_n-\a\|_\infty)^{1/4}\\
&\q\q\q\q\q\q+|\Psi(D_m(\a))-\Psi(\a)|.
\end{align*}
On letting first $n\to\infty$ and then $m\to\infty$, we see that $\Psi(\a_n)\to\Psi(\a)$ as required.

Finally, suppose that $\Psi:L(\Sbb_T)\to\C$ is another function with the same
properties.
For each combinatorial planar loop $\mfl$, define a function
$$
\psi_\mfl:\overline{\Delta_\mfl(T)}\to\C
$$
by
$$
\psi_\mfl(a)=\Psi(l(a))
$$
where $l(a)$ is chosen as in Proposition \ref{DTC}. 
Since $\Psi$ and $a\mapsto l(a)$ are continuous, so is $\psi_\mfl$.
Then, since $\overline{\Delta_\mfl(T)}$ is compact, $\psi_\mfl$ is
uniformly continuous.
But $\Psi(l)=\psi_\mfl(a)$ for all $a\in\Delta_\mfl(T)$ and all $l\in\mfl(a)$,
so the inequality \eqref{MC} shows that $\Psi$ is uniformly continuous on $L_n(\Sbb_T)$ for all $n$.
Then, by Proposition \ref{REG}, $\Psi=\Phi_T$ on regular loops.
But regular loops are dense in $L(\Sbb_T)$, so $\Psi=\Phi_T$ on $L(\Sbb_T)$ by continuity.

\section{Further properties of the master field \label{Prop MF}}
\subsection{Relation with the Hermitian Brownian loop}
Let $W=(W_t:t\ge0)$ be a Brownian motion in the set of $N\times N$ Hermitian matrices $H(N)$ equipped with the inner product
$$
\<h_1,h_2\>=N\Tr(h_1h_2^*).
$$
Let $w=(w_t:t\ge0)$ be a free Brownian motion, defined on some non-commutative probability space $(\cA,\t)$. 
The inner product is scaled with $N$ so that $W$ converges in non-commutative distribution (in probability) to $w$, that is to say, for all $n\in\N$ and all $t_1,\dots,t_n\ge0$,
$$
\tr(W_{t_1}\dots W_{t_n})\to\t(w_{t_1}\dots w_{t_n})
$$
in probability as $N\to\infty$.
Fix $T\in(0,\infty)$ and define for $t\in[0,T]$
$$
B_t=W_t-\tfrac{t}TW_T,\q b_t=w_t-\tfrac{t}Tw_T.
$$
Then $B=(B_t:t\in[0,T])$ is a Hermitian Brownian loop in $H(N)$ and $B$ converges in non-commutative distribution to $b=(b_t:t\in[0,T])$.
The non-commutative process $b$ is called the {\em free Hermitian Brownian loop}.

Let $x=(x_t:t\in[0,T])$ be a free unitary Brownian loop in $(\cA,\t)$, as defined in Subsection \ref{FUBL}.

\begin{prop} 
Suppose that $T\in(0,\pi^2]$.
Then, for all $t\in[0,T]$ and all $n\in\Z$,
$$
\t(x_t^n)=\int_\R e^{inx}s_{\sqrt{t(T-t)/T}}(x)dx=\t(e^{inb_t})
$$
where $s_t$ is the semi-circle density \eqref{SUF} of variance $t$.
On the other hand, for almost all $T$ and almost all $s,t\in(0,T)$ with $s<t$,
$$
\t(x^*_sx_t)\not=\t(e^{-ib_s}e^{ib_t})
$$
so $(e^{ib_t}:t\in[0,T])$ is not a free unitary Brownian loop.
\end{prop}

\begin{proof}  
The first assertion is the content of Proposition \ref{SCC}. 
We turn to the second assertion.
Let $(X_t:t\in[0,T])$ be a Brownian loop in $U(N)$ based at $1$.
Then, since Brownian motion in $U(N)$ is a L\'evy process, $X_s^{-1}X_t$ has same law as $X_{t-s}$. 
On letting $N\to\infty$, we deduce that
$$
\t(x_s^*x_t)=\tau(x_{t-s})=\t(e^{ib_{t-s}})=\t(e^{i(b_t-b_s)})
$$ 
where we used free independence and stationarity of the increments of free Brownian motion for the last equality.
Hence, by the scaling properties of free Brownian motion,
$$
\t(e^{-ib_s}e^{ib_t})-\t(x^*_sx_t)=\t(e^{-ib_s}e^{ib_t}-e^{i(b_t-b_s)})=F_{s/T,t/T}(\sqrt T)
$$
where, taking now $T=1$, 
$$
F_{s,t}(\s)=\t(e^{-i\s b_s}e^{i\s b_t}-e^{i\s (b_t-b_s)}).
$$
By Fubini's theorem, it will suffice to show, for all $s,t\in(0,1)$ with $s<t$, that $F_{s,t}(\s)\not=0$ for almost all $\s\in(0,\pi]$.
We expand the exponential function up to fourth order and use scale invariance of free Brownian motion to obtain
$$
F_{s,t}(0)=F_{s,t}'(0)=F_{s,t}''(0)=F_{s,t}'''(0)=0,\q F_{s,t}''''(0)=2\t(b_s^2b_t^2-b_sb_tb_sb_t).
$$
The variables $(b_t:t\in[0,1])$ are semi-circular, therefore all free cumulants of order more than $3$ vanish (see for example \cite[equation 11.4]{MR2266879}). 
So, using the decomposition of moments into free cumulants,%
%%%%%%%%%%%%%%%%%%%%%%%%%%%%%%%%%%%%%%%%%%%%%%%%%%%%%%%%%%%%%%%%%%%%%%%%%%%%%%%%%%%%%%%%%%%%%%%%%%
\footnote{Here it can be understood as a `non-commutative' Wick formula, with non-crossing matchings in place of all matchings.} 
%%%%%%%%%%%%%%%%%%%%%%%%%%%%%%%%%%%%%%%%%%%%%%%%%%%%%%%%%%%%%%%%%%%%%%%%%%%%%%%%%%%%%%%%%%%%%%%%%%
(see \cite[equation 11.8]{MR2266879}),
$$ 
\t(b_s^2b_t^2-b_sb_tb_sb_t)=\t(b_s^2)\t(b_t^2)-\t(b_sb_t)^2=s(t-s)(1-t)>0.
$$
Since $F_{s,t}$ is analytic in $\s$ on $(0,\pi]$, this implies that it has at most finitely many zeros.
\end{proof}

\subsection{Duality at the midpoint of the loop}
Recall from \eqref{SCD} and \eqref{DCG} the form of $\rho_T$.
It will be convenient to set $\a=0$ and $\b=2/\sqrt T$ in the subcritical case $T\in(0,\pi^2]$.
The following relation appeared first in the physics literature \cite[equation 1.2]{MR1321333}, without a mathematical proof.  

\begin{prop} 
Let $(x_t:t\in[0,T])$ be a free unitary Brownian loop.
Then, for all $T>0$, the spectral measure of $x_{T/2}$ has  a density $\rho_T^*$ with respect to Lebesgue measure on $\U$ (of mass $2\pi$), 
which is invariant under complex conjugation and is such that 
$$
\pi\rho^*_{T}:\U\cap\H\to(\a,\b) 
$$
is the inverse mapping of 
$$
e^{i\pi\rho_T}:(\a,\b)\to\U\cap\H. 
$$
\end{prop}

\begin{proof}  
We write the proof for the supercritical case $T>\pi^2$, leaving the minor adjustments needed when $T\le\pi^2$ to the reader.
The function $\rho_T:(\a,\b)\to(0,1)$ is continuous and strictly decreasing, with $\rho_T(\a)=1$ and $\rho_T(\b)=0$.  
Indeed, according to formula (\ref{DCG}) and an elementary computation (see for example \cite[equation 150]{MR3474469}), for $x\in(\a,\b)$,
$$ 
\frac{\pi\a}2\sqrt{(x^2-\a^2)(\b^2-x^2)}\rho_T'(x)=\int_0^1\frac{\a^2s^2-x^2}{\b^2\sqrt{(1-s^2)(1-k^2s^2)}}ds<0.
$$
Write $\psi$ for the inverse of the bijection $\pi\rho_T:(\a,\b)\to(0,\pi)$.
For all $n\in\Z\sm\{0\}$, by Lemma \ref{formula sin}, 
$$
\tau(x_{T/2}^n)=\frac2{n\pi}\int_\a^\b\sin\{n\pi\rho_T(x)\}dx=-\frac2{n\pi}\int_0^\pi\sin(n\th)\psi'(\th)d\th.
$$
We integrate by parts to obtain
$$
\tau(x_{T/2}^n)=\frac2\pi\int_0^\pi\cos(n\th)\psi(\th)d\th.
$$
Hence the spectral measure of $x_{T/2}$ has a density $\rho^*_T$ with respect to Lebesgue measure on $\U$ given by
\begin{equation*}
\rho^*_T(e^{i\th})=\psi(|\th|)/\pi,\q |\th|\le\pi.
\end{equation*}
\end{proof}

\subsection{Convergence to the planar master field}
We now investigate the behaviour of the master field $\Phi_T$ as $T\to \infty$.  
For $T>0$, $n\in\N$ and $t\in[0,T]$, set
$$
m_T(n,t)=\Phi_T(l^n)
$$
where $l$ is a simple loop which divides $\Sbb_T$ into components of areas $t$ and $T-t$.
Recall that $m_T(n,t)$ does not depend on the choice of $l$.

\begin{prop}
\label{convergence to  free Brownian motion} 
We have
$$
m_T(n,t)\to
\frac{e^{-nt/2}}{2\pi in}\int_\g\left(1+\frac1z\right)^ne^{-ntz}dz
=e^{-nt/2}\sum_{k=1}^{n-1}\frac{(-t)^k}{k!}\binom{n}{k+1}n^{k-1} 
$$
uniformly in $t\in[0,T]$ as $T\to\infty$, where $\g$ is any positively oriented loop in $\C$ winding once around $0$.
\end{prop}
\begin{proof} 
Since the second complete elliptic integral $E(k)$ is bounded and the first $K(k)$ is bounded on compacts in $[0,1)$, the relation
$$
T=8EK-4(1-k^2)K^2
$$
forces $k\to1$ as $T\to\infty$.
Since $\a=k\b\le1/2$ and $\b\ge1/2$ for all $T$, this implies $\a,\b\to1/2$ as $T\to\infty$.
Hence $\rho_T(x)\to1$ for $|x|<1/2$ and $\rho_T(x)\to0$ for $|x|>1/2$, so
$$
G_T(z)=\int_\R\frac{\rho_T(x)}{z-x}dx\to\int_{-1/2}^{1/2}\frac{dx}{z-x}=\Log
\left(\frac{z+1/2}{z-1/2}\right)
$$
uniformly on compacts in $\{z\in\C:|z|>1/2\}$.
By Proposition \ref{formula sin}, for $R>1/2$ and $T$ sufficiently large, uniformly in $t\in[0,T]$,
\begin{align*}
m_T(n,t)
&=\frac{1}{2\pi in}\int_{\g_R}\exp\{-n(tz-G_T(z))\}dz\\
&\to\frac{1}{2\pi in}\int_{\g_R}e^{-ntz}\left(\frac{z+1/2}{z-1/2}\right)^ndz
=\frac{e^{-nt/2}}{2\pi in}\int_\g\left(1+\frac1z\right)^ne^{-ntz}dz
\end{align*}
where $\g_R$ is the positively oriented boundary of $\{z\in\C:|z|=R\}$.
\end{proof}

Denote by $L(\R^2)$ be the set of loops of finite length in $\R^2$ and let 
$$
\Phi:L(\R^2)\to[-1,1]
$$
be the planar master field as defined in \cite{MR3636410}.

\begin{prop} \label{Convergence Sphere Plane}
For each $T>0$, fix a point $x_T\in\Sbb_T$ and denote by $p_T$ the inverse map of the stereographic projection $\Sbb_T\sm\{x_T\}\to\C$.
Then, for all $l\in L(\R^2)$, 
$$
\Phi_T(p_T(l))\to\Phi(l)\q\text{as $T\to \infty$.}
$$
\end{prop}
\begin{proof}
Let $l$ be a simple loop in $L(\R^2)$ and denote by $a$ the finite area enclosed by $l$.
Then $p_T(l)$ is a simple loop in $L(\Sbb_T)$ which divides $\Sbb_T$ into two components and does not pass through $x_T$.
Denote by $a_T$ the area of the component which does not contain $x_T$.
Then $a_T\to a$ as $T\to\infty$.
By Proposition \ref{convergence to  free Brownian motion}, this implies
$$
\Phi_T(p_T(l^n))\to e^{-na/2}\sum_{k=1}^{n-1}\frac{(-a)^k}{k!}\binom{n}{k+1}n^{k-1}=\Phi(l^n)
$$
as $T\to\infty$, where we used \cite[equation (2)]{MR3636410} for the last equality. 

Now $\Phi$ also satisfies the Makeenko--Migdal equations \cite{MR3636410}. 
By a variation of the argument used to prove Theorem \ref{REG}, we can extend convergence from
powers of simple loops to all regular loops.
We sketch the small change which is needed.
There is now a face, $k_\infty$ say, of infinite area.
So we work in the orthant 
$$
Y_\mfl=\{(a_k:k\in\cF_\mfl):a_{k_\infty}=\infty\text{ and }a_k\in(0,\infty)\text{ for all }k\not=k_\infty\}.
$$
Set
$$
\overline{a}=\sum_{k\not=k_\infty}a_k.
$$
Write $k_0,k_*$, as before, for the faces of minimal and maximal winding number, now choosing the additive constant so that
$n_\mfl(k_\infty)=0$.
Given $a\in Y_\mfl$, either $\<a,n_\mfl\>\ge0$, or $\<a,n_\mfl\><0$.
(We use here the convention that $\infty\times0=0$.)
In the first case, $k_*\not= k_\infty$ and there exists $a'\in\overline{Y_\mfl}$ with $a'_k=0$ for $k\not=k_*,k_\infty$ such that
$$
\<a',n_\mfl\>=a_{k_*}'n_\mfl(k_*)=\<a,n_\mfl\>.
$$
Set
$$
v_k=
\begin{cases}
a'_k-a_k,&\text{if $k\not=k_\infty$},\\
\overline{a}-\overline{a'},&\text{if $k=k_\infty$}.
\end{cases}
$$
and set $a(t)=a+tv$.
Then $a'=a(1)$ and $a(t)\in Y_\mfl$ for all $t\in[0,1)$, and $v\in\mfm_\mfl$ by Proposition \ref{DELTA}.
An analogous argument holds in the second case.
We can then proceed as in Subsection \ref{PRREG}.
The arguments of Section \ref{RL} also carry over to extend the limit
$$
\Phi_T(p_T(l))\to\Phi(l)
$$
to all $l\in L(\R^2)$.
We omit the details.
\end{proof}

\subsection{Uniqueness of the master field}\label{MOREU}
In Theorem \ref{EUMF}, we showed that the master field is characterized by certain properties.
In fact there is some redundancy in this characterization, as the following result shows.

\begin{prop}\label{UNI}
Let $\Phi:L(\Sbb_T)\to\C$ be a continuous function, which is invariant under reduction and under area-preserving, orientation-preserving homeomorphisms, satisfies the Makeenko--Migdal equations on regular loops, and 
is given on simple loops $l$ by
\begin{equation}\label{SLC1}
\Phi_T(l)=\frac2\pi\int_0^\infty\cosh\left\{(a-b)x/2\right\}\sin\{\pi\rho_T(x)\}dx
\end{equation}
where $a$ and $b$ are the areas of the connected components of $\Sbb_T\sm\{l^*\}$. 
Then $\Phi$ is the master field $\Phi_T$.
\end{prop}

The proof will be based on an argument for a special class of loops which we now introduce.
Informally, for $n\ge1$ fix an initial point $x_1$ and draw an inward anticlockwise spiral which winds $n$ times around another point $o$,
crossing the line $ox_1$ at points $x_2,\dots,x_n$ then, on hitting $ox_1$ for the $n$th time, returning to $x_1$ along $ox_1$.
Thus we obtain a combinatorial planar loop $\mfw_n$ whose combinatorial graph is given as follows:
$$
\cV=\{1,\dots,n\},\q\cE=\{1,\dots,n\}\cup\{1',\dots,(n-1)'\},\q\cF=\{0,1,\dots,n\}
$$
where, for $j=1,\dots,n-1$,
$$
s(j)=j,\q t(j)=j+1,\q s(j')=j+1,\q t(j')=j
$$
and 
$$
l(j)=l(j')=j,\q r(j)=r(j')=j-1
$$
while
$$
s(n)=t(n)=n,\q l(n)=n,\q r(n)=n-1.
$$
See Figure \ref{Max Winding}. 
\begin{figure}\centering 
\includegraphics[width=100 mm,height=60 mm]{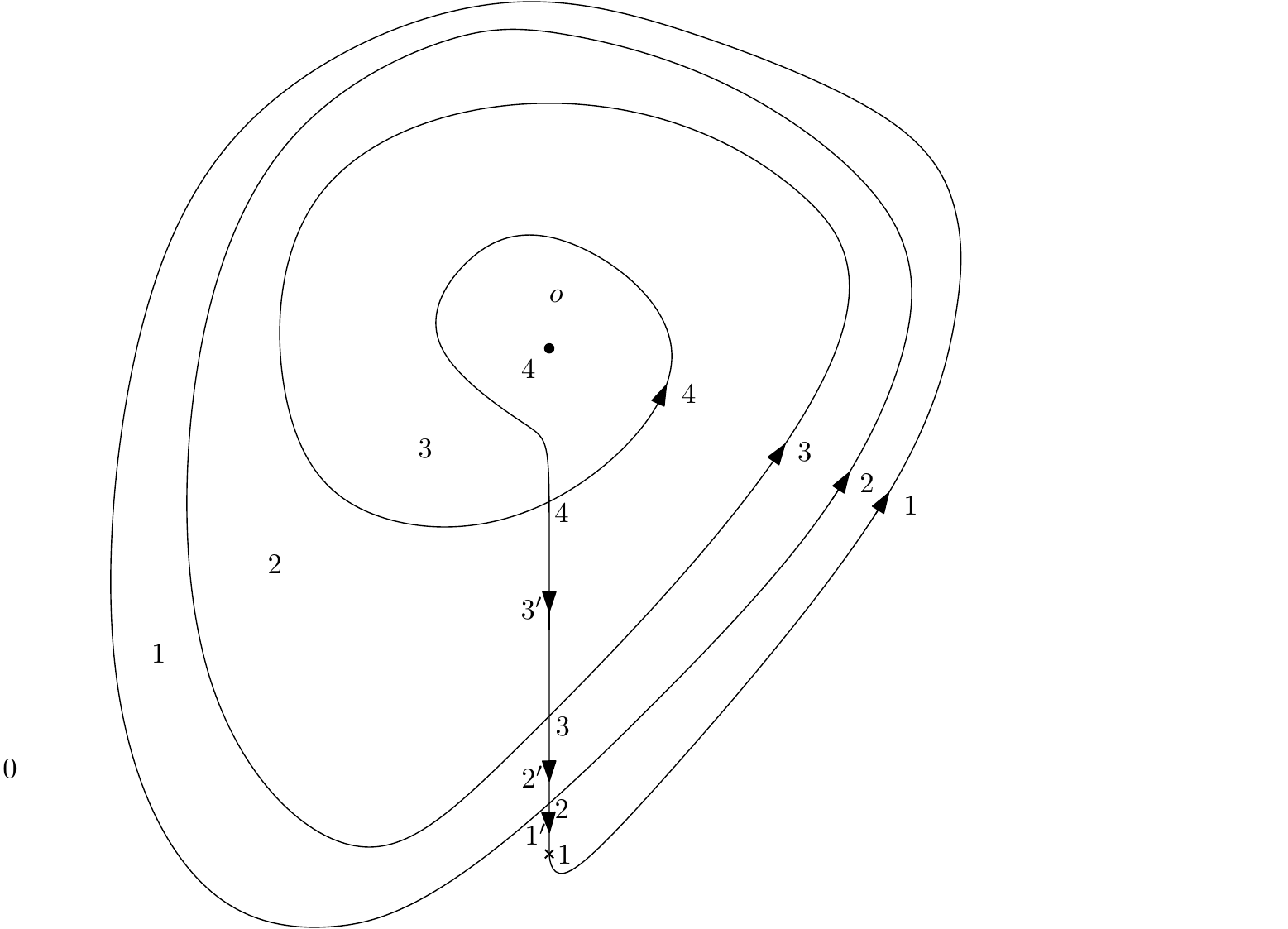}\caption{\label{Max Winding} A drawing of the maximally winding loop $\mfw_4$.}
\end{figure}  
Here, we have used a non-standard labelling for the edges and faces which is adapted to the structure of the graph.
Note that the self-intersections of $\mfl_n$ are labelled by $\{2,\dots,n\}$.
If we fix the additive constant for the winding number so that $n_{\mfw_n}(0)=0$, then $n_{\mfw_n}(n)=n$.
For $n\ge1$ and for any combinatorial planar loop $\mfl$ with $n-1$ self-intersections, we have 
$$
n_*=\max\{n_\mfl(k)-n_\mfl(k'):k,k'\in\cF\}\le n.
$$
We call $\mfw_n$, and any associated regular embedded loop $l$, and any rerooting of $l$, a {\em maximally winding} loop.

\begin{proof}[Proof of Proposition \ref{UNI}]
Let $n\ge1$ and suppose inductively that $\Phi(l^m)=\Phi_T(l^m)$ for all $m\le n$ and all simple loops $l$.
A comparison of equations \eqref{SLD} and \eqref{SLC1} shows that this is true for $n=1$.
Let $l$ be a simple loop which divides $\Sbb_T$ into components of areas $a_0,a_*\in(0,T)$.
We can find $(\a_1,\a_2,\dots,\a_{n+2})$ such that
$$
\a_1=0,\q\a_2=a_0,\q\a_{n+1}=a_*,\q\a_{n+2}=0
$$
and, for $m=2,\dots,n+1$,
$$
\a_{m-1}-2\a_m+\a_{m+1}<0.
$$
Consider the (constant) vector field $v$ on $\Delta_{\mfl_{n+1}}(T)$ given by
$$
v=\sum_{i=2}^{n+1}\a_i\mu_i.
$$
Then $v_0=-a_0$ and $v_{n+1}=-a_*$ and $v_k>0$ for $k=1,\dots,n$.
Set
$$
a(t)=(a_0,0,\dots,0,a_*)+tv
$$
then $a(t)\in\Delta_{\mfl_{n+1}}(T)$ for all $t\in(0,1)$ and $a_0(1)=a_{n+1}(1)=0$.
There exists a continuous family loops $(l(t):t\in[0,1])$, with a common basepoint such that, 
$l(0)=l^{n+1}$, $l(t)\in\cG_{\mfl_{n+1}}(a(t))$ for all $t\in(0,1)$, and $l(1)$ is a maximally winding loop with $n-2$ self-intersections.
Then, by the arguments used in the proof of Proposition \ref{REG},
$$
\Phi(l(1))=\Phi(l^{n+1})+\sum_{i=2}^{n+1}\a_i\int_\t^1\Phi(l_i(s))\Phi(\hat l_i(s))ds
$$
where $l_i(s)$ and $\hat l_i(s)$ are maximally winding loops having $i-2$ and $n+1-i$ self-intersections.
But the same equation holds for $\Phi_T$ and the inductive hypothesis, combined with the argument of the proof of Proposition \ref{REG}, implies that
$$
\Phi(l(1))=\Phi_T(l(1)),\q 
\Phi(l_i(s))=\Phi_T(l_i(s)),\q
\Phi(\hat l_i(s))=\Phi_T(\hat l_i(s)).
$$
Hence $\Phi(l^{n+1})=\Phi_T(l^{n+1})$ and the induction proceeds.
Finally, by Proposition \ref{REG}, it follows that $\Phi(l)=\Phi_T(l)$ for all $l\in L(\Sbb_T)$.
\end{proof}

On the other hand, condition \eqref{SLC1} is not redundant in Proposition \ref{UNI}, as we now show.
Each loop $l\in L(\Sbb_T)$ has a winding number function
$$
n_l:\Sbb_T\sm l^*\to\Z
$$
which is unique up to an additive constant.
By the Banchoff--Pohl inequality \cite{MR0305319}, we know that $n_l\in L^2(\Sbb_T)$ so $n_l$ has a well-defined average value $\<n_l\>$ 
with respect to the uniform distribution on $\Sbb_T$, up to the same additive constant.
Hence, we can define a unique function $\Psi:L(\Sbb_T)\to\C$ by
$$
\Psi(l)=e^{2\pi i\<n_l\>}.
$$
For loops $l_1,l_2$ based at the same point, we have $n_{l_1l_2}=n_{l_1}+n_{l_2}$, so 
$$
\Psi(l_1l_2)=\Psi(l_1)\Psi(l_2).
$$
Morever, $\Psi$ is invariant under any area-preserving, orientation-preserving diffeomorphism so, in particular, under any Makeenko--Migdal flow.
Consider, for $n\in\Z$, the twisted master field $\Phi_T^{(n)}:L(\Sbb_T)\to\C$ given by
$$
\Phi^{(n)}_T(l)=\Psi(l)^n\Phi_T(l).
$$
Then $\Phi_T^{(n)}$ is continuous, invariant under reduction and area-preserving, orientation-preserving homeomorphisms and satisfies the Makeenko--Migdal equations on regular loops.
However, for a simple loop $l$ which winds positively around a domain of area $a$, we have
$$
\Psi(l)=e^{2\pi ia/T}
$$
so, for $n\not=0$, $\Phi_T^{(n)}$ is not the master field.
Hence, by Proposition \ref{UNI}, or by inspection, $\Phi_T^{(n)}$ does not satisfy \eqref{SLC1}.
For $n\not=0$, $\Phi_T^{(n)}$ also fails to be invariant under orientation-reversing homeomorphisms.
We do not know whether this stronger invariance condition would allow one to dispense with \eqref{SLC1} in Proposition \ref{UNI}.
 
\subsection{Combinatorial formulas for the master field \label{planarloops}} 
Rusakov \cite{Rusakov} proposed, without proof, that there should be a closed formula for the value of the master field for any regular loop on the sphere.  
We now prove a formula with a slightly different form to the one given in \cite{Rusakov} and for the following restricted class of loops introduced in \cite{MR605753}.
Let us say that a combinatorial planar loop $\mfl$ is {\em splittable} if for all self-intersection points $i$ of $\mfl$, 
the two loops $\mfl_i,\hat\mfl_i$, obtained by following outgoing strands of $\mfl$ starting from $i$, intersect only at $i$.

\begin{figure}
\label{fig planar loop}
\centering 
\includegraphics[width=50 mm,height=44 mm]{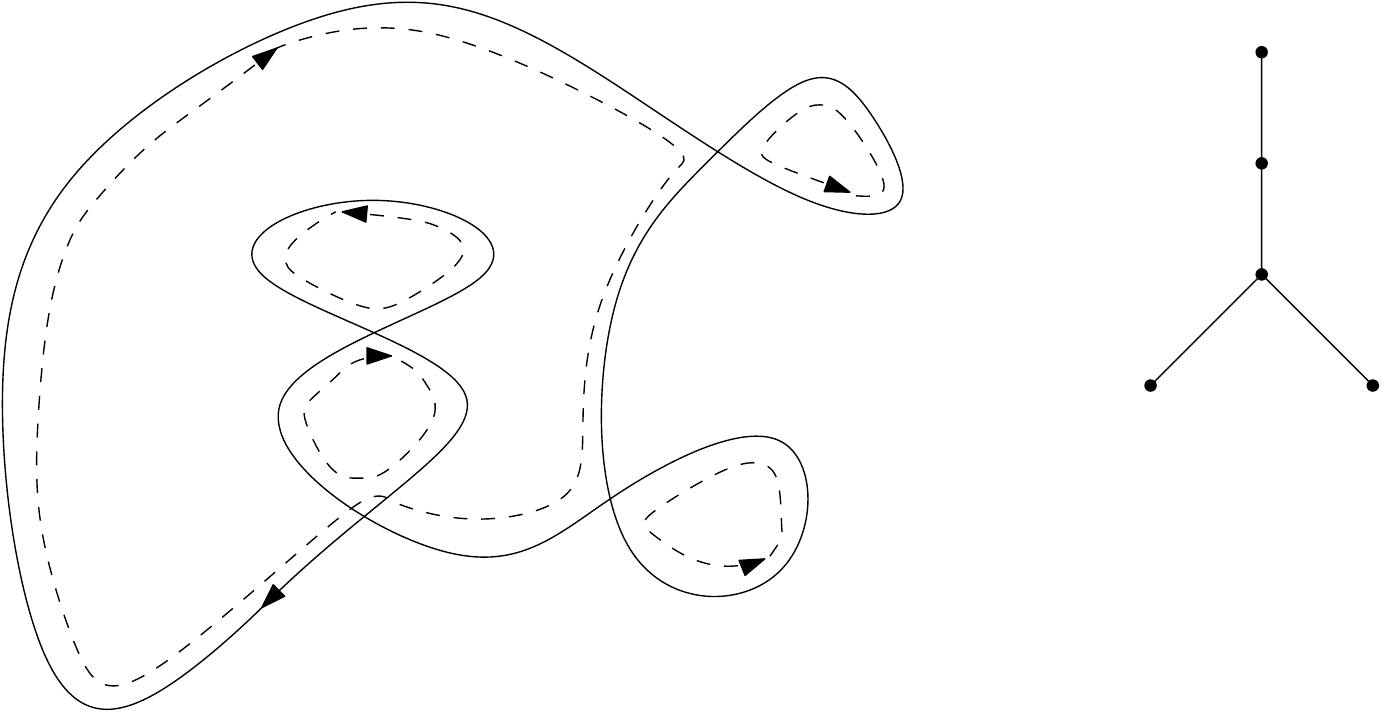}
\caption{A splittable combinatorial planar loop with its family of simple loops $\cS_\mfl$ drawn in dashed lines,
next to a combinatorial representation of the tree structure of $\cS_\mfl$.}
\end{figure}

Let $\mfl$ be a splittable combinatorial planar loop with $n$ points of intersection.
On splitting $\mfl$ at all points of intersection, 
we obtain a family of simple combinatorial loops $\cS_\mfl=\{\mfs_1,\dots\mfs_{n+1}\}$ in $\mfl$, 
which has the structure of a tree,
in which $\mfs_j$ and $\mfs_{j'}$ are adjacent if they share a point of intersection of $\mfl$.
We choose the sequence $(\mfs_1,\dots,\mfs_{n+1})$ to be an {\em adapted labelling} of $\cS_\mfl$, 
meaning that $\mfs_{k+1}$ is adjacent to at least one of $\mfs_1,\dots,\mfs_k$ for all $k\in\{1,\dots,n\}$.
Given $T\in(0,\infty)$, 
a distinguished face $k\in\cF_\mfl$ and an adapted labelling $(\mfs_1,\dots,\mfs_{n+1})$ of $\cS_\mfl$, 
let us say that a sequence $(\g_1,\dots,\g_{n+1})$ of disjoint simple loops in $\C$ around $[-\b,\b]$ is {\em admissible} if
\begin{enumerate}
\item[(a)]$\g_{j+1}$ lies in the infinite component of $\C\sm\g_j^*$ for all $j\le n$,
\item[(b)]$\g_j$ has the same orientation in $\C$ as $\mfs_j$ has around $k$ for all $j$.
\end{enumerate}
For any self-intersection point $i$ of $\mfl$, we label the loops among $\mfl_i,\hat\mfl_i$ 
using the left and right outgoing edges at $i$ by $\mfl_{i,l}$ and $\mfl_{i,r}$ respectively.
The loops $\mfl_{i,l}$ and $\mfl_{i,r}$ are also splittable, 
and the pair $\{\cS_{\mfl_{i,l}},\cS_{\mfl_{i,r}}\}$ is a partition of $\cS_\mfl$.
Write $j(i,l)$ and $j(i,r)$ for the loop labels in $\cS_\mfl$ such that $\mfs_{j(i,l)}$ and $\mfs_{j(i,r)}$ use the left and right outgoing edges at $i$ respectively.
Let $n_\mfl$ be the winding number function of $\mfl$, where the additive constant is chosen so that $n_\mfl(k)=0$.
Set $\ve_j=-1$ or $\ve_j=1$ according as $\mfs_j$ winds positively or negatively around $k$.
Set
$$
\cO_\mfl=\{(z_1,\dots,z_{n+1})\in\C^{n+1}:z_j\not=z_{j'}\text{ for all $j,j'$ distinct}\}
$$
and, for $a\in\Delta_\mfl(T)$ and $z\in\cO_\mfl$, define
$$   
Q_{\mfl,k}(a,z)=\frac{\prod_{j=1}^{n+1}\exp\{\<n_\mfl,a\>z_j+\ve_jG_T(z_j)\}}{\prod_{i\in\cI}(z_{j(i,r)}-z_{j(i,l)})}.
$$
Recall from Subsection \ref{PRREG} that, for all combinatorial planar loops $\mfl$, there is a uniformly continuous map
$$
\phi_\mfl:\Delta_\mfl(T)\to\R
$$
such that $\Phi_T(l)=\phi_\mfl(a)$ for all $a\in\Delta_\mfl(T)$ and all $l\in\mfl(a)$.

\begin{prop}
\label{formula planar loops}
For all $T\in(0,\infty)$, 
all splittable combinatorial planar loops $\mfl$ with $n$ self-intersections and equipped with a distinguished face $k$,
all adapted labellings $(\mfs_1,\dots,\mfs_{n+1})$ of $\cS_\mfl$,
and all admissible sequences of closed loops $(\g_1,\dots,\g_{n+1})$,
we have, for all $a\in\Delta_\mfl(T)$,
\begin{equation}\label{CFM}
\phi_\mfl(a)=\left(\frac1{2\pi i}\right)^{n+1}\int_{\g_1}dz_1\dots\int_{\g_{n+1}}dz_{n+1}Q_{\mfl,k}(a,z).
\end{equation}
\end{prop}

We will need the following technical lemma.
Set
$$
\Delta_{\mfl,\C}(T)=\left\{(a_k:k\in\cF_\mfl):a_k\in\C,\,\sum_ka_k=T\right\}.
$$

\begin{lem}\label{ANA}
The map $\phi_\mfl$ has an analytic extension $\Delta_{\mfl,\C}(T)\to\C$.
\end{lem}
\begin{proof}
The following formula is the case $t=1$ of \eqref{LOQ}:
\begin{equation}\label{LOW}
\phi_T(n,a_0,a_*)=\phi_\mfl(a)+\sum_{i\in\cI}\a_i\int_0^1\phi_{\mfl_i}(a(s))\phi_{\hat\mfl_i}(a(s))ds
\end{equation}
where the left-hand side is defined by \eqref{PHITD}.
We see from \eqref{PHITD} that $\phi_T(n,.,.)$ has an analytic extension to $\Delta_{\mfs,\C}(T)$.
Also, the real linear maps
$$
a\mapsto\a:\Delta_\mfl(T)\to\R^\cI,\q a\mapsto(a_0,a_*):\Delta_\mfl(T)\to\Delta_\mfs(T)
$$
extend to complex linear maps $\Delta_{\mfl,\C}(T)\to\C^\cI$ and $\Delta_{\mfl,\C}(T)\to\Delta_{\mfs,\C}(T)$.
We can therefore use \eqref{LOW} recursively to construct the desired analytic extension of $\phi_\mfl$.
\end{proof}

\begin{proof}[Proof of Proposition \ref{formula planar loops}]
Since $Q_{\mfl,k}(a,z)$ is 
continuous in $z=(z_1,\dots,z_{n+1})$ on $\cO_\cS$,
analytic in $a$, 
and uniformly bounded on compacts in $\Delta_{\mfl,\C}(T)$,
the right-hand side of \eqref{CFM} is a well-defined multiple contour integral,
does not depend on the order of integration, 
does not depend on the choice of admissible family $(\g_1,\dots,\g_{n+1})$,
and defines an analytic function $\psi_\mfl$ on $\Delta_{\mfl,\C}(T)$.
Set
$$
\d_\mfl(a)=\phi_\mfl(a)-\psi_\mfl(a).
$$
Then $\d_\mfl$ is also analytic on $\Delta_{\mfl,\C}(T)$ by Proposition \ref{ANA}.
We will show by induction on $n$ that $\d_\mfl(a)=0$.

For $n=0$, this follows from Proposition \ref{NFL}.   
Suppose inductively that the statement holds for $n-1$ and let $\mfl$ be a splittable combinatorial planar loop with $n$ intersections. 
Fix $i\in\cI_\mfl$, to be chosen later, and write $k_l$ and $k_r$ for the labels in $\mfl_{i,l}$ and $\mfl_{i,r}$ of the faces containing the face $k$ in $\mfl$.
For $a\in\Delta_\mfl(T)$, write $a_l$ and $a_r$ for the images of $a$ under the natural submersions
$$
\Delta_{\mfl,\C}(T)\to\Delta_{\mfl_{i,l},\C}(T),\q\Delta_{\mfl,\C}(T)\to\Delta_{\mfl_{i,r},\C}(T).
$$
For $(z_1,\dots,z_{n+1})\in\C^{n+1}$, set
$$
z_l=(z_j:\mfs_j\in\cS_{\mfl_{i,l}}),\q z_r=(z_j:\mfs_j\in\cS_{\mfl_{i,r}}).
$$
Then, for $a\in\Delta_{\mfl,\C}(T)$ and $\mfs\in\cS$, we have 
$$
\mu_i\<n_\mfs,a\>=
\begin{cases}
1,&\text{if $\mfs$ uses the right outgoing edge at $i$},\\
-1,&\text{if $\mfs$ uses the left outgoing edge at $i$},\\
0,&\text{otherwise}.
\end{cases}
$$
Hence
\begin{equation}
\label{Wilson Pol}
\mu_iQ_{\mfl,k}(a,z)=Q_{\mfl_{i,l},k_l}(a_l,z_l)Q_{\mfl_{i,r},k_r}(a_r,z_r).
\end{equation}

Since $n\ge1$, the tree $\cS_\mfl$ has at least two leaves, and one of them, say $\mfs_m$, 
is not the boundary of the distinguished face $k$. 
Since the labelling is adapted, there exists $i\le m-1$ such that $\mfs_i$ is adjacent to $\mfs_m$.
Denote by $k_c$ the component of its complement which does not include $k_\infty$.  

The sequence $(\mfs_1,\dots,\mfs_{m-1},\mfs_{m+1},\dots,\mfs_n)$ is an adapted labelling of $\mfl_{i,l}$
and the family of loops $(\g_1,\dots,\g_{m-1},\g_{m+1},\dots,\g_n)$ is admissible for this sequence and for the 
distinguished face $k_l$.
Also, $\mfs_m$ is an adapted labelling of $\mfl_{i,r}$ with admissible loop $\g_m$.
Since the right-hand-side of \eqref{Wilson Pol} is uniformly bounded on any compact subset of $\Delta_{\mfl,\C}(T)\times\cO_n$,  
we deduce that, for all $a\in\Delta_{\mfl,\C}(T)$, 
$$
\mu_i\psi_\mfl(a)=\psi_{\mfl_{i,l}}(a_l)\psi_{\mfl_{i,r}}(a_r).
$$
On the other hand, since $\Phi_T$ satisfies the Makeenko--Migdal equations, for all $a\in\Delta_\mfl(T)$,
$$
\mu_i\phi_\mfl(a)=\phi_{\mfl_{i,l}}(a_l)\phi_{\mfl_{i,r}}(a_r)
$$
and this extends to $a\in\Delta_{\mfl,\C}(T)$ by analyticity.
But $\mfl_{i,l}$ and $\mfl_{i,r}$ are splittable and have no more than $n-1$ points of intersection.
So we have shown that, for all $a\in\Delta_{\mfl,\C}(T)$,
\begin{equation}\label{analytic MM}
\mu_i\d_\mfl(a)=0.
\end{equation}

We check now the boundary condition of this equation. 
Since $\mfl$ is splittable, there is a splittable loop $\tilde\mfl$, with exactly $n-1$ intersections, 
an affine map 
$$
\iota_c:\Delta_\mfl(T)\cap\{a: a_{k_c}=0\}\to\Delta_{\tilde\mfl}(T)
$$ 
and a distinguished face $\tilde k\in\cF_{\tilde\mfl}$ such that,
for any $a\in\Delta_\mfl(T)$ with $a_{k_c}=0$, 
$$
\mfl(a)\cap\tilde\mfl(\iota_c(a))\not=\emptyset
$$
and $\iota_c(a)_{\tilde k}=0$ if and only if $a_k=0$. 
Moreover, for all $a\in\Delta_\mfl(T)$, 
\begin{equation}\label{continuity for analytic extension}
\phi_\mfl(a)=\phi_{\tilde\mfl}(\iota_c(a)).
\end{equation}
Furthermore, by analyticity of $\phi_\mfl$ and $\phi_{\tilde\mfl}$, this equality holds true for all $a\in\Delta_{\mfl,\C}(T)$ with $a_{k_c}=0$.
Let $\nu\in\Z^{\cF_\mfl}$ be the vector with $\nu_{k_c}=1$ which is proportional to $\mu_i$, 
viewed as an element of $\{1_{\cF_\mfl}\}^\bot\cap\R^{\cF_\mfl}$, where  $i$ is the only vertex adjacent to $F_c$.  
Then, by \eqref{analytic MM}, for all $a\in\Delta_{\mfl,\C}(T)$,
$$
\psi_\mfl(a)=\psi_\mfl(a-a_{k_c}\nu).
$$
As $a-a_{k_c}\nu\in\Delta_{\mfl,\C}(T)\cap\{a:a_{k_c}=0\}$, by \eqref{continuity for analytic extension}, in order to conclude, 
it is sufficient to show that, for all $a\in\Delta_{\mfl,\C}(T)$ with $a_{k_c}=0$,
$$
\phi_\mfl(a)=\phi_{\tilde\mfl}(\iota_c(a)).
$$
For such a vector $a$ and for $z\in\cO_n$, set $\tilde z=(z_j:j\not=o)$.
Then
$$
Q_{\mfl,k}(a,z)=Q_{\tilde\mfl,\tilde k}(a,\tilde z)\frac{\ve_{\mfs_o}e^{\ve_{\mfs_o}G_T(z_o)}}{z_o-z_i}.   
$$
For $a\in\Delta_{\mfl,\C}(T)$, the only singularity of $z_o\in\C\sm[-\b,\b]\mapsto Q_{\mfl,k}(a,z)$ is at $z_i$. 
Since the family of loops $(\g_1,\dots,\g_{n+1})$ is admissible, by deforming $\g_o$, 
we can assume that  the bounded component of $\C\sm\g_{\mfs_o}$ contains all $\g_j$ with $j\not=o$.     
Then, for all $\tilde z\in\cO_{n-1}$,
\begin{align*}
\frac1{2\pi i}\int_{\g_{\mfs_o}}Q_{\mfl,k}(a,z)dz_{\mfs_o}
=Q_{\tilde\mfl,\tilde k}(a,\tilde z) 
\frac1{2\pi i}\int_C\frac{e^{\ve_oG_T(z_o)}}{z_o-z_i}dz_o
\end{align*}
with $C$ an anticlockwise circle with centre $0$, whose interior contains all contours $(\g_j:j\not=o)$. 
Since $G_T(z)\sim1/z$ as $z\to\infty$, it follows that 
$$
\frac1{2\pi i}\int_C\frac{e^{\ve_oG_T(z_o)}}{z_o-z_i}dz_o
=-\frac1{2\pi i}\int_{1/C}\frac{e^{\ve_oG_T(1/y)}}{y(1-yz_i)}dy=1.  
$$
Therefore, performing the integration in $\phi_\mfl(a)$ first with respect to $z_m$, we obtain, 
when $a\in\Delta_{\mfl,\C}(T)$, with $a_{k_c}=0$,
\begin{align*}
\phi_\mfl(a)&=\left(\frac1{2\pi i}\right)^n\int_{z_j\in\g_j,\text{ for }j\not=o}
Q_{\tilde\mfl,\tilde k}(a,(z_\mfs)_{\mfs\in\cS_{\mfl}\sm\{\mfs_o\}})\prod_{j\not=o}dz_j\\
&=\phi_{\tilde\mfl}(\iota_c(a)).
\end{align*}
\end{proof}

\noindent\textbf{Acknowledgments.} 
The authors wish to thank Guillaume C\'ebron, Franck Gabriel, Thierry L\'evy, and Myl\`ene Ma\"ida for several motivating and fruitful discussions about this project.  They are grateful to Franck Gabriel for his comments  on a first version of this work.

%\bibliography{mfs}

\begin{thebibliography}{35}




\bibitem{MR2864481}
Michael Anshelevich and Ambar~N. Sengupta.
\newblock Quantum free {Y}ang-{M}ills on the plane.
\newblock {\em J. Geom. Phys.}, 62(2):330--343, 2012.

\bibitem{MR0305319}
Thomas~F. Banchoff and William~F. Pohl.
\newblock A generalization of the isoperimetric inequality.
\newblock {\em J. Differential Geometry}, 6:175--192, 1971/72.

\bibitem{MR1426833}
Philippe Biane.
\newblock Free {B}rownian motion, free stochastic calculus and random matrices.
\newblock In {\em Free probability theory ({W}aterloo, {ON}, 1995)}, volume~12
  of {\em Fields Inst. Commun.}, pages 1--19. Amer. Math. Soc., Providence, RI,
  1997.

\bibitem{MR1262293}
D.~V. Boulatov.
\newblock Wilson loop on a sphere.
\newblock {\em Modern Phys. Lett. A}, 9(4):365--374, 1994.

\bibitem{MR1692402}
A.~Boutet~de Monvel and M.~V. Shcherbina.
\newblock On free energy in two-dimensional {${\rm U}(n)$}-gauge field theory
  on the sphere.
\newblock {\em Teoret. Mat. Fiz.}, 115(3):389--401, 1998.

\bibitem{1601.00214}
G.~{C{\'e}bron}, A.~{Dahlqvist}, and F.~{Gabriel}.
\newblock The generalized master fields.
\newblock {\tt arxiv 1601.00214}, 2016.

\bibitem{1502.06186}
B.~Collins, A.~Dahlqvist, and T.~Kemp.
\newblock Strong convergence of unitary {B}rownian motion.
\newblock {\tt arxiv 1502.06186}, 2015.

\bibitem{MR3554890}
Antoine Dahlqvist.
\newblock Free {E}nergies and {F}luctuations for the {U}nitary {B}rownian
  {M}otion.
\newblock {\em Comm. Math. Phys.}, 348(2):395--444, 2016.

\bibitem{MR1297298}
Jean-Marc Daul and Vladimir~A. Kazakov.
\newblock Wilson loop for large {$N$} {Y}ang-{M}ills theory on a
  two-dimensional sphere.
\newblock {\em Phys. Lett. B}, 335(3-4):371--376, 1994.

\bibitem{DOUGLAS1993219}
Michael~R. Douglas and Vladimir~A. Kazakov.
\newblock Large n phase transition in continuum qcd2.
\newblock {\em Physics Letters B}, 319(1):219 -- 230, 1993.

\bibitem{MR1006295}
Bruce~K. Driver.
\newblock Y{M{${}\sb 2$}}: continuum expectations, lattice convergence, and
  lassos.
\newblock {\em Comm. Math. Phys.}, 123(4):575--616, 1989.

\bibitem{MR3631396}
Bruce~K. Driver, Franck Gabriel, Brian~C. Hall, and Todd Kemp.
\newblock The {M}akeenko-{M}igdal equation for {Y}ang-{M}ills theory on compact
  surfaces.
\newblock {\em Comm. Math. Phys.}, 352(3):967--978, 2017.

\bibitem{MR3613519}
Bruce~K. Driver, Brian~C. Hall, and Todd Kemp.
\newblock Three proofs of the {M}akeenko-{M}igdal equation for {Y}ang-{M}ills
  theory on the plane.
\newblock {\em Comm. Math. Phys.}, 351(2):741--774, 2017.

\bibitem{MR2483725}
D.~F{\'e}ral.
\newblock On large deviations for the spectral measure of discrete {C}oulomb
  gas.
\newblock In {\em S\'eminaire de probabilit\'es {XLI}}, volume 1934 of {\em
  Lecture Notes in Math.}, pages 19--49. Springer, Berlin, 2008.

\bibitem{MR1124272}
Dana~S. Fine.
\newblock Quantum {Y}ang-{M}ills on a {R}iemann surface.
\newblock {\em Comm. Math. Phys.}, 140(2):321--338, 1991.

\bibitem{MR2747559}
Peter~J. Forrester, Satya~N. Majumdar, and Gr\'egory Schehr.
\newblock Non-intersecting {B}rownian walkers and {Y}ang-{M}ills theory on the
  sphere.
\newblock {\em Nuclear Phys. B}, 844(3):500--526, 2011.

\bibitem{MR2874214}
Peter~J. Forrester, Satya~N. Majumdar, and Gr\'egory Schehr.
\newblock Erratum to ``{N}on-intersecting {B}rownian walkers and {Y}ang-{M}ills
  theory on the sphere''.
\newblock {\em Nuclear Phys. B}, 857(3):424--427, 2012.

\bibitem{MR1106850}
F.~D. Gakhov.
\newblock {\em Boundary value problems}.
\newblock Dover Publications, Inc., New York, 1990.
\newblock Translated from the Russian, Reprint of the 1966 translation.

\bibitem{MR1401299}
Rajesh Gopakumar.
\newblock The master field in generalised {${\rm QCD}\sb 2$}.
\newblock {\em Nuclear Phys. B}, 471(1-2):246--260, 1996.

\bibitem{MR1352420}
Rajesh Gopakumar and David~J. Gross.
\newblock Mastering the master field.
\newblock {\em Nuclear Phys. B}, 451(1-2):379--415, 1995.

\bibitem{MR1321333}
David~J. Gross and Andrei Matytsin.
\newblock Some properties of large-{$N$} two-dimensional {Y}ang-{M}ills theory.
\newblock {\em Nuclear Phys. B}, 437(3):541--584, 1995.

\bibitem{MR1015789}
Leonard Gross, Christopher King, and Ambar Sengupta.
\newblock Two-dimensional {Y}ang-{M}ills theory via stochastic differential
  equations.
\newblock {\em Ann. Physics}, 194(1):65--112, 1989.

\bibitem{MR2198201}
Alice Guionnet and Myl{\`e}ne Ma{\"{\i}}da.
\newblock Character expansion method for the first order asymptotics of a
  matrix integral.
\newblock {\em Probab. Theory Related Fields}, 132(4):539--578, 2005.

\bibitem{1705.07808}
B.C. Hall.
\newblock The large-{N} limit for two-dimensional {Y}ang--{M}ills theory.
\newblock {\tt arxiv 1705.07808}, 2017.

\bibitem{MR1737991}
Kurt Johansson.
\newblock Shape fluctuations and random matrices.
\newblock {\em Comm. Math. Phys.}, 209(2):437--476, 2000.

\bibitem{MR859186}
Richard~V. Kadison and John~R. Ringrose.
\newblock {\em Fundamentals of the theory of operator algebras. {V}ol. {II}},
  volume 100 of {\em Pure and Applied Mathematics}.
\newblock Academic Press, Inc., Orlando, FL, 1986.
\newblock Advanced theory.

\bibitem{MR605753}
V.~A. Kazakov.
\newblock Wilson loop average for an arbitrary contour in two-dimensional
  {${\rm U}(N)$} gauge theory.
\newblock {\em Nuclear Phys. B}, 179(2):283--292, 1981.

\bibitem{MR596907}
V.~A. Kazakov and I.~K. Kostov.
\newblock Nonlinear strings in two-dimensional {${\rm U}(\infty )$} gauge
  theory.
\newblock {\em Nuclear Phys. B}, 176(1):199--215, 1980.

\bibitem{MR1920389}
Anthony~W. Knapp.
\newblock {\em Lie groups beyond an introduction}, volume 140 of {\em Progress
  in Mathematics}.
\newblock Birkh\"auser Boston, Inc., Boston, MA, second edition, 2002.

\bibitem{MR2006374}
Thierry L{\'e}vy.
\newblock Yang-{M}ills measure on compact surfaces.
\newblock {\em Mem. Amer. Math. Soc.}, 166(790):xiv+122, 2003.

\bibitem{MR2407946}
Thierry L{\'e}vy.
\newblock Schur-{W}eyl duality and the heat kernel measure on the unitary
  group.
\newblock {\em Adv. Math.}, 218(2):537--575, 2008.

\bibitem{MR2667871}
Thierry L{\'e}vy.
\newblock Two-dimensional {M}arkovian holonomy fields.
\newblock {\em Ast\'erisque}, 329:172, 2010.

\bibitem{MR3636410}
Thierry L\'evy.
\newblock The master field on the plane.
\newblock {\em Ast\'erisque}, 388:IX+201, 2017.

\bibitem{MR2727643}
Thierry L{\'e}vy and Myl{\`e}ne Ma{\"{\i}}da.
\newblock Central limit theorem for the heat kernel measure on the unitary
  group.
\newblock {\em J. Funct. Anal.}, 259(12):3163--3204, 2010.

\bibitem{MR3440793}
Thierry L{\'e}vy and Myl{\`e}ne Ma{\"{\i}}da.
\newblock On the {D}ouglas-{K}azakov phase transition. {W}eighted potential
  theory under constraint for probabilists.
\newblock In {\em Mod\'elisation {A}l\'eatoire et {S}tatistique---{J}ourn\'ees
  {MAS} 2014}, volume~51 of {\em ESAIM Proc. Surveys}, pages 89--121. EDP Sci.,
  Les Ulis, 2015.

\bibitem{MR3474469}
Karl Liechty and Dong Wang.
\newblock Nonintersecting {B}rownian motions on the unit circle.
\newblock {\em Ann. Probab.}, 44(2):1134--1211, 2016.

\bibitem{MakeenkoMigdal}
Y.~M. Makeenko and Migdal~A. A.
\newblock Exact equation for the loop average in multicolor {QCD}.
\newblock {\em Phys. Lett. B}, 88:135--137, 1979.

\bibitem{MR2266879}
Alexandru Nica and Roland Speicher.
\newblock {\em Lectures on the combinatorics of free probability}, volume 335
  of {\em London Mathematical Society Lecture Note Series}.
\newblock Cambridge University Press, Cambridge, 2006.

\bibitem{Rusakov}
B.~Rusakov.
\newblock Wilson loops in large {N} {QCD} on a sphere.
\newblock {\em Phys. Lett. B}, 329:338--344, 1994.

\bibitem{MR1346931}
Ambar Sengupta.
\newblock Gauge theory on compact surfaces.
\newblock {\em Mem. Amer. Math. Soc.}, 126(600):viii+85, 1997.

\bibitem{MR2757706}
Ambar~N. Sengupta.
\newblock The large-{$N$} {Y}ang-{M}ills field on the plane and free noise.
\newblock In {\em Geometric methods in physics}, volume 1079 of {\em AIP Conf.
  Proc.}, pages 121--132. Amer. Inst. Phys., Melville, NY, 2008.

\bibitem{MR1373007}
I.~M. Singer.
\newblock On the master field in two dimensions.
\newblock In {\em Functional analysis on the eve of the 21st century, {V}ol.\ 1
  ({N}ew {B}runswick, {NJ}, 1993)}, volume 131 of {\em Progr. Math.}, pages
  263--281. Birkh\"auser Boston, Boston, MA, 1995.

\bibitem{MR1873025}
M.~Takesaki.
\newblock {\em Theory of operator algebras. {I}}, volume 124 of {\em
  Encyclopaedia of Mathematical Sciences}.
\newblock Springer-Verlag, Berlin, 2002.
\newblock Reprint of the first (1979) edition, Operator Algebras and
  Non-commutative Geometry, 5.

\bibitem{MR1133264}
Edward Witten.
\newblock On quantum gauge theories in two dimensions.
\newblock {\em Comm. Math. Phys.}, 141(1):153--209, 1991.

\bibitem{MR1185834}
Edward Witten.
\newblock Two-dimensional gauge theories revisited.
\newblock {\em J. Geom. Phys.}, 9(4):303--368, 1992.

\bibitem{MR1489573}
Feng Xu.
\newblock A random matrix model from two-dimensional {Y}ang-{M}ills theory.
\newblock {\em Comm. Math. Phys.}, 190(2):287--307, 1997.



\end{thebibliography}
\def\cprime{$'$}

\end{document}